\documentclass[11pt]{article}

\usepackage{amssymb,amsmath,amsfonts,amsthm}

\usepackage{caption2}

\renewcommand{\subsection}{\subsubsection}

\renewcommand\appendix{\par
  \setcounter{section}{0}
   \renewcommand\thesection{Appendix \Alph{section}.}
 }

\newtheorem{theorem}{Theorem}[section]

\newtheorem{lemma}{Lemma}[section]
\newtheorem{proposition}{Proposition}[section]
\newtheorem{definition}{Definition}[section]

\newtheorem{remark}{Remark}[section]

\def\nt{|\hspace{-0.7pt}|\hspace{-0.7pt}|}

\begin{document}

\title{\bf Well-posedness of the free boundary problem in compressible elastodynamics}
\author{{\bf Yuri Trakhinin}\\
Sobolev Institute of Mathematics, Koptyug av. 4, 630090 Novosibirsk, Russia\\
and\\
Novosibirsk State University, Pirogova str. 2, 630090 Novosibirsk, Russia\\
E-mail: trakhin@math.nsc.ru
}

\date{
}
%
%
\maketitle
\begin{abstract}
We study the free boundary problem for the flow of a compressible isentropic inviscid elastic fluid. At the free boundary moving with the velocity of the fluid particles the columns of the deformation gradient are tangent to the boundary and the pressure vanishes outside the flow domain. We prove the local-in-time existence of a unique smooth solution of the free boundary problem provided that among three columns of the deformation gradient there are two which are non-collinear vectors at each point of the initial free boundary. If this non-collinearity condition fails, the local-in-time existence is proved under the classical Rayleigh-Taylor sign condition satisfied at the first moment. By constructing an Hadamard-type ill-posedness example for the frozen coefficients linearized problem we show that the simultaneous failure of the non-collinearity condition and the Rayleigh-Taylor sign condition leads to Rayleigh-Taylor instability.
\end{abstract}

\section{Introduction}
\label{s1}

We consider the equations of elastodynamics governing the motion of a compressible isentropic inviscid elastic fluid \cite{Daf,Gurt,Jos}:
\begin{equation}\label{7}
\left\{
\begin{array}{l}
 \partial_t\rho  +{\rm div}\, (\rho v )=0, \\[2pt]
 \partial_t(\rho v ) +{\rm div}\,(\rho v\otimes v  ) + {\nabla}p={\rm div}\,(\rho FF^{\top}),\\[6pt]
{\displaystyle \frac{{\rm d} F}{{\rm d}t}}=\nabla v F,
\end{array}
\right.
\end{equation}
where $\rho$ denotes the density, $v\in\mathbb{R}^3$  the velocity, $F\in \mathbb{M}(3,3)$ the deformation gradient, ${\rm d} /{\rm d} t =\partial_t+({v} \cdot{\nabla} )$ the material derivative, and the pressure $p=p(\rho)$ is a smooth function of $\rho$. Moreover, system \eqref{7} is supplemented by the identity ${\rm div}\,(\rho F^{\top})=0$ which is the set of the three divergence constraints
\begin{equation}\label{8}
{\rm div}\,(\rho F_j)=0
\end{equation}
($j=1,2,3$) on initial data for the Cauchy problem, where $F_j=(F_{1j},F_{2j},F_{3j})$ is the vector field corresponding to the $j$th column of the deformation gradient, i.e., one can show that if the initial data for \eqref{7} satisfy \eqref{8}, then the divergence constraints \eqref{8} hold for all $t > 0$. The first-order system \eqref{7} written in the Eulerian coordinates describes the motion of elastic waves in a compressible material for which the Cauchy stress tensor has the form $\rho FF^{\top}$ corresponding to the elastic energy $W(F)=\frac{1}{2}|F|^2$ for the Hookean linear elasticity. At last, we note that system \eqref{7} arises as the inviscid limit of the equations of compressible viscoelasticity \cite{Daf,Gurt,Jos} of Oldroyd type \cite{Old1,Old2} (see, e.g., \cite{HuWang1,HuWang2,HuWang3,HuWang4,Qian,QianZhang,Renardy} and references therein for various aspects of analysis of these equations).

Taking into account the divergence constraints \eqref{8}, we easily symmetrize system \eqref{7} by rewriting it as
\begin{equation}
\left\{
\begin{array}{l}
{\displaystyle\frac{1}{\rho c^2}\,\frac{{\rm d} p}{{\rm d}t} +{\rm div}\,{v} =0,}\\[6pt]
{\displaystyle\rho\, \frac{{\rm d}v}{{\rm d}t}+{\nabla}p -\rho\sum_{j=1}^{3}(F_j\cdot\nabla )F_j =0 ,}\\[6pt]
\rho\,{\displaystyle \frac{{\rm d} F_j}{{\rm d}t}-\rho \,(F_j\cdot\nabla )v =0,}
\end{array}\right. \label{9}
\end{equation}
where $c^2=p'(\rho)$ is the square of the sound speed. Equations \eqref{9} form the symmetric system
\begin{equation}
\label{10}
A_0(U )\partial_tU+\sum_{k=1}^{3}A_k(U)\partial_kU=0
\end{equation}
for $U=(p,v,F_1,F_2,F_3)$, with $A_0= {\rm diag} (1/(\rho c^2) ,\rho I_{12})$ and
\[
A_k=\begin{pmatrix}
{\displaystyle\frac{v_k}{\rho c^2}} & e_k & \underline{0} & \underline{0} & \underline{0} \\[7pt]
e_k^{\top}&\rho v_kI_3 & -\rho F_{k1}I_3 & -\rho F_{k2}I_3 & -\rho F_{k3}I_3 \\[3pt]
\underline{0}^{\top} &-\rho F_{k1}I_3 & \rho v_kI_3 & O_3 & O_3 \\
\underline{0}^{\top} &-\rho F_{k2}I_3 & O_3 & \rho v_kI_3 & O_3 \\
\underline{0}^{\top} &-\rho F_{k3}I_3 & O_3 & O_3 & \rho v_kI_3
\end{pmatrix},
\]
where $e_k=(\delta_{1k},\delta_{2k},\delta_{3k})$ and $\underline{0}=(0,0,0)$. Here and below $I_m$ and $O_m$ denote the unit and zero matrices of order $m$ respectively. In \eqref{10} we think of the density as a function of the pressure: $\rho =\rho (p)$, $c^2=1/\rho'(p)$. System \eqref{10} is symmetric hyperbolic if   $A_0>0$, i.e.,
\begin{equation}
\rho >0,\quad  \rho '(p)>0. \label{11}
\end{equation}

We now consider system \eqref{7} (or \eqref{10}) in a space-time domain $\Omega (t)$ whose boundary $\Gamma (t) =\{{\eta}(t,x)=0\}$ is to be determined and moves with the velocity of the material particles at the boundary, while the pressure vanishes outside the flow domain and the columns of the deformation gradient are tangent to the free boundary:
\begin{equation}
\frac{{\rm d}{\eta} }{{\rm d} t}=0,\quad p=0, \quad F_j\cdot N=0,\quad \mbox{on}\ \Gamma (t)\label{4}
\end{equation}
(for all $t\in [0,T]$), where $N=\nabla {\eta}$. Note that the conditions $(F_j\cdot N)|_{\Gamma}=0$ coming from the constraint equations \eqref{8} are not real boundary conditions and must be regarded as restrictions (boundary constraints) on the initial data. This fact together with the fact that identities \eqref{8} are preserved in $\Omega (t)$ was proved in \cite{HaoWang} (for incompressible elastic fluids with $\rho \equiv {\rm const}$) by passing to the Lagrangian coordinates and we will later present the proof in the Eulerian coordinates (see Proposition \ref{p1}).

As for the free boundary problem for the compressible Euler equations in \cite{Lind,Tcpam}, we will assume that the hyperbolicity conditions \eqref{11} are satisfied up to the boundary $\Gamma$. Since $p|_{\Gamma}=0$, this excludes the consideration of polytropic processes, i.e., the $\gamma$-law equation of state $p=A\rho^{\gamma}$ ($A>0,\ \gamma >1$). Moreover, it follows from \eqref{11} and $p|_{\Gamma}=0$ that
\begin{equation}
\rho|_{\Gamma}=\rho (p|_{\Gamma})=\rho (0)=\bar{\rho}_0>0,\label{rho1}
\end{equation}
i.e., as in \cite{Lind}, we assume that
\begin{equation}
p(\bar{\rho}_0)=0\quad \mbox{and}\quad p'(\rho) >0,\quad \mbox{for}\ \rho \geq \bar{\rho}_0,\label{rho2}
\end{equation}
where $\bar{\rho}_0$ is a non-negative constant.

Without the deformation gradient $F$, problem \eqref{7}, \eqref{4} becomes the free boundary problem for the compressible isentropic Euler equations. Under assumptions \eqref{rho1} and \eqref{rho2} corresponding to the case of compressible liquid, by using the Lagrangian framework the local-in-time existence of smooth solutions of this problem was proved by Lindblad \cite{Lind}, provided that the Rayleigh-Taylor sign condition
\begin{equation}
\frac{\partial p}{\partial N}\leq -\epsilon <0\quad \mbox{on}\ \Gamma (0)\label{5}
\end{equation}
holds and the initial domain $\Omega (0)$ is diffeomorfic to a ball, where $\partial /\partial N=(N\cdot\nabla )$ and $\epsilon$ is a fixed constant. For an unbounded initial domain, this result was recovered by Trakhinin \cite{Tcpam} in Eulerian coordinates and extended to non-isentropic flow. For the much more complicated case $\rho|_{\Gamma}=0$ when the hyperbolic system of compressible Euler equations  degenerates on the boundary, we refer the reader to the result of Coutand and Shkoller in \cite{Cout-Shkol} where the local-in-time existence of unique smooth solutions was proved for a $\gamma$-law gas flow satisfying the so-called physical vacuum condition (see also references in \cite{Cout-Shkol} for a series of preceding results in this direction). Regarding the incompressible Euler equations, the result analogous to that in \cite{Lind} was obtained by Lindblad \cite{Lind_incomp} and the case of an unbounded initial domain was considered by P. Zhang and Z. Zhang \cite{PZZhang}. We also refer to references in \cite{Lind_incomp,PZZhang} for a huge literature about the case of irrotational flow ($\nabla\times v =0$) known as the water wave problem.

Returning to our free boundary problem \eqref{7}, \eqref{4}, we note that for the case of incompressible elastic fluid with constant density it was studied by Hao and Wang in \cite{HaoWang} where a priori estimates in Sobolev norms of solutions were derived through a geometrical point of view of Christodoulou and Lindblad \cite{ChristLind} under the fulfilment of the Rayleigh-Taylor sign condition \eqref{5}. We also note that a more complicated free boundary problem for system \eqref{7} was recently studied by Chen, Hu and Wang \cite{ChanHuWang}. Namely, they studied the free boundary problem for compressible vortex sheets for the two-dimensional version of equations \eqref{7} on the linear level of constant coefficients. Necessary and sufficient conditions for the linear stability of the rectilinear vortex sheets were found by spectral analysis and a priori $L^2$ estimates were obtained in \cite{ChanHuWang} by the Kreiss symmetrization technique \cite{Kreiss}.

It is worth noting that the well-posedness of the incompressible counterpart of problem \eqref{7}, \eqref{4} was not established in \cite{HaoWang}. In this connection, our final goal is to prove the local-in-time existence of unique smooth solutions of problem \eqref{7}, \eqref{4} under appropriate ``stability'' conditions for the initial data. The results obtained in the present paper for problem \eqref{7}, \eqref{4} stay valid for its incompressible counterpart (with some natural modifications connected with the ``ellipticity'' of the unknown $p$  for the incompressible case). We prefer, however, to restrict ourselves to the compressible case. Moreover, it seems that the free boundary problem for incompressible elastic fluids could be more effectively treated by the approach based on the idea of Wu \cite{Wu1,Wu2} of getting an evolution problem of the free surface and applied recently by Sun, Wang and Zhang for current-vortex sheets \cite{SunWangZhang1} and the plasma-vacuum problem \cite{SunWangZhang2} in ideal incompressible magnetohydrodynamics (see also \cite{PZZhang} mentioned above). At the same time, in the present paper our result connected with Rayleigh-Taylor instability detected as ill-posedness for frozen coefficients is also obtained for the incompressible case.

Stabilization effects of the elasticity were established in \cite{ChanHuWang} for vortex sheets. For problem \eqref{7}, \eqref{4}, we also show that the elasticity plays a stabilization role. Namely, we manage to prove the local-in-time existence of a unique smooth solution of our free boundary problem provided that among the three vectors $F_1$, $F_2$ and $F_3$ there are two which are non-collinear at each point of the initial free boundary, i.e.,
\begin{equation}
\label{6}
\exists\ \mu,\nu\in\{1,2,3\},\ \mu\neq \nu\ :\quad
|F_{\mu}\times F_\nu|\geq \delta >0  \quad  \mbox{on}\ \Gamma (0),
\end{equation}
where $\delta$ is a fixed constant. That is, we show that the Rayleigh-Taylor sign condition \eqref{5} is not necessary for well-posedness. However, if the {\it non-collinearity condition} \eqref{6} fails, we prove well-posedness under the classical condition \eqref{5}. Moreover, by constructing an Hadamard-type ill-posedness example for the frozen coefficients linearized problem we show that the simultaneous  failure of the non-collinearity condition and the Rayleigh-Taylor sign condition leads to Rayleigh-Taylor instability.

It should be noted that, as for the case without the deformation gradient $F$ in \cite{Lind,Tcpam}, the linear constant coefficients problem associated with problem \eqref{7}, \eqref{4} always satisfies the Kreiss-Lopatinski condition but violates the uniform Kreiss--Lopatinski condition \cite{Kreiss}. That is, the linearized problem can be well-posed only in a weak sense. This yields {\it losses of derivatives} in a priori estimates for the linearized problem. Since in the a priori estimates obtained in this paper for the linearized problem we have a fixed loss of derivatives from the source terms and with respect to the coefficients, we prove the existence of solutions to the original nonlinear problem by a suitable Nash-Moser-type iteration scheme. This scheme is similar to that in \cite{CS2,MTTcont,Tcpam} and the proof of its convergence is based on the usage of suitable tame a priori estimates in Sobolev spaces deduced for the linearized problem. The uniqueness of the solution to the nonlinear problem follows from a basic a priori estimate for the linearized problem and is proved by standard argument.

The non-collinearity condition \eqref{6} appears in our analysis as the requirement that the symbol associated to the free surface is elliptic. This means that the boundary conditions on the free boundary $F(t,x)=0$, which can be locally considered as the graph $x_1=\varphi (t,x_2,x_3)$, are resolvable for the space-time gradient $(\partial_t\varphi ,\partial_2\varphi ,\partial_3\varphi )$. The analogous non-collinearity condition for the magnetic field was introduced in \cite{Tjde} for the plasma-vacuum interface problem in ideal compressible magnetohydrodynamics whose well-posedness was proved in \cite{ST} under this condition satisfied for the initial data. Also the same non-collinearity condition for the magnetic field appears for compressible current-vortex sheets \cite{T05,T09} and ensures that the front symbol for them is elliptic.

If the front symbol is not elliptic, problem \eqref{7}, \eqref{4} is not a quite standard ``weakly stable'' hyperbolic free boundary problem. Actually, regardless of the fact that the constant coefficients problem for \eqref{7}, \eqref{4} always satisfies the weak Kreiss--Lopatinski condition, the corresponding variable coefficients problem is not unconditionally well-posed, and \eqref{5} is an {\it extra} condition which is necessary for well-posedness if \eqref{6} fails. As was already mentioned above, we prove ill-posedness under the simultaneous  failure of the non-collinearity condition and the Rayleigh-Taylor sign condition for frozen coefficients.

Although the free surface $\Gamma (t)$ is a characteristic for the symmetric hyperbolic system \eqref{10} that implies a natural loss of control on derivatives in the normal direction we manage to compensate this loss.
This is achieved for the linearized problem by estimating missing normal derivatives through equations satisfied by the linearized divergences associated with \eqref{8} and a symmetric hyperbolic system for the derivatives of the perturbations associated with  $\nabla\times {F}_j$ and the vorticity $\nabla\times v$. This is why, unlike, for example, the problems with characteristic boundary in \cite{ST,T05,T09,Tjde}, we are able to prove our a priori estimates in usual Sobolev norms of solutions.

The rest of the paper is organized as follows. In Section \ref{s2}, we reduce the free boundary problem \eqref{7}, \eqref{4} to an initial-boundary value problem in a fixed domain and discuss properties of the reduced problem. In Section \ref{s3}, we obtain the linearized problem and a corresponding frozen coefficients problem. In Section \ref{s04}, we formulate our main results which are Theorems \ref{t01} and \ref{t02} about the well-posedness in Sobolev spaces of the reduced nonlinear problem in a fixed domain under either the Rayleigh-Taylor sign condition or the non-collinearity condition satisfied by the initial data and Theorem \ref{t3} about ill-posedness under the simultaneous failure of these conditions for frozen coefficients. In Section \ref{s4} we prove the well-posedness of the linearized problem (see Theorem \ref{t1}) and in Section \ref{s4'} we derive for it the tame estimates mentioned above (see Theorems \ref{t4.1} and \ref{t4.2}). In Section \ref{s4^}, we specify compatibility conditions for the initial data and, by constructing an approximate solution, reduce the nonlinear initial-boundary value problem to that with zero initial data.  In Section \ref{s5^}, we prove Theorems \ref{t01} and \ref{t02} by solving the reduced problem by a suitable Nash-Moser-type iteration scheme. In Section \ref{s5}, we prove Theorem \ref{t3}. At last, in Section \ref{s6} we discuss open problems.

\section{Reduced problem in a fixed domain}
\label{s2}

In \cite{Tcpam} the free boundary problem for the compressible Euler equations was considered in the unbounded flow domain $\Omega (t)=\{x_1>\varphi (t,x_2,x_3)\}$ whose boundary has the form of a graph. In this case one has to introduce gravity in the equations because otherwise the Rayleigh-Taylor sign condition \eqref{5} cannot be satisfied. On the other hand, gravity plays no role in the proof of well-posedness and can be neglected as a lower-order term as was done in \cite{Lind} for the case of a bounded domain. In this paper, we prefer to do not introduce gravity. But, to avoid using local coordinate charts necessary for a bounded domain (within the framework of our approach), and for the sake of simplicity, we will just pose periodic boundary conditions in the tangential directions. More precisely, let
\[
\Omega (t)=\{x\in\mathbb{R}^3\,|\, x_1>\varphi (t,x'),\ x'=(x_2,x_3)\in \mathbb{T}^2\}
\]
be the domain occupied by the elastic fluid at time $t\in[0,T]$, where $ \mathbb{T}^2$ denotes the $2$-torus, which can be thought as the unit square with periodic boundary conditions. Then the free boundary $\Gamma (t)$ has the form
\[
\Gamma(t) = \{ (x_1,x') \in \mathbb{R} \times \mathbb{T}^2, \ x_1=\varphi(t,x')\}, \quad t \in [0,T].
\]
With our parametrization of $\Gamma (t)$, an equivalent formulation of the boundary conditions \eqref{4} at the free boundary is
\begin{equation}
\partial_t\varphi=v_N,\quad p=0, \quad F_N^j=0,\quad \mbox{on}\ \Gamma (t),\label{12}
\end{equation}
where $v_N=v\cdot N$, $F_N^j=F_j\cdot N$, $N=(1,-\partial_2\varphi ,-\partial_3\varphi )$.

We note that if instead of the Rayleigh-Taylor sign condition we assume the fulfilment of the non-collinearity condition \eqref{6}, then we may consider the unbounded free surface $\Gamma (t)=\{ x_1=\varphi(t,x')\}$ without posing periodic boundary conditions in the tangential directions.

We now reduce the free boundary problem for system \eqref{10} to that in a fixed domain.  We straighten the free surface $\Gamma$ by using the same simplest change of independent variables as in \cite{Cont1,MTTcont,T09,Tcpam,Tjde}. That is, the unknown $U$ being smooth in $\Omega (t)$ is replaced by the vector-function
\begin{equation}
\widetilde{U}(t,x ):= U (t,\Phi (t, x),x'),
\label{13}
\end{equation}
which is smooth in the fixed domain
\[
\Omega =\{x_1>0,\ x'\in \mathbb{T}^2\}
\]
with the boundary
\[
\partial\Omega =\{x_1=0,\ x'\in \mathbb{T}^2\},
\]
where
\begin{equation}
\Phi(t,x ):= x_1+\Psi(t,x ),\quad \Psi(t,x ):= \chi (x_1)\varphi (t,x'),
\label{14}
\end{equation}
and $\chi\in C^{\infty}_0(\mathbb{R})$ equals to 1 on $[0,1]$, and $\|\chi'\|_{L_{\infty}(\mathbb{R})}<1/2$. Here we use the cut-off function $\chi$ to avoid assumptions about compact support of the initial data in our future existence theorem. This change of variables is admissible if $\partial_1\Phi\neq 0$. The latter is guaranteed, namely, the inequality $\partial_1\Phi >0$ is fulfilled if we consider solutions for which $\|\varphi \|_{L_{\infty}([0,T]\times\mathbb{R}^2)}\leq 1$. The last inequality holds if, without loss of generality, we consider the initial data satisfying $\|\varphi_0\|_{L_{\infty}(\mathbb{R}^2)}\leq 1/2$, and the time $T$ in our existence theorem is sufficiently small.

\begin{remark}{\rm
Since the domain $\Omega$ is unbounded in the normal direction, smooth solutions belonging to Sobolev spaces should vanish at infinity. In particular, $p|_{x_1\rightarrow +\infty}\rightarrow 0$. In view of \eqref{rho1} and \eqref{rho2}, this means that $\rho|_{x_1\rightarrow +\infty}\rightarrow \bar{\rho}_0$. On the other hand, it is worth noting that the results of this paper stay valid if we alternatively consider the flow domain
\[
\Omega (t)=\{x\in\mathbb{R}^3\,|\, \varphi (t,x')<x_1<1,\ x'=(x_2,x_3)\in \mathbb{T}^2\}
\]
with the additional fixed boundary (rigid wall)
\[
\Sigma =\{ (1,x'),\ x'\in \mathbb{T}^2\}
\]
on which we prescribe the boundary conditions
\[
v_1=0,\quad F_{1j}=0,\quad \mbox{on}\ [0,T]\times\Sigma
\]
(one can prove that the identities $F_{1j}|_{\Sigma}=0$ are just restrictions on the initial data). Under the change of variables \eqref{13}, \eqref{14} (with such a modified cut-off function $\chi$ that $\chi (1)=0$) the above flow domain $\Omega (t)$ is transformed into the fixed bounded domain $\Omega =\{x_1\in (0,1),\ x'\in \mathbb{T}^2\}$.
\label{r1}
}\end{remark}

Making the change of variables \eqref{13}, \eqref{14} and dropping for convenience the tilde in $\widetilde{U}$, we reduce \eqref{4}, \eqref{10} to the following initial-boundary value problem in the space-time domain $[0,T]\times\Omega$:
\begin{equation}
\mathbb{L}(U,\Psi)=0\quad\mbox{in}\ [0,T]\times \Omega,\label{11.1}
\end{equation}
\begin{equation}
\mathbb{B}(U,\varphi )=0\quad\mbox{on}\ [0,T]\times\partial\Omega,\label{12.1}
\end{equation}
\begin{equation}
U|_{t=0}=U_0\quad\mbox{in}\ \Omega, \qquad \varphi |_{t=0}=\varphi_0\quad \mbox{on}\ \partial\Omega,\label{13.1}
\end{equation}
where
\[
\mathbb{L}(U,\Psi)=L(U,\Psi)U,\quad
L(U,\Psi)=A_0(U)\partial_t +\widetilde{A}_1(U,\Psi)\partial_1+A_2(U )\partial_2+A_3(U )\partial_3,
\]
\[
\widetilde{A}_1(U,\Psi )=\frac{1}{\partial_1\Phi}\left(
A_1(U )-A_0(U)\partial_t\Psi-A_2(U)\partial_2\Psi -A_3(U)\partial_3\Psi\right)
\]
($\partial_1\Phi= 1 +\partial_1\Psi$), and \eqref{12.1} is the compact form of the boundary conditions
\[
\partial_t\varphi-v_N=0,\quad p=0, \quad\mbox{on}\ [0,T]\times\partial\Omega ,
\]
with $v_N= v_1-v_2\partial_2\Psi -v_3\partial_3\Psi$. Here the identities $F^j_N|_{\partial\Omega}=0$ with $F^j_N=F_{1j}-F_{2j}\partial_2\Psi -F_{3j}\partial_3\Psi$ have not been included in the boundary conditions \eqref{12.1} because we can prove the following proposition.

\begin{proposition}
Let the initial data \eqref{13.1} satisfy
\begin{equation}
{\rm div}\,(\rho \mathcal{F}_j) =0\quad\mbox{in}\ \Omega
\label{20}
\end{equation}
and the boundary conditions
\begin{equation}
{F}_N^j =0\quad\mbox{on}\ \partial\Omega ,
\label{21}
\end{equation}
where $\mathcal{F}_j =({F}_N^j,F_{2j}\partial_1\Phi ,F_{3j}\partial_1\Phi )$. If problem \eqref{11.1}--\eqref{13.1} has a sufficiently smooth solution, then this solution satisfies \eqref{20} and \eqref{21} for all $t\in [0,T]$.
\label{p1}
\end{proposition}

\begin{proof}
The proof that \eqref{20} and \eqref{21} are satisfied for all $t\in [0,T]$ if they are true at $t=0$ is, in fact, the same as the proof in \cite{T09} for the divergence and boundary constraints for the magnetic field. Namely, it follows from the first and the last nine equations of system \eqref{11.1} that
\begin{equation}
\partial_t(\rho F_j)+\frac{1}{\partial_1\Phi}\left\{ (w \cdot\nabla ) (\rho F_j)- (\rho\mathcal{F}_j\cdot \nabla ) v + \rho F_j{\rm div}\,u\right\} =0,
\label{22}
\end{equation}
where $u =(v_N,v_2\partial_1\Phi ,v_3\partial_1\Phi )$ and $w= u-(\partial_t\Psi ,0,0)$. In view of the first boundary condition in \eqref{12.1}, we have $w_1|_{x_1=0}=0$. For the rest of the proof we can just refer to \cite{T09}, where the magnetic field $H$ should be formally replaced with $\rho {F}_j$. Briefly speaking, applying div to a consequence of \eqref{22} gives a linear equation for $a_j={\rm div}\, (\rho \mathcal{F}_j)/\partial_1\Phi$ for given $v$ and $\varphi$. It is crucial that this linear equation (see \cite{T09}) does not need a boundary condition for $a_j$ because $w_1|_{x_1=0}=0$. Then by standard method of characteristic curves, we get \eqref{20} for all $t\in [0,T]$. At last, considering \eqref{22} at $x_1=0$ and using again the fact that $w_1|_{x_1=0}=0$, we deduce  \eqref{21} for all $t\in [0,T]$.
\end{proof}

Equations \eqref{20} are just constraints \eqref{8} written in the straightened variables. Proposition \ref{p1} stays valid if  we replace \eqref{11.1} by system \eqref{7} in the straightened variables. This means that these systems are equivalent on solutions of our free boundary problem and we may justifiably replace equations \eqref{7} by \eqref{9} under the fulfilment of the hyperbolicity conditions \eqref{11}.

Concerning the boundary conditions \eqref{21}, we must regard them as the restrictions on the initial data \eqref{13.1}. Otherwise, problem \eqref{11.1}--\eqref{13.1} does not have a correct number of boundary conditions. Indeed, the boundary matrix reads
\begin{equation}
\widetilde{A}_1(U,\Psi )=\frac{1}{\partial_1\Phi}\begin{pmatrix}
{\displaystyle\frac{w_1}{\rho c^2}} & \mathcal{N} & \underline{0} & \underline{0} & \underline{0} \\[7pt]
\mathcal{N}^{\top}&\rho w_1I_3 & -\rho F_N^1I_3 & -\rho F_N^2I_3 & -\rho F_N^3I_3 \\
\underline{0}^{\top} &-\rho F_N^1I_3 & \rho w_1I_3 & O_3 & O_3 \\
\underline{0}^{\top} &-\rho F_N^2I_3 & O_3 & \rho w_1I_3 & O_3 \\
\underline{0}^{\top} &-\rho F_N^3I_3 & O_3 & O_3 & \rho w_1I_3
\end{pmatrix},
\label{23'}
\end{equation}
where $\mathcal{N}=(1, -\partial_2\Psi, -\partial_3\Psi )$ and $w_1$ is the first component of the vector $w$ which was defined just after \eqref{22}. Since $w_1|_{x_1=0}=F_N^j|_{x_1=0}=0$ and $\mathcal{N}|_{x_1=0}=N$, we have
\begin{equation}
\widetilde{A}_1(U,\Psi )|_{x_1=0}=
\begin{pmatrix}
0 & \widetilde{N} \\
\widetilde{N}&O_{12}
\end{pmatrix},
\label{23}
\end{equation}
where  $\widetilde{N}= (N,\underline{0},\underline{0},\underline{0})$. The boundary matrix $\widetilde{A}_1$ on the boundary $x_1=0$ has the eigenvalues $\lambda_1=|N|$, $\lambda_2=-|N|$ and $\lambda_i=0$, $i=\overline{3,13}$. That is, we have one incoming characteristic and the boundary $x_1=0$ is characteristic of constant multiplicity \cite{Rauch}. Since one of the boundary conditions is needed for determining the function $\varphi $, the correct number of boundary conditions is two (that is the case in \eqref{12.1}).

\section{Linearized problem}
\label{s3}

\subsection{The basic state}

Consider
\begin{equation}
 \Omega_T:= (-\infty, T]\times\Omega,\quad \partial\Omega_T:=(-\infty ,T]\times\partial\Omega .
\label{OmegaT}
\end{equation}
Let the basic state
\begin{equation}
(\widehat{U}(t,x),\hat{\varphi}(t,{x}'))
\label{a21}
\end{equation}
upon which we perform the linearization of problem \eqref{11.1}--\eqref{13.1} be a given sufficiently smooth vector-function with $\widehat{U}=(\hat{p},\hat{v},\widehat{F}_1,\widehat{F}_2,\widehat{F}_3)$ and
\begin{equation}
\|\widehat{U}\|_{W^2_{\infty}(\Omega_T)}+\|\hat{\varphi}\|_{W^3_{\infty}(\partial\Omega_T)} \leq K,
\label{a22}
\end{equation}
where $K>0$ is a constant, and below we will also use the notations
\[
\widehat{\Phi}(t,x )= x_1 +\widehat{\Psi}(t,x ),\quad \widehat{\Psi}(t,x )=\chi(x_1)\hat{\varphi}(t,x'),
\]
i.e., all of the ``hat'' values are determined like corresponding values for $(U, \varphi)$:
\[
\hat{\rho}=\rho (\hat{p}),\quad
\hat{v}_{N}=\hat{v}_1- \hat{v}_2\partial_2\widehat{\Psi}- \hat{v}_3\partial_3\widehat{\Psi},\quad \widehat{F}_{N}^j=\widehat{F}_{1j}- \widehat{F}_{2j}\partial_2\widehat{\Psi}- \widehat{F}_{3j}\partial_3\widehat{\Psi},
 \]
etc.
Moreover, without loss of generality we assume that $\|\hat{\varphi}\|_{L_{\infty}(\partial\Omega_T)}<1$. This implies
$\partial_1\widehat{\Phi}\geq 1/2$ .

We assume that the basic state defined in ${\Omega_T}$ satisfies the hyperbolicity conditions \eqref{11},
\begin{equation}
\rho (\hat{p})\geq \bar{\rho}_0>0,\quad \rho' (\hat{p})\geq\bar{\rho}_1>0,\quad\mbox{in}\ \Omega_T,  \label{29}
\end{equation}
the first boundary condition in \eqref{12.1} together with the boundary constraints \eqref{21},
\begin{equation}
\partial_t\hat{\varphi}-\hat{v}_N=0,\quad  \widehat{F}_{N}^j=0,\quad\mbox{on}\ \partial\Omega_T ,
\label{30}
\end{equation}
and the linear equations for $F_{2j}$ and $F_{3j}$ (for a given ${v}$) contained in \eqref{11.1} and considered at $x_1=0$,
\begin{equation}
\partial_t\widehat{F}_{kj}+\hat{v}_2\partial_2\widehat{F}_{kj}+\hat{v}_3\partial_3\widehat{F}_{kj}-\widehat{F}_{2j}\partial_2\hat{v}_k-
\widehat{F}_{3j}\partial_3\hat{v}_k=0\quad (k=2,3)\quad \mbox{on}\ \partial\Omega_T,
\label{31}
\end{equation}
where $j=1,2,3$ and assumptions \eqref{30} were taken into account while writing down \eqref{31}.

\begin{remark}{\rm
For the proof of an a priori estimate for the linearized problem under the fulfilment of the non-collinearity condition \eqref{6} at $x_1=0$ by the basic state \eqref{a21} we will need linearized versions of the boundary constraints \eqref{21}. To deduce these linearized versions of \eqref{21} it is not enough that constraints \eqref{21} themselves are satisfied by the basic state ($\widehat{F}_{N}^j|_{x_1=0}=0$) and we will also need that assumption \eqref{31} holds. Note that assumptions \eqref{29}--\eqref{31} are nonlinear constraints on the basic state which are automatically satisfied if the basic state is an exact solution (unperturbed flow) of problem \eqref{11.1}--\eqref{13.1}. As in \cite{CS2,MTTcont,ST,T09,Tcpam}, the Nash-Moser procedure in Section \ref{s5^} is not completely standard. Namely, at each $n$th Nash-Moser iteration step we have to construct an intermediate state $(U_{n+1/2},\varphi_{n+1/2})$ satisfying constraints  \eqref{29}--\eqref{31}.
\label{r2}
}\end{remark}

\subsection{The linearized equations}

The linearized equations for \eqref{11.1} read:
\[
\mathbb{L}'(\widehat{U},\widehat{\Psi})(\delta U,\delta\Psi):=
\frac{d}{d\varepsilon}\mathbb{L}(U_{\varepsilon},\Psi_{\varepsilon})|_{\varepsilon =0}={f}
\quad \mbox{in}\ \Omega_T,
\]
\[
\mathbb{B}'(\widehat{U},\hat{\varphi})(\delta U,\delta \varphi):=
\frac{d}{d\varepsilon}\mathbb{B}(U_{\varepsilon},\varphi_{\varepsilon})|_{\varepsilon =0}={g}
\quad \mbox{on}\ \partial\Omega_T
\]
where $U_{\varepsilon}=\widehat{U}+ \varepsilon\,\delta U$,
$\varphi_{\varepsilon}=\hat{\varphi}+ \varepsilon\,\delta \varphi$, and
\[
\Psi_{\varepsilon}(t,{x} ):=\chi ( x_1)\varphi _{\varepsilon}(t,{x}'),\quad
\Phi_{\varepsilon}(t,{x} ):=x_1+\Psi_{\varepsilon}(t,{x} ),\quad
\]
\[
\delta\Psi(t,{x} ):=\chi ( x_1)\delta \varphi (t,{x} ).
\]
Here we introduce the source terms
\[
{f}(t,{x} )=(f_1(t,{x} ),\ldots ,f_{13}(t,{x} ))\quad \mbox{and}\quad {g}(t,{x}' )=(g_1(t,{x}' ),g_2(t,{x}' ))
\]
to make the interior equations and the boundary conditions inhomogeneous.

We easily compute the exact form of the linearized equations (below we drop $\delta$):
\[
\mathbb{L}'(\widehat{U},\widehat{\Psi})(U,\Psi)
=
L(\widehat{U},\widehat{\Psi})U +\mathcal{C}(\widehat{U},\widehat{\Psi})
U -   \bigl\{L(\widehat{U},\widehat{\Psi})\Psi\bigr\}\frac{\partial_1\widehat{U}}{\partial_1\widehat{\Phi}},
\]
\[
\mathbb{B}'(\widehat{U},\hat{\varphi})(U,\varphi )=
\left(
\begin{array}{c}
\partial_t\varphi +\hat{v}_2\partial_2\varphi+\hat{v}_3\partial_3\varphi -v_{N}\\
p
\end{array}
\right),
\]
where $v_{N}=v_1-v_2\partial_2\widehat{\Psi}-v_3\partial_3\widehat{\Psi}$, and the matrix
$\mathcal{C}(\widehat{U},\widehat{\Psi})$ is determined as follows:
\[
\mathcal{C}(\widehat{U},\widehat{\Psi})Y
= (Y ,\nabla_yA_0(\widehat{U} ))\partial_t\widehat{U}
 +(Y ,\nabla_y\widetilde{A}_1 (\widehat{U},\widehat{\Psi}))\partial_1\widehat{U}
+ \sum_{k=2}^3(Y ,\nabla_yA_k(\widehat{U} ))\partial_k\widehat{U},
\]
\[
(Y ,\nabla_y A(\widehat{U})):=\sum_{i=1}^{13}y_i\left.\left(\frac{\partial A (Y )}{
\partial y_i}\right|_{Y =\widehat{U}}\right),\quad Y =(y_1,\ldots ,y_{13}).
\]

The differential operator $\mathbb{L}'(\widehat{U},\widehat{\Psi})$ is a first order operator in
$\Psi$. This fact can give some trouble in obtaining a priori estimates for the linearized problem by the energy method. Following
\cite{Al}, we overcome this difficulty by introducing the ``good unknown''
\begin{equation}
\dot{U}:=U -\frac{\Psi}{\partial_1\widehat{\Phi}}\,\partial_1\widehat{U}.
\label{32}
\end{equation}
Omitting simple calculations, we rewrite the linearized interior equations in terms of the new unknown \eqref{32}:
\begin{equation}
L(\widehat{U},\widehat{\Psi})\dot{U} +\mathcal{C}(\widehat{U},\widehat{\Psi})
\dot{U} + \frac{\Psi}{\partial_1\widehat{\Phi}}\,\partial_1\bigl\{\mathbb{L}
(\widehat{U},\widehat{\Psi})\bigr\}=f.
\label{33}
\end{equation}
Dropping as in \cite{CS2,Cont1,MTTcont,ST,T09,Tcpam} the zero-order term in $\Psi$  in \eqref{33},\footnote{In the nonlinear analysis in Section \ref{s5^} the dropped term in \eqref{33} is considered as an error term at each Nash-Moser iteration step.} we write down the final form of our linearized problem for $(\dot{U},\varphi )$:
\begin{align}
\mathbb{L}'_{e}(\widehat{U},\widehat{\Psi})\dot{U} =f &\qquad \mbox{in}\ \Omega_T,\label{34}\\
\mathbb{B}'_{e}(\widehat{U},\hat{\varphi})(\dot{U},\varphi)=
g &\qquad \mbox{on}\ \partial\Omega_T,
\label{35}\\
 (\dot{U},\varphi )=0 &\qquad \mbox{for}\ t<0,\label{36}
\end{align}
where
\begin{align}
& \mathbb{L}'_{e}(\widehat{U},\widehat{\Psi})\dot{U}:=L(\widehat{U},\widehat{\Psi})\dot{U} +\mathcal{C}(\widehat{U},\widehat{\Psi})
\dot{U},\label{L'e}\\[6pt]
& \mathbb{B}'_{e}(\widehat{U},\hat{\varphi})(\dot{U},\varphi):=\left(
\begin{array}{c}
\partial_t\varphi +\hat{v}_2\partial_2\varphi +\hat{v}_3\partial_3\varphi -\dot{v}_{N} - \varphi \partial_1\hat{v}_N\\[3pt]
\dot{p}+ \varphi \partial_1\hat{p}
\end{array}
\right),\label{B'e}
\end{align}
and $\dot{v}_{\rm N}=\dot{v}_1-\dot{v}_2\partial_2\widehat{\Psi}-\dot{v}_3\partial_3\widehat{\Psi}$.  We assume that $f$ and $g$ vanish in the past and consider the case of zero initial data, which is the usual assumption.\footnote{The case of nonzero initial data is postponed to the nonlinear analysis (construction of a so-called approximate solution; see Section \ref{s4^}).}

\subsection{The frozen coefficients problem}

Below we will also need to consider problem \eqref{34}--\eqref{36} with frozen coefficients. More precisely, we consider it in the whole half-space $\mathbb{R}^3_+=\{ x_1>0,\ x'\in\mathbb{R}^2\}$ (without the periodicity conditions) and for all times $t>0$, introduce non-zero initial data, drop the source terms $f$ and $g$ and then freeze the coefficients at a point of the boundary $x_1=0$. For technical simplicity and without loss of generality we consider the case of a planar unperturbed free boundary by assuming that $\hat{\varphi}=0$. Moreover, we can omit the zero-order term $\mathcal{C}\dot{U}$ in \eqref{34} because its presence is not important for the process of construction of an Hadamard-type ill-posedness example (see \cite{Tcpaa}). Then, the final form of the frozen coefficients problem reads (below we drop the dots from the unknowns):
\begin{equation}\label{frozen}
\left\{
\begin{array}{ll}
\displaystyle
\frac{1}{\hat{\rho}\hat{c}^2}\,\partial_t p+{\rm div}\,v =0, &
\\[6pt]
\displaystyle
\hat{\rho}\,\partial_tv +\nabla p -\hat{\rho} \sum_{j=1}^3\mathcal{L}_jF_j=0, &
\\[12pt]
\partial_tF_j -\mathcal{L}_jv =0, &  \ \mbox{in}\ \mathbb{R}_+\times \mathbb{R}^3_+,
\end{array}
\right.
\end{equation}
\begin{equation}\label{frozen_bound}
\partial_t\varphi=v_1 +\hat{a}_0\varphi, \quad
 p=\hat{a} \varphi ,  \quad\mbox{on}\ \mathbb{R}_+\times \{x_1=0\}\times\mathbb{R}^2,
\end{equation}
with some initial data, where $\hat{c}$ is the constant sound speed, $\mathcal{L}_j=\widehat{F}_{2j}\partial_2+\widehat{F}_{3j}\partial_3$, and the coefficients $\hat{a}=-\partial_1\hat{p}$ and $\hat{a}_0=\partial_1\hat{v}_1$ are given constants. In view of assumption \eqref{30} for $\hat{\varphi}=0$, we have $\hat{v}_1=0$ and $\widehat{F}_{1j}=0$. Since $\hat{v}_2$ and $\hat{v}_3$ are now constants, we can apply a Galilean transformation so that the operator $\partial_t+\hat{v}_2\partial_2 +\hat{v}_3\partial_3$ appearing in the problem with constant coefficients becomes $\partial_t$. This is why, without loss of generality it is supposed is \eqref{frozen} that $\hat{v}_2=\hat{v}_3=0$.
Note that for $\hat{a}_0=\hat{a}=0$ problem \eqref{frozen}, \eqref{frozen_bound} is just the result of linearization of the corresponding nonlinear problem about its exact constant solution with $\widehat{F}_{1j}=0$.

\section{Main results}

\label{s04}

We are now in a position to state the main results of this paper.

\begin{theorem} Let $m\in\mathbb{N}$ and $m\geq 6$. Suppose the initial data \eqref{13.1}, with
\[
\left(U_0,\varphi_0\right)\in H^{m+15/2}(\Omega)\times H^{m+15/2}(\partial\Omega),
\]
satisfy the hyperbolicity conditions \eqref{11} and the divergence constraints \eqref{20} for all $x\in\Omega$. Let the initial data satisfy
the Rayleigh-Taylor sign condition\footnote{Assumption \eqref{RT} is just condition \eqref{5} written for the straightened unperturbed free boundary.}
\begin{equation}
\partial_1p|_{x_1=0} \geq \epsilon >0
\label{RT}
\end{equation}
and the boundary constraints \eqref{21} for all $x\in \partial\Omega$. Assume also that the initial data are compatible up to order $m+7$ in the sense of Definition \ref{d1}. Then there exists a sufficiently short time $T>0$ such that problem \eqref{11.1}--\eqref{13.1} has a unique solution
\[
\left(U,\varphi\right)\in H^{m}([0,T]\times\Omega)\times H^{m}([0,T]\times\partial\Omega).
\]
\label{t01}
\end{theorem}

\begin{theorem}
Let $m\in\mathbb{N}$ and $m\geq 8$. Suppose the initial data \eqref{13.1}, with
\[
\left(U_0,\varphi_0\right)\in H^{m+15/2}(\Omega)\times H^{m+15/2}(\partial\Omega),
\]
satisfy the hyperbolicity conditions \eqref{11} and the divergence constraints \eqref{20} for all $x\in\Omega$. Let the initial data satisfy
the non-collinearity condition  (cf. \eqref{6})
\begin{equation}
\label{NC}
\exists\ \mu,\nu\in\{1,2,3\},\ \mu\neq \nu\ :\quad
|{F}_\mu\times {F}_\nu |\geq \delta >0  \quad  \mbox{at}\ x_1=0
\end{equation}
and the boundary constraints \eqref{21} for all $x\in \partial\Omega$. Assume also that the initial data are compatible up to order $m+7$ in the sense of Definition \ref{d1}. Then there exists a sufficiently short time $T>0$ such that problem \eqref{11.1}--\eqref{13.1} has a unique solution
\[
\left(U,\varphi\right)\in H^{m}([0,T]\times\Omega)\times H^{m}([0,T]\times\partial\Omega).
\]
\label{t02}
\end{theorem}

There appears the natural question: What happens if both the non-collinearity and Rayleigh-Taylor sign conditions fail in some points/regions of the initial free boundary? Our hypothesis is that the free boundary problem is not well-posed in this case. However, it is rather difficult and not really necessary to show this on the original nonlinear level and we restrict ourselves to the consideration of the linearized problem.\footnote{We know only one example \cite{Ebin} of the justification of the ill-posedness of a similar but much simpler nonlinear free boundary problem in fluid dynamics. This is the free boundary problem for the incompressible Euler equations with a vacuum boundary condition \cite{Lind_incomp}.}

A direct proof that problem \eqref{34}--\eqref{36} is ill-posed under the simultaneous failure of conditions \eqref{RT} and \eqref{NC} for the basic state is still difficult. On the other hand, the ill-posedness of the corresponding frozen coefficients problem indirectly points out that the variable coefficients problem cannot be well-posed. In fact, in our case this also points out to Rayleigh-Taylor instability.

\begin{theorem}
The frozen coefficients problem \eqref{frozen}, \eqref{frozen_bound} is ill-posed if and only if the three constant vectors $\widehat{F}_j=(0,\widehat{F}_{2j},\widehat{F}_{3j})$, $j=1,2,3$, are collinear,
\begin{equation}
\widehat{F}_{21}\widehat{F}_{32}-\widehat{F}_{31}\widehat{F}_{22}=\widehat{F}_{21}\widehat{F}_{33}-\widehat{F}_{31}\widehat{F}_{23}=
\widehat{F}_{22}\widehat{F}_{33}-\widehat{F}_{32}\widehat{F}_{23}=0,
\label{collin-fr}
\end{equation}
and the constant (frozen) coefficient $\partial_1\hat{p}$ is negative (i.e. when the Rayleigh-Taylor sign condition fails):
\begin{equation}
\partial_1\hat{p}<0.
\label{antiRT}
\end{equation}
\label{t3}
\end{theorem}

\section{Well-posedness of the linearized problem}
\label{s4}

\setcounter{subsubsection}{0}

\subsection{Main theorem for the linearized problem}

The result on the well-posedness of the linearized problem stated below in Theorem \ref{t1} will be used towards the proof of Theorems \ref{t01} and \ref{t02} for the original nonlinear problem.

\begin{theorem}
Let the basic state \eqref{a21} satisfies assumptions \eqref{29}, \eqref{30}, the non-collinearity condition
\begin{equation}
\label{40}
\exists\ \mu,\nu\in\{1,2,3\},\ \mu\neq \nu\ :\quad
|\widehat{F}_\mu\times \widehat{F}_\nu |\geq \delta >0  \quad  \mbox{on}\ \partial\Omega_T,
\end{equation}
and assumption \eqref{31} for $j=\mu$ and $j=\nu $, where $\mu$ and $\nu$ are taken from \eqref{40}.
Then, for all $(f,g) \in H^{3/2}(\Omega_T)\times H^{2}(\partial\Omega_T)$  vanishing in the past problem \eqref{34}--\eqref{36} has a unique solution $(\dot{U},\varphi )\in H^1(\Omega_T)\times H^1(\partial\Omega_T)$ for a sufficiently short time $T$. Moreover, this solution obeys the a priori estimate
\begin{equation}
\|\dot{U} \|_{H^{1}(\Omega_T)}+\|\varphi\|_{H^1(\partial\Omega_T)} \leq C\left\{\|f \|_{H^{3/2}(\Omega_T)}+ \|g\|_{H^{2}(\partial\Omega_T)}\right\},
\label{41}
\end{equation}
where $C=C(K,\bar{\rho}_0,\bar{\rho}_1,\delta ,T)>0$ is a constant independent of the data $f$ and $g$.

Let the basic state \eqref{a21} satisfies assumptions \eqref{29}, \eqref{30} and the Rayleigh-Taylor sign condition
\begin{equation}
\partial_1\hat{p} \geq \epsilon >0 \quad  \mbox{on}\ \partial\Omega_T.
\label{42}
\end{equation}
Then, for all $(f,g) \in H^1(\Omega_T)\times H^{3/2}(\partial\Omega_T)$ vanishing in the past problem \eqref{34}--\eqref{36} has a unique solution $(\dot{U},\varphi )\in H^1(\Omega_T)\times H^1(\partial\Omega_T)$. This solution obeys the a priori estimate
\begin{equation}
\|\dot{U} \|_{H^{1}(\Omega_T)}+\|\varphi\|_{H^1(\partial\Omega_T)} \leq C\left\{\|f \|_{H^{1}(\Omega_T)}+ \|g\|_{H^{3/2}(\partial\Omega_T)}\right\},
\label{43}
\end{equation}
where $C=C(K,\bar{\rho}_0,\bar{\rho}_1,\epsilon ,T)>0$ is a constant independent of the data $f$ and $g$.
\label{t1}
\end{theorem}

\begin{remark}{\rm
Under the fulfilment of the non-collinearity condition \eqref{40} we prove the existence of solutions of problem \eqref{34}--\eqref{36} for a sufficiently short time $T$ whereas we do not assume that $T$ is small for the case when the Rayleigh-Taylor sign condition \eqref{42} holds. The point is that for the first case our proof of existence is based on a fixed-point argument and we manage to use the contraction mapping principle for a sufficiently short time $T$. Probably one could try to prove existence by another method without the short-time assumption but it is not really necessary because our final goal is the proof of the existence of a unique smooth solution of the original nonlinear problem on a small time interval $[0,T]$. Moreover, for both cases \eqref{40} and \eqref{42} the so-called tame a priori estimate for the linearized problem \eqref{34}--\eqref{36} (see Section \ref{s4'}) which is used for the proof of the convergence of Nash-Moser iterations will be deduced for a sufficiently short time $T$.
\label{r3.1}
}\end{remark}

\begin{remark}{\rm
In view of the assumptions $\widehat{F}_N^j|_{x_1=0}=0$, cf. \eqref{30}, one has
\[
|\widehat{F}_\mu\times \widehat{F}_\nu|=|\widehat{F}_{2\mu}\widehat{F}_{3\nu}-\widehat{F}_{3\mu}\widehat{F}_{2\nu}|\sqrt{1+(\partial_2\hat{\varphi})^2+
(\partial_3\hat{\varphi})^2} \quad  \mbox{on}\ \partial\Omega_T.
\]
That is, \eqref{a22} and \eqref{40} imply
\[
\exists\ \mu,\nu\in\{1,2,3\},\ \mu\neq \nu\ :\quad
|\widehat{F}_{2\mu}\widehat{F}_{3\nu}-\widehat{F}_{3\mu}\widehat{F}_{2\nu}|\geq\delta_0 >0\quad  \mbox{on}\ \partial\Omega_T,
\]
with $\delta_0=\delta /\sqrt{1+K^2}$.
\label{r3}
}\end{remark}

\begin{remark}{\rm
Under the fulfilment of the Rayleigh-Taylor sign condition \eqref{42} we will first prove that  the solution to problem \eqref{34}--\eqref{36} obeys the $L^2$ estimate
\begin{equation}
\|\dot{U} \|_{L^{2}(\Omega_T)}+\|\varphi\|_{L^2(\partial\Omega_T)} \leq C\left\{\|f \|_{L^{2}(\Omega_T)}+ \|g\|_{H^{1/2}(\partial\Omega_T)}\right\},
\label{43'}
\end{equation}
where $C=C(K,\bar{\rho}_0,\bar{\rho}_1,\epsilon ,T)>0$ is a constant independent of the data $f$ and $g$.
\label{r3'}
}\end{remark}

\subsection{Reduction to homogeneous boundary conditions and an equivalent reformulation of the interior equations}

Technically, it is more convenient to derive first a priori estimates for a reduced linearized problem with homogeneous boundary conditions (with $g=0$) and then get estimates \eqref{41}, \eqref{43} and \eqref{43'} as their consequences. Using the classical argument, we subtract from the solution a more regular function $\widetilde{U}=(\tilde{p},\tilde{v}_1,0,\ldots ,0)\in H^{s+1}(\Omega_T)$ satisfying the boundary conditions \eqref{35} with $\varphi =0$. Then, the new unknown
\begin{equation}
U^{\natural}=\dot{U}-\widetilde{U} ,\label{a87'}
\end{equation}
with
\begin{equation}
\|\widetilde{U} \|_{H^{s+1}(\Omega_T)}\leq C\|g \|_{H^{s+1/2}(\partial\Omega_T)},
\label{tildU}
\end{equation}
satisfies problem \eqref{34}--\eqref{36} with $f$ replaced by
\begin{equation}
\mathfrak{f} =f-\widehat{A}_0\partial_t\widetilde{U} -\sum_{k=1}^3\widehat{A}_k\partial_k\widetilde{U}  -\widehat{\mathcal{C}}\widetilde{U},
\label{a87''}
\end{equation}
where
\[
\widehat{A}_{\alpha}:=A_{\alpha}(\widehat{U}), \quad \alpha=0,2,3,\quad \widehat{A}_1:=\widetilde{A}_1(\widehat{U},\widehat{\Psi}),\quad
\widehat{\mathcal{C}}:=\mathcal{C}(\widehat{U},\widehat{\Psi}).
\]
Moreover, here and later on $C$ is a positive constant that can change from line to line, and it may depend on other constants, in particular, in \eqref{tildU} the constant $C$ depends on $s$ (sometimes, as in Theorem \ref{t1} we show the dependence of $C$ from another constants). It follows from \eqref{tildU} and \eqref{a87''} (and \eqref{a22}) that
\begin{equation}
\|\mathfrak{f}\|_{H^{s}(\Omega_T)}\leq \|f\|_{H^{s}(\Omega_T)} + C\|g \|_{H^{s+1/2}(\partial\Omega_T)}.
\label{frakf}
\end{equation}

Dropping for convenience the index $^{\natural}$ in \eqref{a87'}, we get our reduced linearized problem:
\begin{align}
\widehat{A}_0\partial_t{U} +\sum_{k=1}^3\widehat{A}_k\partial_k{U}+\widehat{\mathcal{C}}{U} =\mathfrak{f}\qquad &\mbox{in}\ \Omega_T,\label{b48b}\\[3pt]
 {v}_{N}= D_0(\hat{v})\varphi - \varphi \partial_1\hat{v}_N ,\quad p=-\varphi\partial_1\hat{p},  \qquad &\mbox{on}\ \partial\Omega_T
\label{b50b.4} \\[3pt]
({U},\varphi )=0\qquad &\mbox{for}\ t<0,\label{b51b}
\end{align}
where $D_0(\hat{v}):=  \partial_t +\hat{v}_2\partial_2+\hat{v}_3\partial_3$ and $v_{N}=v_1-v_2\partial_2\widehat{\Psi}-v_3\partial_3\widehat{\Psi}$.

As was noted above (see \eqref{23}), the boundary is characteristic of constant multiplicity. It will be convenient to separate the ``noncharacteristic'' part $U_n=(p,v_N)$ of the unknown $U$. For this purpose we introduce the new unknown
\[
\mathcal{U}=(U_n,v_2\partial_1\widehat{\Phi} ,v_3\partial_1\widehat{\Phi},F_1,F_2,F_3)= (p,u,F_1,F_2,F_3),
\]
where $u =(v_N,v_2\partial_1\widehat{\Phi} ,v_3\partial_1\widehat{\Phi} )$. We have $U=\widehat{J}\mathcal{U}$, with
\begin{equation}
\widehat{J}= \begin{pmatrix}
1 & \underline{0} & \underline{0} &\underline{0} &\underline{0} \\
\underline{0}^{\top} & \widehat{\mathfrak{j}}  & O_3 & O_3 &O_3 \\
\underline{0}^{\top} & O_3 & I_3 & O_3 &O_3 \\
\underline{0}^{\top} & O_3 & O_3 & I_3 &O_3 \\
\underline{0}^{\top} & O_3 & O_3 & O_3 &I_3
\end{pmatrix},\quad
\widehat{\mathfrak{j}}=\frac{1}{\partial_1\widehat{\Phi}}\begin{pmatrix} \partial_1\widehat{\Phi}& \partial_2\widehat{\Psi}&  \partial_3\widehat{\Psi}\\
0&1&0\\
 0& 0& 1\end{pmatrix}.
\label{J}
\end{equation}
Then, system \eqref{b48b} is equivalently rewritten as
\begin{equation}\label{b48b'}
\widehat{\mathcal{A}}_0\partial_t{\mathcal U} +\sum_{k=1}^3\widehat{\mathcal{A}}_k\partial_k{\mathcal U}+\widehat{\mathcal{A}}_4{\mathcal U} =\tilde{\mathfrak{f}}\quad \mbox{in}\ \Omega_T,
\end{equation}
where
\[
\widehat{\mathcal{A}}_{\alpha}=\partial_1\widehat{\Phi}\,\widehat{J}^{\,\top} \widehat{A}_{\alpha}\widehat{J},\quad \tilde{\mathfrak{f}}= \partial_1\widehat{\Phi}\,\widehat{J}^{\,\top} \mathfrak{f},
\]
\[
\widehat{\mathcal{A}}_4=
\partial_1\widehat{\Phi}\,\widehat{J}^{\,\top} \Bigl\{ \widehat{A}_0\partial_t\widehat{J}+\sum_{k=1}^3\widehat{A}_k\partial_k\widehat{J}+\widehat{\mathcal{C}}\widehat{J}\Bigr\}.
\]
The symmetric matrices $\widehat{\mathcal{A}}_k$ ($k=1,2,3$) can be represented as
\begin{equation}\label{ak}
\widehat{\mathcal{A}}_k=\mathcal{E}_{1k+1}+\widehat{\mathfrak{A}}_k,
\end{equation}
where (cf. \eqref{23'})
\[
\mathcal{E}_{1k+1}=
\begin{pmatrix}
0 & \tilde{e}_k  \\
\tilde{e}_k^{\top}& O_{12}
\end{pmatrix},
\]
\[
\widehat{\mathfrak{A}}_k=
\begin{pmatrix}
{\displaystyle\frac{\hat{w}_k}{\hat{\rho} \hat{c}^2}} & \underline{0} & \underline{0} & \underline{0} & \underline{0} \\[7pt]
\underline{0}^{\top}&\hat{\rho} \hat{w}_k\widehat{\mathfrak{j}}^{\,\top}\,\widehat{\mathfrak{j}} & -\hat{\rho} \widehat{\mathcal{F}}_{k1}I_3 & -\hat{\rho} \widehat{\mathcal{F}}_{k2}I_3 & -\hat{\rho} \widehat{\mathcal{F}}_{k3}I_3 \\
\underline{0}^{\top} &-\hat{\rho} \widehat{\mathcal{F}}_{k1}I_3 & \hat{\rho} \hat{w}_kI_3 & O_3 & O_3 \\
\underline{0}^{\top} &-\hat{\rho} \widehat{\mathcal{F}}_{k2}I_3 & O_3 & \hat{\rho} \hat{w}_kI_3 & O_3 \\
\underline{0}^{\top} &-\hat{\rho} \widehat{\mathcal{F}}_{k3}I_3 & O_3 & O_3 & \hat{\rho} \hat{w}_kI_3
\end{pmatrix},
\]
$\tilde{e}_k=({e}_k,\underline{0}, \underline{0},\underline{0})$, and
$\widehat{\mathcal{F}}_{kj}$ are the $k$th components of the vectors $\widehat{\mathcal{F}}_j =(\widehat{\mathcal{F}}_{1j},\widehat{\mathcal{F}}_{2j},\widehat{\mathcal{F}}_{3j})=(\widehat{F}_N^j,\widehat{F}_{2j}\partial_1\widehat{\Phi} ,\widehat{F}_{3j}\partial_1\widehat{\Phi} )$. In view of assumption \eqref{30}, we have $\hat{w}_1|_{x_1=0}=\widehat{\mathcal{F}}_{1j}|_{x_1=0}=0$. This implies $\widehat{\mathfrak{A}}_1|_{x_1=0}=0$, i.e.,
\begin{equation}\label{a1}
\widehat{\mathcal{A}}_1|_{x_1=0}=\mathcal{E}_{12}
\end{equation}
(cf. \eqref{23}).

Below we will use the notations
\[
 \Omega_t:= [0, t]\times\Omega\quad\mbox{and}\quad \partial\Omega_t:=[0 ,t]\times\partial\Omega .
\]

\begin{proposition}
Let the basic state \eqref{a21} satisfies assumptions \eqref{30} and \eqref{31} for a certain index $j$. Then sufficiently smooth solutions of problem \eqref{b48b}--\eqref{b51b} satisfy
\begin{equation}
F_{N}^j=\widehat{F}_{2j}\partial_2\varphi +\widehat{F}_{3j}\partial_3\varphi -
\varphi\,\partial_1\widehat{F}_{N}^j+R_j\quad\mbox{on}\ \partial\Omega_T,
\label{59}
\end{equation}
where $F^j_N=F_{1j}-F_{2j}\partial_2\widehat{\Psi} -F_{3j}\partial_3\widehat{\Psi}$ and the function $R_j=R_j(t,x')$ obeys the estimate
\begin{equation}
\|R_j\|_{H^1(\partial\Omega_t)}\leq C\left\{\|\mathfrak{f}\|_{H^{3/2}(\Omega_T)} +\|\varphi\|_{H^{1}(\partial\Omega_t)}\right\}
\label{60}
\end{equation}
for all $t\in [0,T]$.
\label{p2}
\end{proposition}

\begin{proof}
Let $R_j:=F_{N}^j-\widehat{F}_{2j}\partial_2\varphi -\widehat{F}_{3j}\partial_3\varphi +
\varphi\,\partial_1\widehat{F}_{N}^j$. Considering the equations
\[
D(\hat{v},\widehat{\Psi})F_j+\frac{1}{\partial_1\widehat{\Phi}}\left\{ (u\cdot\nabla )\widehat{F}_j- (\widehat{\mathcal{F}}_j\cdot \nabla ) v -(\mathcal{F}_j\cdot \nabla ) \hat{v}\right\} =\mathfrak{f}^{(j)}
\]
contained in \eqref{b48b} at $x_1=0$, using \eqref{30}, \eqref{31} and the first boundary condition in \eqref{b50b.4}, after long but straightforward calculations we obtain
\begin{equation}
D_0(\hat{v})R_j-(\partial_1\hat{v}_N)R_j  =\mathfrak{f}^{(j)}_N+\hat{c}^{\,j}\varphi\quad\mbox{on}\ \partial\Omega_T,
\label{61}
\end{equation}
where
\[
D(\hat{v},\widehat{\Psi}):= \partial_t+\frac{1}{\partial_1\widehat{\Phi}}\, (\hat{w} \cdot\nabla ),\quad
D(\hat{v},\widehat{\Psi})|_{\partial\Omega}=D_0(\hat{v})|_{\partial\Omega},
\]
\[
\mathcal{F}_j =({F}_N^j,F_{2j}\partial_1\widehat{\Phi } ,F_{3j}\partial_1\widehat{\Phi} ),\quad
\mathfrak{f}^{(j)}=(\mathfrak{f}_{2+3j},\mathfrak{f}_{3+3j},\mathfrak{f}_{4+3j}),
\]
\[
\mathfrak{f}^{(j)}_N=\mathfrak{f}^{(j)}\cdot\widehat{\mathcal{N}},\quad
\widehat{\mathcal{N}}=(1,-\partial_2\widehat{\Psi},-\partial_3\widehat{\Psi}) ,\quad
\hat{c}^{\,j}=\partial_1\bigl(\hat{b}_j\cdot \widehat{\mathcal{N}}\bigr)\bigr|_{x_1=0},
\]
\[
 \hat{b}_j=
D(\hat{v},\widehat{\Psi})\widehat{F}_j-\frac{1}{\partial_1\widehat{\Phi}} (\widehat{\mathcal{F}}_j\cdot \nabla ) \hat{v}.
\]
Applying standard arguments of the energy method and the trace theorem for $\mathfrak{f}^{(j)}_N|_{x_1=0}$, from \eqref{61} we easily deduce estimate \eqref{60}.
\end{proof}

\begin{remark}{\rm
The proof of Proposition \ref{p2} is analogous to the proof in \cite{T09} of a linear equation associated with the boundary constraint $(H\cdot N)|_{x_1=0}=0$ for the magnetic field $H$. The term $\hat{c}_0\varphi$ in the right-hand side of \eqref{61} would be zero if, as in \cite{T09}, we assumed that the basic state satisfies corresponding equations contained in our system in the whole domain $\Omega_T$ (i.e., $\hat{b}_j=0$ in $\Omega_T$). Note also that the coefficients $\hat{c}_k^{\,j}$ by the terms $\hat{c}_k^{\,j}\partial_k\varphi$ ($k=2,3$) which could appear in the right-hand side of \eqref{61} vanish thanks to assumption \eqref{31}.
\label{r5}
}\end{remark}

\subsection{Proof of well-posedness under the fulfilment of the Rayleigh-Taylor sign condition}

We first prove the a priori estimate
\begin{equation}
\|{U} \|_{L^{2}(\Omega_T)}+\|\varphi\|_{L^2(\partial\Omega_T)} \leq C\|\mathfrak{f} \|_{L^{2}(\Omega_T)}
\label{62}
\end{equation}
which, by virtue of \eqref{frakf} for $s=0$, implies the $L^2$ estimate \eqref{43'}. Taking into account \eqref{30} and \eqref{a1}, by a standard argument we get for system \eqref{b48b'} the energy inequality
\begin{equation}
I(t)- 2\int\limits_{\partial\Omega_t}({p}{v}_N)|_{x_1=0}\,{\rm d}x'{\rm d}s\leq C
\biggl\{ \| \mathfrak{f}\|^2_{L^2(\Omega_T)} +\int\limits_0^tI(s){\rm d}s\biggr\},
\label{63}
\end{equation}
where $I(t)=\int_{\Omega}(\widehat{\mathcal{A}}_0\mathcal{U},\mathcal{U})\,{d}x =\int_{\Omega}\partial_1\widehat{\Phi}(\widehat{A}_0{U},{U})\,{d}x$. Using the same simple calculations as in \cite{Tcpam}, in view of the boundary conditions \eqref{b50b.4}, we obtain
\[
\begin{split}
-2({p}{v}_N)|_{x_1=0}  = & \;2(\partial_1\hat{p})\varphi (D_0(\hat{v})\varphi - \varphi \partial_1\hat{v}_N)|_{x_1=0} \\
 = &\; \partial_t\left(\partial_1\hat{p}|_{x_1=0}\,\varphi^2\right)+\partial_2\left((\hat{v}_2\partial_1\hat{p})|_{x_1=0}\,\varphi^2\right)
+\partial_3\left((\hat{v}_3\partial_1\hat{p})|_{x_1=0}\,\varphi^2\right) \\
 & +\left.\left(\partial_t\partial_1\hat{p}+\partial_2(\hat{v}_2\partial_1\hat{p})+\partial_3(\hat{v}_3\partial_1\hat{p})
-2\partial_1\hat{p}\,\partial_1\hat{v}_N\right)\right|_{x_1=0}\varphi^2.
\end{split}
\]
Then, it follows from \eqref{63} that
\[
I(t)+\int\limits_{\partial\Omega }  \partial_1\hat{p}|_{x_1=0}\,\varphi^2 \,{\rm d}x'
\leq C \biggl\{ \| \mathfrak{f}\|^2_{L^2(\Omega_T)}
+\int\limits_0^t\left(I(s)
+\|\varphi (s)\|^2_{L^2(\partial\Omega )}\right)\,{\rm d}s\biggr\}.
\]
Taking into account assumptions \eqref{29} and  \eqref{42} and applying Gronwall's lemma, we finally deduce the basic a priori $L^2$ estimate \eqref{62}.

Having in hand the $L^2$ estimate \eqref{62} {\it with no loss of derivatives}, the existence of a weak $L^2$ solution to problem \eqref{b48b}--\eqref{b51b} can be obtained by the classical duality argument. As in \cite{Tcpam}, we can define a dual problem for \eqref{b48b}--\eqref{b51b} and then get for it an $L^2$ a priori estimate, provided that condition \eqref{42} holds for the basic state. We omit detailed calculations which are really similar to those in \cite{Tcpam}. Then, tangential differentiation (with respect to $t$, $x_2$ and $x_3$) and the estimation of the normal ($x_1$-)derivative of the unknowns through tangential ones (see just below) give the existence of an $H^1$ solution to problem \eqref{b48b}--\eqref{b51b}. Its uniqueness follows from the $L^2$ estimate \eqref{62}. Returning to the original unknown $\dot{U}$ (see \eqref{a87'}) and using estimate \eqref{frakf} for $s=1$, we obtain the well-posedness of problem \eqref{34}--\eqref{36} stated in Theorem \ref{t1} under the fulfilment of the Rayleigh-Taylor sign condition \eqref{42}.

It remains to prove the a priori estimate \eqref{43}. Taking into account \eqref{frakf} for $s=1$, it will follow from the $H^1$ estimate
\begin{equation}
\|{U} \|_{H^{1}(\Omega_T)}+\|\varphi\|_{H^1(\partial\Omega_T)} \leq C\|\mathfrak{f} \|_{H^{1}(\Omega_T)}.
\label{64}
\end{equation}
For the proof of \eqref{64} we first deduce an estimate for tangential derivatives. Omitting standard arguments of the energy method, as in \cite{Tcpam}, we easily get this estimate (it is better to call it the preparatory inequality for obtaining \eqref{64}):
\begin{equation}
\nt U(t)\nt^2_{{\rm tan},1}+\nt \varphi (t)\nt^2_{H^1(\partial\Omega )}\leq C\mathcal{M}(t),
\label{65}
\end{equation}
where
\[
\nt u(t)\nt^2_{{\rm tan},1}:=\| u(t)\|^2_{L^2(\Omega )}+\|\partial_t u(t)\|^2_{L^2(\Omega )}+\|\partial_2 u(t)\|^2_{L^2(\Omega )}+\|\partial_3 u(t)\|^2_{L^2(\Omega )},
\]
\[
\nt u (t)\nt^2_{H^1(D )} :=\|u(t) \|^2_{H^{1}(D)} + \|\partial_tu(t) \|^2_{L^2(D)}\qquad (D=\Omega\quad\mbox{or}\quad D=\partial\Omega ),
\]
\[
\mathcal{M}(t)=\|\mathfrak{f} \|^2_{H^{1}(\Omega_T)} +\int\limits_0^t\mathcal{I}(s){\rm d}s ,\quad
\mathcal{I}(t)=\nt U(t)\nt^2_{H^1(\Omega )}+\nt \varphi (t)\nt^2_{H^1(\partial\Omega )}.
\]

We first estimate the normal derivative of the ``noncharacteristic'' unknown $U_n=(p,v_N)$. In view of \eqref{ak}, it follows from \eqref{b48b'} that
\begin{equation}
\partial_1(v_N,p,\ldots ,0)= \tilde{\mathfrak{f}} -\widehat{\mathcal{A}}_0\partial_t\mathcal{U} - \widehat{\mathfrak{A}}_1\partial_1\mathcal{U} -\widehat{\mathcal{A}}_2\partial_2\mathcal{U}-\widehat{\mathcal{A}}_3\partial_3\mathcal{U}-\widehat{\mathcal{A}}_4\mathcal{U} \quad \mbox{in}\ \Omega_T.
\label{n64}
\end{equation}
Since $\widehat{\mathfrak{A}}_1|_{x_1=0}=0$, this gives us the estimate
\[
\|\partial_1U_n(t)\|^2_{L^2(\Omega )}\leq C\left\{ \|\mathfrak{f} (t)\|^2_{L^{2}(\Omega)}+\nt U(t)\nt^2_{{\rm tan},1}+\|\sigma \partial_1 U(t)\|^2_{L^2(\Omega )} \right\}
\]
which, by virtue of the elementary inequality
\begin{equation}
\nt u(t)\nt^2_{H^{s-1}(D )}\leq \int\limits_{0}^{t}\nt u(\tau )\nt^2_{H^s(D )}d\tau =\|u\|^2_{H^s([0,t]\times D)}
\label{elin}
\end{equation}
for $s=1$, implies
\begin{equation}
\|\partial_1U_n(t)\|^2_{L^2(\Omega )}\leq C\left\{ \|\mathfrak{f}\|^2_{H^{1}(\Omega_T)}+\nt U(t)\nt^2_{{\rm tan},1}+\|\sigma \partial_1 U(t)\|^2_{L^2(\Omega )} \right\},
\label{66}
\end{equation}
where $\sigma=\sigma (x_1)\in C^{\infty}(\mathbb{R}_+)$ is a monotone increasing function such that $\sigma (x_1)=x_1$ in a neighborhood of the origin and $\sigma (x_1)=1$ for $x_1$ large enough.

Since $\sigma |_{x_1=0}=0$, we do not need to use boundary conditions to estimate $\sigma\partial_1U$, and we easily get the inequality
\begin{equation}\label{66'}
\| \sigma\partial_1U(t)\|^2_{L^2(\Omega )}\leq C\biggl\{\|\mathfrak{f} \|^2_{H^{1}(\Omega_T)} +\int\limits_0^t\nt U(s)\nt^2_{H^1(\Omega )}{\rm d}s\biggr\}
\end{equation}
which together with \eqref{65} and \eqref{66} yields
\begin{equation}
\|\partial_1U_n(t)\|^2_{L^2(\Omega )}\leq C\mathcal{M}(t).
\label{67}
\end{equation}

The missing normal derivatives of the ``characteristic'' unknowns $v_2$, $v_3$, $F_j$ ($j=1,2,3$)
can be estimated from equations for the linearized divergences $\xi_j={\rm div}\,\mathcal{R}_j$ in \eqref{20} and a symmetric hyperbolic system for the linearized vorticity
$\omega = \nabla\times\mathrm{v}$ and $\eta_j=\nabla\times \mathfrak{F}_j$ obtained by applying the div and curl operators to equations following from \eqref{b48b}, where
\[
\mathcal{R}_j=\hat{\rho}\mathcal{F}_j+\frac{\widehat{\mathcal{F}}_j}{\hat{c}^2}\,p,\quad
\mathrm{v}=(v_1\partial_1\widehat{\Phi},v_{\tau_2},v_{\tau_3}),\quad \mathfrak{F}_j=(F_{1j}\partial_1\widehat{\Phi},F^j_{\tau_2},F^j_{\tau_3}),
\]
\[
v_{\tau_k}=v\cdot\hat{\tau}_k,\quad F^j_{\tau_k}=F_j\cdot\hat{\tau}_k\quad (k=2,3),\quad \hat{\tau}_2= (\partial_2\widehat{\Psi},1,0),\quad
\hat{\tau}_3= (\partial_3\widehat{\Psi},0,1).
\]
Omitting calculations, we write down the equations for $\xi_j$ and the symmetric hyperbolic system for $\omega$ and $\eta_j$:
\begin{align}
D(\hat{v},\widehat{\Psi})\xi_j +{\rm l.o.t.}={\rm div}\,\tilde{\mathfrak{f}}^{\,(j)}&\qquad \mbox{in}\ \Omega_T,\label{68}
\\
D(\hat{v},\widehat{\Psi})\omega -\frac{1}{\partial_1\widehat{\Phi}} \sum_{j=1}^{3}(\widehat{\mathcal{F}}_j\cdot \nabla )\eta_j+{\rm l.o.t.}=\nabla\times\breve{\mathfrak{f}}^{\,v}&\qquad \mbox{in}\ \Omega_T,\label{69}
\\
D(\hat{v},\widehat{\Psi})\eta_j -\frac{1}{\partial_1\widehat{\Phi}} (\widehat{\mathcal{F}}_j\cdot \nabla )\omega +{\rm l.o.t.}=\nabla\times\breve{\mathfrak{f}}^{\,(j)}&\qquad \mbox{in}\ \Omega_T,\label{70}
\end{align}
where
\[
\tilde{\mathfrak{f}}^{\,(j)}=(\mathfrak{f}^{(j)}_N,\mathfrak{f}_{3+3j}\partial_1\widehat{\Phi},
\mathfrak{f}_{4+3j}\partial_1\widehat{\Phi}),\quad \breve{\mathfrak{f}}^{\,v}=(\mathfrak{f}_2\partial_1\widehat{\Phi},\mathfrak{f}^v_{\tau_2},\mathfrak{f}^v_{\tau_3}),
\]
\[
\breve{\mathfrak{f}}^{\,(j)}=(\mathfrak{f}_{2+3j}\partial_1\widehat{\Phi},\mathfrak{f}^{(j)}_{\tau_2},\mathfrak{f}^{(j)}_{\tau_3}),
\]
\[
\mathfrak{f}^v_{\tau_k}=\mathfrak{f}^v\cdot\hat{\tau}_k,\quad \mathfrak{f}^{(j)}_{\tau_k}=\mathfrak{f}^{(j)}\cdot\hat{\tau}_k,\quad k=2,3,\quad \mathfrak{f}^v=(\mathfrak{f}_2,\mathfrak{f}_3,\mathfrak{f}_4);
\]
``l.o.t.'' represents a sum of lower-order terms  $\hat{c}\,\xi_j$, $\hat{c}\,\omega$, $\hat{c}\,\eta_j$, $\hat{c}\,U_m$ and $\hat{c}\,\partial_jU_m$, here and below $\hat{c}$ is the common notation for a generic coefficient (depending on the basic state \eqref{a21}) whose exact form has no meaning, $U_m$ ($m=\overline{1,13}$) is a component of the unknown $U$ in \eqref{b48b}, and the rest notations were introduced just after \eqref{61}.

All of the equations in \eqref{68}--\eqref{70} do not need boundary conditions because, in view of the first assumption in \eqref{30}, the first component of the vector $\hat{w}$ appearing in the definition of the differential operator $D(\hat{v},\widehat{\Psi})$ is zero on the boundary $x_1=0$. Therefore, omitting detailed simple arguments of the energy method, we easily deduce the estimate
\begin{equation}\label{omega}
\begin{split}
\|\omega (t)\|^2_{L^2(\Omega )}+  \sum_{j=1}^{3} & \left\{\|\xi_j (t)\|^2_{L^2(\Omega )}+ \|\eta_j (t)\|^2_{L^2(\Omega )}\right\} \\  & \leq C\biggl\{\|\mathfrak{f} \|^2_{H^{1}(\Omega_T)} +\int\limits_0^t\nt U(s)\nt^2_{H^1(\Omega )}{\rm d}s\biggr\}
\end{split}
\end{equation}
whose combination with \eqref{65} and \eqref{67} implies
\[
\mathcal{I}(t)\leq C\mathcal{M}(t).
\]
Applying then Gronwall's lemma, we obtain the priori estimate
\[
\|{U} \|_{H^{1}(\Omega_T)}+\|\varphi\|_{H^1(\partial\Omega_T)} \leq C\|\mathfrak{f} \|_{H^{3/2}(\Omega_T)}
\]
which yields \eqref{43}. The proof of Theorem \ref{t1} under the fulfilment of the Rayleigh-Taylor sign condition \eqref{42} is thus complete.

\subsection{Proof of the a priori estimate (\ref{41})}

If, instead of \eqref{42}, the basic state satisfies the non-collinearity condition \eqref{40}, then we are not able to obtain an $L^2$ a priori estimate for the linearized problem and have to prolong system \eqref{b48b'} up to first-order tangential derivatives. After differentiating system \eqref{b48b'} with respect to $t$, $x_2$ and $x_3$ we deduce the following energy inequalities for $\partial_{\alpha}U$ ($\alpha =0,2,3$) with $\partial_0:=\partial_t$:
\begin{equation}
I_{\alpha}(t)-2\int\limits_{\partial\Omega_t}(\partial_{\alpha}{p}\,\partial_{\alpha}{v}_N)|_{x_1=0}\,{\rm d}x'{\rm d}s\leq C\mathcal{M}(t),
\label{71}
\end{equation}
where $I_{\alpha}(t)=\int_{\Omega}(\widehat{\mathcal{A}}_0\partial_{\alpha}\mathcal{U},\partial_{\alpha}\mathcal{U})\,{d}x=
\int_{\Omega}\partial_1\widehat{\Phi}(\widehat{A}_0\partial_{\alpha}{U},\partial_{\alpha}{U})\,{d}x$.
In view of the second boundary condition in \eqref{b50b.4}, it follows from \eqref{71} that
\begin{equation}
I_{\alpha}(t)\leq C\mathcal{M}(t)-2\int\limits_{\partial\Omega_t}\left(\partial_1\hat{p}\,\partial_{\alpha}{\varphi}\,\partial_{\alpha}{v}_N+
(\partial_{\alpha}\partial_1\hat{p})\,{\varphi}\,\partial_{\alpha}{v}_N\right)|_{x_1=0}\,{\rm d}x'{\rm d}s.
\label{72}
\end{equation}

Under the non-collinearity condition \eqref{40} the front symbol is elliptic, i.e., we can resolve the boundary conditions for the space-time gradient $\nabla_{t,x'}\varphi =(\partial_t\varphi ,\partial_2\varphi ,\partial_3\varphi )$. Indeed, taking into account \eqref{59} for $j=\mu$ and $j=\nu$ and the first boundary condition in \eqref{b50b.4}, we have the following algebraic system for $\nabla_{t,x'}\varphi$:
\begin{equation}
\left\{
\begin{array}{ll}
\widehat{F}_{2\mu}\partial_2\varphi +\widehat{F}_{3\mu}\partial_3\varphi =F_{N}^\mu-R_\mu+\varphi\,\partial_1\widehat{F}_{N}^\mu,\\
\widehat{F}_{2\nu}\partial_2\varphi +\widehat{F}_{3\nu}\partial_3\varphi =F_{N}^\nu-R_\nu+\varphi\,\partial_1\widehat{F}_{N}^\nu,\\
\hat{v}_{2}\partial_2\varphi +\hat{v}_{3}\partial_3\varphi+ \partial_t\varphi ={v}_{N}+ \varphi \partial_1\hat{v}_N,\qquad \ \mbox{on}\ \partial\Omega_T.
\end{array}\right. \label{73}
\end{equation}
By virtue of  \eqref{40} (see also Remark \ref{r3}), we resolve \eqref{73} for $\nabla_{t,x'}\varphi$:
\begin{equation}
\nabla_{t,x'}\varphi =\hat{a}_1(F_{N}^\mu-R_\mu)+\hat{a}_2(F_{N}^\nu-R_\nu)+\hat{a}_3v_N+\hat{a}_4\varphi\quad \mbox{on}\ \partial\Omega_T,
\label{74}
\end{equation}
where the vector-functions $\hat{a}_{\alpha}={a}_{\alpha}(\widehat{U}_{|x_1=0},\hat{\varphi})$ ($\alpha =\overline{0,4}$) can be easily written in explicit form.

Using \eqref{74}, we reduce the term $-2\partial_1\hat{p}\,\partial_{\alpha}{\varphi}\,\partial_{\alpha}{v}_N|_{x_1=0}$ appearing in the boundary integral in \eqref{72} to the sum of the ``lower-order'' terms
\begin{equation}
\begin{split}
\hat{c}F_{N}^\mu\partial_{\alpha}{v}_N,\quad \hat{c}F_{N}^\nu\partial_{\alpha}{v}_N,\quad \hat{c}R_{\mu}\partial_{\alpha}{v}_N, & \\
\hat{c}R_{\nu}\partial_{\alpha}{v}_N,\quad \hat{c}v_N\partial_{\alpha}{v}_N,\quad \hat{c}\varphi\partial_{\alpha}{v}_N &\quad \mbox{on}\ \partial\Omega_T .
\end{split}
\label{bterms}
\end{equation}
The terms $\hat{c}F_{N}^j\partial_{\alpha}{v}_N|_{x_1=0}$ (with $j=\mu$ and $j=\nu$) are estimated by passing to the volume integral and integrating by parts:
\[
\begin{split}
\int\limits_{\partial\Omega_t} &\hat{c}\,F_{N}^j\partial_{\alpha}{v}_N|_{x^1=0}\,{\rm d}x'{\rm d}s=-\int\limits_{\Omega_t}\partial_1\bigl(\tilde{c}F_{N}^j\partial_{\alpha}{v}_N\bigr){\rm d}x{\rm d}s \\
&=\int\limits_{\Omega_t}\Bigl\{\tilde{c}\partial_{\alpha}F_{N}^j\partial_1{v}_N +(\partial_{\alpha}\tilde{c})F_{N}^j\partial_1{v}_N -\tilde{c}\partial_1F_{N}^j\partial_{\alpha}{v}_N-(\partial_1\tilde{c})F_{N}^j\partial_1{v}_N\Bigr\}{\rm d}x{\rm d}s \\
&\quad-\int\limits_{\Omega_t}\partial_{\alpha}\bigl(\tilde{c}F_{N}^j\partial_1{v}_N\bigr){\rm d}x{\rm d}s\\
& \leq \|U\|^2_{H^1(\Omega_t )}-\int\limits_{\Omega_t}\partial_{\alpha}\bigl(\tilde{c}F_{N}^j\partial_1{v}_N\bigr){\rm d}x{\rm d}s,
\end{split}
\]
where $\tilde{c}|_{x_1=0}=\hat{c}$. If $\alpha =2$ or $\alpha =3$ the last integral above is equal to zero. But for $\alpha =0$ we have:
\[
-\int\limits_{\Omega_t}\partial_s\bigl(\tilde{c}F_{N}^j\partial_1{v}_N\bigr){\rm d}x{\rm d}s=
-\int\limits_{\Omega}\tilde{c}F_{N}^j\partial_1{v}_N{\rm d}x.
\]
Using the Young inequality and the elementary inequality \eqref{elin}, we estimate the last integral as follows:
\[
\begin{split}
-\int\limits_{\Omega}\tilde{c}F_{N}^j\partial_1{v}_N{\rm d}x &\leq C\left\{ \tilde{\varepsilon}\nt U(t)\nt^2_{H^1(\Omega )}
+\frac{1}{\tilde{\varepsilon}}\|U (t)\|^2_{L^2(\Omega )}\right\}\\ &\leq C\left\{ \tilde{\varepsilon}\nt U(t)\nt^2_{H^1(\Omega )}
+\frac{1}{\tilde{\varepsilon}}\|U\|^2_{H^1(\Omega_t )}\right\},
\end{split}
\]
where $\tilde{\varepsilon} $ is a small positive constant.

The rest terms in \eqref{bterms} are estimated in the same way as above. The functions $R_j$ and $\varphi$ appear in the volume integral as $\chi (x_1)R_j$ and $\Psi$ respectively, where $\chi (x_1)$ is the ``lifting'' function from \eqref{14}. We also use estimate \eqref{60} for $R_j$. Omitting technical details, we finally obtain the following estimate for the boundary integral in \eqref{72}:
\begin{equation}\label{77}
\begin{split}
-2\int\limits_{\partial\Omega_t}(\partial_1\hat{p}\,\partial_{\alpha}{\varphi}\,\partial_{\alpha}{v}_N+
(\partial_{\alpha}\partial_1\hat{p})&\,{\varphi}\,\partial_{\alpha}{v}_N)|_{x_1=0} \,{\rm d}x'{\rm d}s \\ & \leq
\tilde{\varepsilon}C \nt U(t)\nt^2_{H^1(\Omega )}
+\frac{C}{\tilde{\varepsilon}}\mathcal{N}(t),
\end{split}
\end{equation}
where
\[
\mathcal{N}(t)=\|\mathfrak{f} \|^2_{H^{3/2}(\Omega_T)} +\int\limits_0^t\mathcal{I}(s){\rm d}s.
\]
Then, \eqref{72} and \eqref{77} for $\alpha =0,2,3$ imply the inequality
\[
\nt U(t)\nt^2_{{\rm tan},1}\leq
\tilde{\varepsilon}C \nt U(t)\nt^2_{H^1(\Omega )}
+\frac{C}{\tilde{\varepsilon}}\mathcal{N}(t)
\]
whose combination with \eqref{66}, \eqref{66'} and \eqref{omega} gives
\[
\nt U(t)\nt^2_{H^1(\Omega )}\leq
\tilde{\varepsilon}C \nt U(t)\nt^2_{H^1(\Omega )}
+\frac{C}{\tilde{\varepsilon}}\mathcal{N}(t).
\]
We then absorb the term $\tilde{\varepsilon}C \nt U(t)\nt^2_{H^1(\Omega )}$ in the left-hand side of the last inequality by choosing $\tilde{\varepsilon}$ small enough:
\begin{equation}\label{80}
\nt U(t)\nt^2_{H^1(\Omega )}\leq C\mathcal{N}(t).
\end{equation}

Using \eqref{74}, inequality \eqref{elin}, estimate \eqref{60} and the trace theorem, we obtain the inequality
\[
\nt \varphi (t)\|_{H^1(\partial\Omega )}^2\leq C\left\{ \nt U(t)\nt^2_{H^1(\Omega )}+\mathcal{N}(t)\right\}
\]
which together with \eqref{80} implies
\[
\mathcal{I}(t)\leq C\mathcal{N}(t).
\]
Applying then Gronwall's lemma, we obtain the priori estimate
\[
\|{U} \|_{H^{1}(\Omega_T)}+\|\varphi\|_{H^1(\partial\Omega_T)} \leq C\|\mathfrak{f} \|_{H^{3/2}(\Omega_T)}
\]
which, in view of \eqref{frakf} for $s=3/2$, yields \eqref{41}.

\subsection{Proof of well-posedness under the fulfilment of the non-collinearity condition}

To prove the existence of the solution of problem \eqref{b48b}--\eqref{b51b} under the fulfilment of the non-collinearity condition \eqref{40} we use the same idea as in \cite{SecTra}. We first solve the problem
\begin{align}
\widehat{A}_0\partial_t{U} +\sum_{k=1}^3\widehat{A}_k\partial_k{U}+\widehat{\mathcal{C}}{U} =\mathfrak{f}\qquad &\mbox{in}\ \Omega_T,\label{81}\\[3pt]
p=-\overline{\varphi}\partial_1\hat{p},  \qquad &\mbox{on}\ \partial\Omega_T
\label{82} \\[3pt]
{U}=0\qquad &\mbox{for}\ t<0,\label{83}
\end{align}
under the assumption that $\overline{\varphi}$ is given.

\begin{lemma}
Let the basic state \eqref{a21} satisfies assumptions \eqref{29} and \eqref{30}. Then, for all given $(\mathfrak{f},\overline{\varphi} ) \in H^{1}(\Omega_T)\times H^{3/2}(\partial\Omega_T)$  vanishing in the past problem \eqref{81}--\eqref{83}  has a unique solution $U\in H^{1}(\Omega_T)$ such that
\begin{equation}\label{84}
\|{U} \|_{H^{1}(\Omega_T)}\leq C\left\{\|\mathfrak{f} \|_{H^{1}(\Omega_T)}+ \|\overline{\varphi}\|_{H^{3/2}(\partial\Omega_T)}\right\}.
\end{equation}
\label{lem1}
\end{lemma}

\begin{proof}
If we consider the boundary condition \eqref{82} in homogeneous form, i.e., if we set $\overline{\varphi} =0$, then
\[
(\widehat{A}_1U,U)|_{\partial\Omega}=(\mathcal{E}_{12}\mathcal{U},\mathcal{U})|_{\partial\Omega}=2(pv_N)|_{x_1=0}=0.
\]
That is, the boundary condition \eqref{82} is nonnegative for \eqref{81}. As system \eqref{81} has one incoming characteristic and this is in agreement with the fact that we have one boundary condition we infer that the boundary condition \eqref{82} is maximally nonnegative (but not strictly dissipative). Following the classical argument, we reduce our problem (with non-zero $\overline{\varphi}$) to one with a homogeneous boundary condition (with  $\overline{\varphi}=0$) by subtracting from $U$ a more regular function $U^\sharp=(p^\sharp,0,\ldots ,0)\in H^{2}(\Omega_T)$ such that $p^\sharp=-\overline{\varphi}\partial_1\hat{p}$ on $\partial\Omega_T$. Since the boundary is characteristic of constant multiplicity, we may apply the result of \cite{Sec96} and we get the solution $U\in H^1_*(\Omega_T)$ with $U_n\in H^1(\Omega_T)$, where $H^1_*$ is the conormal Sobolev space equipped with the norm
$
\|u\|^2_{H^{1}_*(\Omega_T)}:=\int_0^T\nt u(s)\nt^2_{{\rm tan},1}\,{\rm d}s +\|\sigma\partial_1  u\|^2_{L^2(\Omega_T )}.
$
Taking then \eqref{68}--\eqref{omega} into account, we get estimate \eqref{84} and the solution $U\in H^1(\Omega_T)$.
\end{proof}

\begin{remark}{\rm
In fact, before obtaining the a priori estimate \eqref{84}, by using Gronwall's lemma, we first get the estimate
\[
\|{U}(t) \|^2_{H^{1}(\Omega)}\leq Ce^{CT}\left\{\|\mathfrak{f} \|^2_{H^{1}(\Omega_T)}+ \|\overline{\varphi}\|^2_{H^{3/2}(\partial\Omega_T)}\right\}
\]
whose integration over the interval $[0,T]$ gives estimate \eqref{84} with the constant $C=C(T)$ being proportional to $T^{1/2}$. That is, the constant $C=C(T)$ in the right-hand side of \eqref{84} tends to zero as $T\rightarrow 0$.
\label{r84}
}\end{remark}

The existence of the unique solution of problem \eqref{b48b}--\eqref{b51b}  is given by the following theorem.

\begin{theorem}
Let the basic state \eqref{a21} satisfies assumptions \eqref{29}, \eqref{30}, the non-collinearity condition \eqref{40} and assumption \eqref{31} for $j=\mu$ and $j=\nu $, where $\mu$ and $\nu$ are taken from \eqref{40}. Then, for all $\mathfrak{f} \in H^{1}(\Omega_T)$  vanishing in the past problem \eqref{b48b}--\eqref{b51b} has a unique solution $(U,\varphi )\in H^1(\Omega_T)\times H^{3/2}(\partial\Omega_T)$.
\label{t4}
\end{theorem}

\begin{proof}
We prove the existence of the solution to \eqref{b48b}--\eqref{b51b} by a fixed-point argument. Consider $\overline{\varphi}\in H^{3/2}(\partial\Omega_T)$ vanishing in the past. By Lemma \ref{lem1}, there exists a unique solution $U\in H^1(\Omega_T)$ of problem \eqref{81}--\eqref{83} enjoying the a priori estimate \eqref{84}. We now consider the Cauchy problem (cf. \eqref{b51b})
\begin{align}
D_0(\hat{v})\varphi - \varphi \partial_1\hat{v}_N ={v}_{N}& \qquad \mbox{on}\ \partial\Omega_T, \label{85} \\
\varphi =0& \qquad \mbox{for}\ t<0, \label{85'}
\end{align}
where ${v_N}_{|x_1=0}\in H^{1/2}(\partial\Omega_T)$ is the trace of the normal component of $v$ contained in the solution $U\in H^1(\Omega_T)$ of problem \eqref{81}--\eqref{83}. Clearly, there exists a unique solution $\varphi\in H^{1/2}(\partial\Omega_T)$ of \eqref{85}, \eqref{85'} such that
\begin{equation}\label{86}
\|\varphi\|_{H^{1/2}(\partial\Omega_T)}\leq C\|{v_N}_{|x_1=0}\|_{H^{1/2}(\partial\Omega_T)}.
\end{equation}

Using \eqref{85}, we can obtain equation \eqref{59} for $j=\mu$ and $j=\nu$. Moreover, the unique solution $R_j$ of  \eqref{61} obeys the estimate
\begin{equation}
\|R_j\|_{H^{1/2}(\partial\Omega_T)}\leq C\left\{\|\mathfrak{f}\|_{H^{1}(\Omega_T)} +\|\varphi\|_{H^{1/2}(\partial\Omega_T)}\right\}.
\label{87}
\end{equation}
That is, we get system \eqref{73} which implies \eqref{74}. In view of \eqref{87}, the trace theorem and \eqref{84}, from  \eqref{86} and \eqref{74} we obtain the estimate
\[
\begin{split}
\|\varphi\|_{H^{1/2}(\partial\Omega_T)}+\|\nabla_{t,x'}\varphi\|_{H^{1/2}(\partial\Omega_T)} &
\leq C\Bigl\{ \|U_{|x_1=0}\|_{H^{1/2}(\partial\Omega_T)}+\|R_\mu\|_{H^{1/2}(\partial\Omega_T)} \\  & \quad\qquad +\|R_\nu\|_{H^{1/2}(\partial\Omega_T)}+
\|\varphi\|_{H^{1/2}(\partial\Omega_T)}\Bigr\}\\
& \leq C\Bigl\{ \|\mathfrak{f} \|_{H^{1}(\Omega_T)}+ \|\overline{\varphi}\|_{H^{3/2}(\partial\Omega_T)}\Bigr\}.
\end{split}
\]
which yields
\begin{equation}\label{88}
\|\varphi\|_{H^{3/2}(\partial\Omega_T)}\leq  C\left\{ \|\mathfrak{f} \|_{H^{1}(\Omega_T)}+ \|\overline{\varphi}\|_{H^{3/2}(\partial\Omega_T)}\right\}.
\end{equation}
This defines a map $\overline{\varphi}\rightarrow \varphi$ in $H^{3/2}(\partial\Omega_T)$. Let $\overline{\varphi}^1,\; \overline{\varphi}^2\in H^{3/2}(\partial\Omega_T)$, and $(U^1,\varphi^1)$, $(U^2,\varphi^2)$ be the corresponding solutions of problem \eqref{81}--\eqref{83}, \eqref{85}, \eqref{85'}. Thanks to the linearity of this problem we obtain, as for \eqref{88},
\begin{equation}\label{90}
\|\varphi^1-\varphi^2\|_{H^{3/2}(\partial\Omega_T)}\leq  C (T)\|\overline{\varphi}^1-\overline{\varphi}^2\|_{H^{3/2}(\partial\Omega_T)},
\end{equation}
where $C(T)\rightarrow 0$ as $T\rightarrow 0$ because, as in \eqref{84} (see Remark \ref{r84}), the constants appearing in the right-hand sides of \eqref{86} and \eqref{87} tend to zero as $T\rightarrow 0$.
Then there exists a sufficiently short time $T>0$ such that $C(T)<1$ in \eqref{90} and, hence, the map $\overline{\varphi}\rightarrow \varphi$ has a unique fixed point by the contraction mapping principle. At this fixed point $\overline{\varphi}=\varphi$ problem \eqref{81}--\eqref{83}, \eqref{85}, \eqref{85'} coincides with \eqref{b48b}--\eqref{b51b}, i.e., $(U,\varphi)$ is the unique solution of problem \eqref{b48b}--\eqref{b51b}.
\end{proof}

It follows from Theorem \ref{t4} and \eqref{frakf} for $s=1$ that for all $(f,g) \in H^1(\Omega_T)\times H^{3/2}(\partial\Omega_T)$ vanishing in the past problem \eqref{34}--\eqref{36} has a unique solution $(\dot{U},\varphi )\in H^1(\Omega_T)\times H^{3/2}(\partial\Omega_T)$. Together with the deduced a priori estimate \eqref{41} for $(f,g) \in H^{3/2}(\Omega_T)\times H^{2}(\partial\Omega_T)$  this completes the proof of Theorem \ref{t1} under the fulfilment of the non-collinearity condition \eqref{40}.

\begin{remark}{\rm
We could improve estimate \eqref{41} by replacing the $H^1$ norm of $\varphi$ with its $H^{3/2}$ norm. This is, however, not necessary for the subsequent nonlinear analysis. On the other hand, we are not able to get the a priori estimate \eqref{43} corresponding to the regularity $ H^1(\Omega_T)\times H^{3/2}(\partial\Omega_T)$ of $(f,g)$ for which we have proved the existence of the solution because after passing to the volume integral in \eqref{72} we integrate by parts and cannot gain  ``1/2 derivative'' for $R_j$ in order to apply \eqref{87} instead of \eqref{60}.
\label{r6}
}\end{remark}

\section{Tame estimates}
\label{s4'}

\setcounter{subsubsection}{0}

\subsection{Tame a priori estimates for problem (\ref{34})--(\ref{36})}

We are going to derive tame a priori estimates in $H^s$ for problem \eqref{34}--\eqref{36}, with $s$ large enough.
These tame estimates (see Theorems \ref{t4.1} and \ref{t4.2} below) being, roughly speaking, linear in high norms (that are
multiplied by low norms) are with  a fixed loss of derivatives from the data $f$ and $g$ and with respect to the coefficients, i.e., with respect to the basic state (\ref{a21}).

\begin{theorem}
Let $T>0$ and $s\in \mathbb{N}$, with $s\geq 3$. Assume that the basic state $(\widehat{U} ,\hat{\varphi})\in
H^{s+3}(\Omega_T )\times H^{s+3}(\partial\Omega_T)$ satisfies assumptions \eqref{29}, \eqref{30}, the Rayleigh-Taylor sign condition \eqref{42} and
\begin{equation}
\|\widehat{U}\|_{H^6(\Omega_T )} +\|\hat{\varphi} \|_{H^{6}(\partial\Omega_T)}\leq \widehat{K},
\label{37}
\end{equation}
where $\widehat{K}>0$ is a constant. Let also the data $(f ,g)\in
H^{s}(\Omega_T )\times H^{s+1}(\partial\Omega_T)$ vanish in the past. Then there exists a positive constant $K_0$ that does not depend on $s$ and $T$ and there exists a constant $C(K_0) >0$ such that, if $\widehat{K}\leq K_0$, then there exists a unique solution $(\dot{U} ,\varphi)\in H^{s}(\Omega_T )\times H^{s}(\partial\Omega_T)$ to problem \eqref{34}--\eqref{36} that obeys the tame a priori  estimate
\begin{equation}
\begin{split}
\|\dot{U}\|_{H^s(\Omega_T )} & +\|\varphi\|_{H^{s}(\partial\Omega_T)}\leq  C(K_0)\Bigl\{
\|f\|_{H^{s}(\Omega_T )}+ \|g \|_{H^{s+1}(\partial\Omega_T)} \\
 &+\bigl( \|f\|_{H^{3}(\Omega_T )}+ \| g\|_{H^{4}(\partial\Omega_T)} \bigr)\bigl(
\|\widehat{U}\|_{H^{s+3}(\Omega_T )}+\|\hat{\varphi}\|_{H^{s+3}(\partial\Omega_T)}\bigr)\Bigr\}
\end{split}
\label{38}
\end{equation}
for a sufficiently short time $T$ (the constant $C(K_0)$ depends also on the fixed constants $\bar{\rho}_0$, $\bar{\rho}_1$ and $\epsilon$
from \eqref{29} and \eqref{42}).
\label{t4.1}
\end{theorem}

\begin{remark}{\rm
Theorem \ref{t4.1} looks similar to that in \cite{Tcpam} for the free boundary problem for the compressible Euler equations with a vacuum boundary condition. As in \cite{Tcpam}, and unlike Theorem \ref{t1}, here we prefer to work with integer indices of Sobolev spaces, i.e., we derive a little bit roughened version of the tame estimate where, in particular, we loose one but not ``half'' derivative from $g$. We do so just for technical convenience because an additional gain of ``half'' derivative in the local-in-time existence theorem is not really principal (e.g., from the physical point of view).
\label{r4}
}\end{remark}

\begin{theorem}
Let $T>0$ and $s\in \mathbb{N}$, with $s\geq 3$. Assume that the basic state \eqref{a21} satisfies assumptions \eqref{29}, \eqref{30}, the non-collinearity condition \eqref{40} and assumption \eqref{31} for $j=\mu$ and $j=\nu $, where $\mu$ and $\nu$ are taken from \eqref{40} and
\begin{equation}
\|\widehat{U}\|_{H^5(\Omega_T )} +\|\hat{\varphi} \|_{H^{5}(\partial\Omega_T)}\leq \widehat{K},
\label{37'}
\end{equation}
where $\widehat{K}>0$ is a constant. Let also the data $(f ,g)\in
H^{s+1}(\Omega_T )\times H^{s+2}(\partial\Omega_T)$ vanish in the past. Then there exists a positive constant $K_0$ that does not depend on $s$ and $T$ and there exists a constant $C(K_0) >0$ such that, if $\widehat{K}\leq K_0$, then there exists a unique solution $(\dot{U} ,\varphi)\in H^{s}(\Omega_T )\times H^{s}(\partial\Omega_T)$ to problem \eqref{34}--\eqref{36} that obeys the tame a priori  estimate
\begin{equation}
\begin{split}
\|\dot{U}\|_{H^s(\Omega_T )} & +\|\varphi\|_{H^{s}(\partial\Omega_T)}\leq  C(K_0)\Bigl\{
\|f\|_{H^{s+1}(\Omega_T )}+ \|g \|_{H^{s+2}(\partial\Omega_T)} \\
 &+\bigl( \|f\|_{H^{4}(\Omega_T )}+ \| g\|_{H^{5}(\partial\Omega_T)} \bigr)\bigl(
\|\widehat{U}\|_{H^{s+2}(\Omega_T )}+\|\hat{\varphi}\|_{H^{s+2}(\partial\Omega_T)}\bigr)\Bigr\}
\end{split}
\label{38"}
\end{equation}
for a sufficiently short time $T$ (the constant $C(K_0)$ depends also on the fixed constants $\bar{\rho}_0$, $\bar{\rho}_1$ and $\delta$
from \eqref{29} and \eqref{40}).
\label{t4.2}
\end{theorem}

\subsection{Tame a priori estimates for the reduced problem (\ref{b48b})--(\ref{b51b})}

We first prove tame estimates for the reduced linearized problem \eqref{b48b}--\eqref{b51b} with homogeneous boundary conditions. After that we get tame estimates for problem \eqref{34}--\eqref{36} by using  the roughened version of \eqref{tildU} (see Remark \ref{r4})
\begin{equation}
\|\widetilde{U} \|_{H^{s+1}(\Omega_T)}\leq C\|g \|_{H^{s+1}(\partial\Omega_T)}.
\label{tildU'}
\end{equation}

That is, from now on we concentrate on the proof of tame estimates for the reduced problem \eqref{b48b}--\eqref{b51b}. Namely, we will prove the following lemmata.

\begin{lemma}
Let $T>0$ and $s\in \mathbb{N}$, with $s\geq 3$. Assume that the basic state $(\widehat{U} ,\hat{\varphi})\in
H^{s+3}(\Omega_T )\times H^{s+3}(\partial\Omega_T)$ satisfies assumptions satisfies assumptions \eqref{29}, \eqref{30}, the Rayleigh-Taylor sign condition \eqref{42} and inequality \eqref{37}.
Let also $\mathfrak{f}\in H^{s}(\Omega_T )$ vanishes in the past. Then there exists a positive constant $K_0$ that does not depend on $s$ and $T$ and there exists a constant $C(K_0) >0$ such that, if $\widehat{K}\leq K_0$, then there exists a unique solution $(U ,\varphi)\in H^{s}(\Omega_T )\times H^{s}(\partial\Omega_T)$ to problem \eqref{b48b}--\eqref{b51b} that obeys the tame a priori estimate
\begin{equation}
\begin{split}
\|{U} & \|_{H^s(\Omega_T )}+  \|\varphi\|_{H^{s}(\partial\Omega_T)}\\ &\leq  C(K_0)\Bigl\{
\|\mathfrak{f}\|_{H^{s}(\Omega_T )} +\|\mathfrak{f}\|_{H^{3}(\Omega_T )}\bigl(
\|\widehat{U}\|_{H^{s+3}(\Omega_T )}+\|\hat{\varphi}\|_{H^{s+3}(\partial\Omega_T)}\bigr)\Bigr\}
\end{split}
\label{38'}
\end{equation}
for a sufficiently short time $T$ (the constant $C(K_0)$ depends also on the fixed constants $\bar{\rho}_0$, $\bar{\rho}_1$ and $\epsilon$
from \eqref{29} and \eqref{42}).
\label{l3.1}
\end{lemma}

\begin{lemma}
Let $T>0$ and $s\in \mathbb{N}$, with $s\geq 3$. Assume that the basic state $(\widehat{U} ,\hat{\varphi})\in
H^{s+2}(\Omega_T )\times H^{s+2}(\partial\Omega_T)$ satisfies satisfies assumptions \eqref{29}, \eqref{30}, the non-collinearity condition \eqref{40}, assumption \eqref{31} for $j=\mu$ and $j=\nu $  (where $\mu$ and $\nu$ are taken from \eqref{40}) and inequality \eqref{37'}.
Let also $\mathfrak{f}\in H^{s+1}(\Omega_T )$ vanishes in the past. Then there exists a positive constant $K_0$ that does not depend on $s$ and $T$ and there exists a constant $C(K_0) >0$ such that, if $\widehat{K}\leq K_0$, then there exists a unique solution $(U ,\varphi)\in H^{s}(\Omega_T )\times H^{s}(\partial\Omega_T)$ to problem \eqref{b48b}--\eqref{b51b} that obeys the tame a priori estimate
\begin{equation}
\begin{split}
\|{U} & \|_{H^s(\Omega_T )}+  \|\varphi\|_{H^{s}(\partial\Omega_T)}\\ &\leq  C(K_0)\Bigl\{
\|\mathfrak{f}\|_{H^{s+1}(\Omega_T )} +\|\mathfrak{f}\|_{H^{4}(\Omega_T )}\bigl(
\|\widehat{U}\|_{H^{s+2}(\Omega_T )}+\|\hat{\varphi}\|_{H^{s+2}(\partial\Omega_T)}\bigr)\Bigr\}.
\end{split}
\label{38''}
\end{equation}
for a sufficiently short time $T$ (the constant $C(K_0)$ depends also on the fixed constants $\bar{\rho}_0$, $\bar{\rho}_1$ and $\delta$
from \eqref{29} and \eqref{40}).
\label{l3.2}
\end{lemma}

\subsection{Proof of the tame estimate under the fulfilment of the Rayleigh-Taylor sign condition}

Since arguments of the energy method below are quite standard, we will omit detailed calculations. By applying to system \eqref{b48b'} the operator $\partial^{\alpha}_{\rm tan}=\partial_t^{\alpha_0}\partial_2^{\alpha_2}\partial_3^{\alpha_3}$ with $|\alpha |=
|(\alpha_0,\alpha_2, \alpha_3)|\leq s$ and taking into account \eqref{30} and \eqref{a1}, one gets
\begin{equation}
\int\limits_{\Omega}\widehat{\cal A}_0\partial^{\alpha}_{\rm tan}\mathcal{U}\cdot\partial^{\alpha}_{\rm tan}\mathcal{U} \, {\rm d}x - 2\int\limits_{\partial\Omega_t}\partial^{\alpha}_{\rm tan}p\,\partial^{\alpha}_{\rm tan}{v}_N|_{x_1=0}\,{\rm d}x'{\rm d}s={\cal R},\label{39el}
\end{equation}
where
\[
\begin{split}
{\cal R}=\int\limits_{\Omega_t}
\biggl(\,\sum_{l =0}^3\Big\{\partial_l\widehat{\cal A}_l\partial^{\alpha}_{\rm tan}\mathcal{U}  - 2[\partial^{\alpha}_{\rm tan} , & \widehat{\cal A}_l]\partial_l\mathcal{U} \Big\} \\ & - 2\partial^{\alpha}_{\rm tan}(\widehat{\cal A}_4\mathcal{U} )  +2\partial^{\alpha}_{\rm tan}\tilde{\mathfrak{f}}\biggr)\cdot\partial^{\alpha}_{\rm tan}\mathcal{U}\,{\rm d}x {\rm d}s ,
\end{split}
\]
$\partial_0:=\partial_t$, and we use the notation of commutator: $[a,b]c:=a(bc)-b(ac)$. Using the Moser-type calculus inequalities
\begin{align}
\| uv\|_{H^s(\Omega_T)}\leq & C\left( \|u\|_{H^s(\Omega_T)}\|v\|_{L_{\infty}(\Omega_T)}
+\|u\|_{L_{\infty}(\Omega_T)}\|v\|_{H^s(\Omega_T)}\right),\label{c40}\\
\| b(u)\|_{H^s(\Omega_T)}\leq & C(M)\|u\|_{H^s(\Omega_T)},\label{c41}
\end{align}
where the function $b$ is a $C^{\infty}$ function of $u$ with $b(0)=0$, and $M$ is such a positive constant that
$\|u\|_{L_{\infty}(\Omega_T)}\leq M$, we estimate the right-hand side in \eqref{39el}:
\begin{equation}
\begin{split}
\mathcal{ R}\leq C(K)\Bigl\{ & \|\mathcal{U}\|^2_{H^s(\Omega_t)}+\|\mathfrak{f}\|^2_{H^s(\Omega_T)}\\ & +\left(\|{U}\|^2_{W^1_{\infty}(\Omega_T)}+\|\mathfrak{f}\|^2_{L_{\infty}(\Omega_T)}\right)\left( 1+\|{\rm coeff}\|^2_{s+2}\right)\Bigr\},
\label{c42}
\end{split}
\end{equation}
where $\|{\rm coeff}\|_{m}:=\|\widehat{U}\|_{H^{m}(\Omega_T)}+\|\hat{\varphi}\|_{H^{m}(\partial\Omega_T)}$. More precisely, here and below we also use the following refinement of \eqref{c40}
\begin{equation}
\begin{split}
\sum\limits_{|\mu |+|\nu |=s}\|\partial^{\mu }u\, & \partial^{\nu}v\|_{L^2(\Omega_T)} \\ &\leq C\left( \|u\|_{H^s(\Omega_T)}\|v\|_{L_{\infty}(\Omega_T)}
+\|u\|_{L_{\infty}(\Omega_T)}\|v\|_{H^s(\Omega_T)}\right)
\end{split}
\label{c40'}
\end{equation}
which implies
\[
\begin{split}
\bigl\|[\partial^{\mu },u]v\bigr\|_{L^2(\Omega_T)} &\leq
C\sum\limits_{|\gamma |+|\nu |=s,\ |\gamma|\neq 0}\|\partial^{\gamma }u\,\partial^{\nu}v\|_{L^2(\Omega_T)} \\ & \leq C\left( \|u\|_{H^s(\Omega_T)}\|v\|_{L_{\infty}(\Omega_T)}
+\|u\|_{W^1_{\infty}(\Omega_T)}\|v\|_{H^{s-1}(\Omega_T)}\right)
\end{split}
\]
for $|\mu |\leq s$
($\partial^{\mu}=\partial_t^{\mu_0}\partial_1^{\mu_1}\partial_2^{\mu_2}\partial_3^{\mu_3}$, etc.). Note also that before the usage of \eqref{c41} we decompose the matrix functions $\widehat{\mathcal{A}}_l$ (with $l =\overline{0,3}$) as $\widehat{\mathcal{A}}_l(u)=\widehat{\mathcal{B}}_l(u)+\widehat{\mathcal{C}}_l$, where $\widehat{\mathcal{B}}_l(0)=0$ and  the matrix $\widehat{\mathcal{C}}_l$ is a constant matrix.

Using the boundary conditions \eqref{b50b.4}, we obtain
\begin{equation}
\begin{split}
-2   (&\partial^{\alpha}_{\rm tan}{p}\partial^{\alpha}_{\rm tan}{v}_N)|_{x_1=0}  \\  &
\begin{split}
=&\; 2\bigl\{\partial_1\hat{p}\,\partial^{\alpha}_{\rm tan}\varphi +[\partial^{\alpha}_{\rm tan},\partial_1\hat{p}]\varphi \bigr\}\big\{D_0(\hat{v})\partial^{\alpha}_{\rm tan}\varphi - \partial_1\hat{v}_N\partial^{\alpha}_{\rm tan}\varphi \\ & +[\partial^{\alpha}_{\rm tan},\hat{v}_2]\partial_2\varphi + [\partial^{\alpha}_{\rm tan},\hat{v}_3]\partial_3\varphi -[\partial^{\alpha}_{\rm tan},\partial_1\hat{v}_N]\varphi \big\}\big|_{x_1=0}
\end{split}
\\
&
\begin{split}
= &\; \Bigl(\partial_t\Bigl\{\underbrace{\partial_1\hat{p}\,(\partial^{\alpha}_{\rm tan}\varphi)^2}+2\partial^{\alpha}_{\rm tan}\varphi[\partial^{\alpha}_{\rm tan},\partial_1\hat{p}]\varphi\Bigr\} \\ & \quad +\underline{\partial_2\left\{\hat{v}_2\partial_1\hat{p}\,(\partial^{\alpha}_{\rm tan}\varphi)^2+2\hat{v}_2\partial^{\alpha}_{\rm tan}\varphi[\partial^{\alpha}_{\rm tan},\partial_1\hat{p}]\varphi\right\}}
 \\
 & \quad+\underline{\partial_3\left\{\hat{v}_3\partial_1\hat{p}\,(\partial^{\alpha}_{\rm tan}\varphi)^2+2\hat{v}_3\partial^{\alpha}_{\rm tan}\varphi[\partial^{\alpha}_{\rm tan},\partial_1\hat{p}]\varphi\right\}}\,\Bigr)\Big|_{x_1=0} -\mathfrak{R},
\end{split}
\end{split}
\label{b.form}
\end{equation}
where the underbraced term is the {\it  most important} one because under the Rayleigh-Taylor sign condition \eqref{42} it gives us the control on the $L^2$ norm of $\partial^{\alpha}_{\rm tan} \varphi$ (see below); the underlined terms  disappear after the integration over the domain $\partial\Omega_t$; $\mathfrak{R}$ is a sum of terms (lower-order terms, in some sense) which have the form ${\rm coeff}\partial_{\rm tan}^{\beta}\varphi\partial_{\rm tan}^{\gamma}\varphi$ with $|\beta|\leq s$, $|\gamma|\leq s$ and coeff being a coefficient depending on the derivatives of $\widehat{U}_{|x_1=0}$, $\partial_1\widehat{U}_{|x_1=0}$ and $\hat{\varphi}$ of order less than or equal to $s+1$. The explicit form of $\mathfrak{R}$ is of no importance. Using the Moser-type inequalities, we just estimate the integral of a typical term contained in $\mathfrak{R}$:
\[
\begin{split}
-2\int\limits_{\partial\Omega_t}&\partial^{\alpha}_{\rm tan}\varphi\,\partial_2\{ \hat{v}_2[\partial^{\alpha}_{\rm tan},\partial_1\hat{p}]\varphi\}|_{x_1=0}\,{\rm d}x'{\rm d}s\\
 & \leq  \|\varphi\|^2_{H^s(\partial\Omega_t)} +\|\partial_2\left\{ \hat{v}_2[\partial^{\alpha}_{\rm tan},\partial_1\hat{p}]\varphi\right\}|_{x_1=0}\|^2_{L^2(\partial\Omega_t)}\\
& \leq C(K)\left\{ \|\varphi\|^2_{H^s(\partial\Omega_t)}+\|\varphi\|^2_{L_{\infty}(\partial\Omega_T)}\big( 1
+\|\partial_1\widehat{U}_{|x_1=0}\|^2_{H^{s+1}(\partial\Omega_T)}\big)\right\}\\
& \leq
C(K)\left\{ \|\varphi\|^2_{H^s(\partial\Omega_t)}+\|\varphi\|^2_{L_{\infty}(\partial\Omega_T)}
\big( 1+\|\widehat{U}\|^2_{H^{s+3}(\Omega_T)}\big)\right\}.
\end{split}
\]
Estimating analogously the rest terms contained in $\mathfrak{R}$, we get
\begin{equation}
\label{frakR}
\mathfrak{R}\leq C(K)\left\{ \|\varphi\|^2_{H^s(\partial\Omega_t)}+\|\varphi\|^2_{W^1_{\infty}(\partial\Omega_T)}
( 1+\|{\rm coeff}\|^2_{s+3})\right\}.
\end{equation}

Using the Young inequality, the Rayleigh-Taylor sign condition \eqref{42} and again the Moser-type inequalities, from \eqref{39el}, \eqref{c42}, \eqref{b.form} and \eqref{frakR} we deduce
\begin{equation}
\label{n61}
\begin{split}
\int\limits_{\Omega}\widehat{\cal A}_0  \partial^{\alpha}_{\rm tan}\mathcal{U}\cdot\partial^{\alpha}_{\rm tan}\mathcal{U} & \, {\rm d}x+
\frac{\epsilon}{2}\int\limits_{\partial\Omega }  (\partial^{\alpha}_{\rm tan}\varphi)^2 \,{\rm d}x' \\ & \leq
\frac{2}{\epsilon}\int\limits_{\partial\Omega }  \big([\partial^{\alpha}_{\rm tan},\partial_1\hat{p}]\varphi\big)^2_{|x_1=0} \,{\rm d}x'+C(K)\mathcal{M}(t),\\
& = \frac{4}{\epsilon}\int\limits_{\partial\Omega_t }  [\partial^{\alpha}_{\rm tan},\partial_1\hat{p}]\varphi\,\partial_t\big([\partial^{\alpha}_{\rm tan},\partial_1\hat{p}]\varphi\big)\big|_{x_1=0} \,{\rm d}x'{\rm d}s \\ & \qquad \quad+C(K)\mathcal{M}(t) \\
& \leq C(K)\mathcal{M}(t),
\end{split}
\end{equation}
where
\[
\mathcal{ M}(t) =  \mathcal{ N}(T)+\int\limits_0^t\mathcal{ I}(\tau )\,{d}\tau,\qquad
\mathcal{ I}(t)=\nt \mathcal{U}(t)\nt^2_{H^s(\Omega )}+
\nt \varphi (t)\nt^2_{H^s(\partial\Omega )},
\]
\[
\begin{split}
\mathcal{ N}(T)  = &\|\mathfrak{f}\|^2_{H^s(\Omega_T)} \\ &
+\left(\|{U}\|^2_{W^1_{\infty}(\Omega_T)}+\|\varphi\|^2_{W^1_{\infty}(\partial\Omega_T)}+\|\mathfrak{f}\|^2_{L_{\infty}(\Omega_T)}\right)\left( 1+\|{\rm coeff}\|^2_{s+3}\right),
\end{split}
\]
with
\[
\nt u(t)\nt^2_{H^m(D)}:=
\sum\limits_{j=0}^m\|\partial_t^ju(t)\|^2_{H^{m-j}(D)}\qquad (D=\Omega\quad\mbox{or}\quad D=\partial\Omega ).
\]
The constant $C=C(K)$ in the last line of \eqref{n61} depends also on the fixed constant $\epsilon$.
Since only the biggest loss of derivatives from the coefficients will play the role for obtaining the final tame estimate, we have roughened the inequalities in \eqref{n61} by choosing the biggest loss. It follows from  \eqref{n61} that
\begin{equation}
\nt \mathcal{U}(t)\nt^2_{{\rm tan},s}+\nt \varphi (t)\|^2_{H^s(\partial\Omega)}\leq C(K){\cal M}(t),
\label{n63}
\end{equation}
where
\[
\nt u(t)\nt^2_{{\rm tan},m}:=\sum_{|\alpha|\leq m} \| \partial^{\alpha}_{\rm tan} u (t)\|^2_{L^2(\Omega )}.
\]

For ``closing'' the estimate in \eqref{n63} it remains to estimate  the derivatives of $\mathcal{U}$ containing normal ($x_1$-)derivatives.
We first estimate such derivatives of the ``noncharacteristic'' unknown $U_n=(p,v_N)$.
Applying to \eqref{n64} the operator $\partial_{\rm tan}^{\beta}$ with $|\beta |\leq s-1$,
using standard decompositions like
\[
\partial_{\rm tan}^{\beta}(B\partial_i \mathcal{U})=B\partial_{\rm tan}^{\beta}\partial_i \mathcal{U} +
[\partial_{\rm tan}^{\beta},B]\partial_i \mathcal{U},
\]
taking into account the fact that $\widehat{\mathfrak{A}}_1|_{x_1=0}=0$,  and employing counterparts of the calculus inequalities \eqref{c40}, \eqref{c41} and \eqref{c40'} for the ``layerwise'' norms $\nt (\cdot )(t)\nt$ (see \cite{Sch}) as well as the elementary inequality \eqref{elin}, we estimate the normal derivative of $\partial_{\rm tan}^{\beta}U_n$:
\begin{equation}
\begin{split}
\|\partial_1\partial_{\rm tan}^{\beta}U_n (t) & \|^2_{L^2(\Omega )} \\ \leq C(K)\Bigl\{ & \nt \mathcal{U}(t)\nt^2_{{\rm tan},s}
+\|\sigma \partial_1\partial_{\rm tan}^{\beta} \mathcal{U}(t)\|^2_{L^2(\Omega )} \\ &+ \nt \mathcal{U}(t) \nt^2_{H^{s-1}(\Omega )}
+ \nt \mathfrak{f} (t)\nt^2_{H^{s-1}(\Omega )} \\ &+
\left(\|{U}\|^2_{W^1_{\infty}(\Omega_T)}+\|\mathfrak{f}\|^2_{L_{\infty}(\Omega_T)}\right)\left( 1+\nt{\rm coeff} (t) \nt^2_{s}\right)\Bigr\}
\\ \leq C(K)\Bigl\{ & \nt \mathcal{U}(t)\nt^2_{{\rm tan},s}
+\|\sigma \partial_1\partial_{\rm tan}^{\beta} \mathcal{U}(t)\|^2_{L^2(\Omega )} +\mathcal{M}(t)\Bigr\} .
\end{split}
\label{n66}
\end{equation}

Since $\sigma |_{x_1=0}=0$, we do not need to use boundary conditions to estimate $\sigma\partial_1^j\partial_{\rm tan}^{\gamma}\mathcal{U}$, with $j+|\gamma |\leq s$,  and we easily get the inequality
\begin{equation}
\| \sigma\partial_1^j\partial_{\rm tan}^{\gamma}\mathcal{U}(t)\|^2_{L^2(\Omega )}\leq C(K)\mathcal{M}(t).
\label{n67}
\end{equation}
It follows from \eqref{n66} and \eqref{n67} for $j=1$ that
\[
\|\partial_1\partial_{\rm tan}^{\beta}U_n (t)  \|^2_{L^2(\Omega)}  \leq C(K)\left\{  \nt \mathcal{U}(t)\nt^2_{{\rm tan},s} +\mathcal{M}(t)\right\} .
\]
The last inequality implies
\begin{equation}
\sum_{i=1}^{k}\sum_{|\beta|\leq s-i}\|\partial_1^i\partial_{\rm tan}^{\beta}U_n (t)\|^2_{L^2(\Omega )}\leq C(K)\left\{  \nt \mathcal{U}(t)\nt^2_{{\rm tan},s} +\mathcal{M}(t)\right\}
\label{n68}
\end{equation}
with $k=1$. Estimate \eqref{n68} for $k=s$ is easily proved by finite induction. The combination of \eqref{n63} and \eqref{n68} for $k=s$ yields
\begin{equation}
\nt \mathcal{U}(t)\nt^2_{{\rm tan},s}+\|U_n (t)\|^2_{H^s(\Omega )} +\nt \varphi (t)\|^2_{H^s(\partial\Omega )}\leq C(K){\cal M}(t).
\label{n69}
\end{equation}

From equations \eqref{68}--\eqref{70} we estimate missing normal derivatives of the ``characteristic'' part $(v_2,v_3,F_1,F_2,F_3)$ of the unknown $U$:
\begin{equation}\label{omega'}
\|\omega (t)\|^2_{H^{s-1}(\Omega )}+ \sum_{j=1}^{3}\left\{\|\xi_j (t)\|^2_{H^{s-1}(\Omega )}+\|\eta_j (t)\|^2_{H^{s-1}(\Omega )}\right\}\leq C(K){\cal M}(t).
\end{equation}
The combination of \eqref{omega'} with \eqref{n69} implies the ``closed'' energy inequality
\[
\mathcal{ I}(t)\leq C(K)\mathcal{M}(t)=C(K)\biggl\{\mathcal{ N}(T)+\int\limits_0^t\mathcal{ I}(\tau )\,{d}\tau \biggr\}.
\]
Applying Gronwall's lemma, we obtain the energy a priori estimate
\begin{equation}
\mathcal{ I}(t)\leq C(K)\,e^{C(K)T}\mathcal{ N}(T).
\label{enaprest}
\end{equation}

Integrating \eqref{enaprest} over the interval $[0,T]$, we come to the estimate
\begin{equation}
\|\mathcal{U}\|^2_{H^s(\Omega_T )}+\|\varphi\|^2_{H^{s}(\partial\Omega_T)}\leq C(K)Te^{C(K)T}\mathcal{ N}(T).
\label{c51}
\end{equation}
Recall that ${U}=\widehat{J}\mathcal{U}$ (see \eqref{J}). Taking into account the decompositions $\widehat{J}(\hat{\varphi})= I_{13} +\widehat{J}_0(\hat{\varphi})$ and $\widehat{J}_0(0)=0$, using \eqref{c40} and \eqref{c41},
and the inequality
\[
\|u\|^2_{H^m([0,T]\times D)}\leq  T\|u\|^2_{H^{m+1}([0,T]\times D)}
\]
following from the integration of \eqref{elin} over the time interval $[0,T]$,
we obtain
\begin{equation}
\begin{split}
\|{U}\|^2_{H^s(\Omega_T )} & = \|\mathcal{U} + \widehat{J}_0\mathcal{U}\|^2_{H^s(\Omega_T )} \\
 & \leq C(K)\bigl(
\|\mathcal{U}\|^2_{H^s(\Omega_T )} +\|{U}\|^2_{L_{\infty}(\Omega_T)}\|{\rm coeff}\|_{s+1}^2\bigr) \\ &
\leq C(K)\|\mathcal{U}\|^2_{H^s(\Omega_T )}+
TC(K)\|{U}\|^2_{L_{\infty}(\Omega_T)}\|{\rm coeff}\|_{s+2}^2.
\end{split}
\label{c52}
\end{equation}
Inequalities \eqref{c51} and \eqref{c52} imply
\begin{equation}
\|{U}\|^2_{H^s(\Omega_T )}+\|\varphi\|^2_{H^{s}(\partial\Omega_T)}\leq C(K)Te^{C(K)T}\mathcal{ N}(T).
\label{c53}
\end{equation}

Applying Sobolev's embeddings, from (\ref{c53}) with $s\geq 3$, we get
\begin{equation}
\begin{split}
\|{U} & \|_{H^s(\Omega_T )}+\|\varphi\|_{H^{s}(\partial\Omega_T)} \\
&
\leq C(K)T^{1/2}e^{C(K)T}\Bigl\{ \|\mathfrak{f}\|_{H^s(\Omega_T)}
+\big( \|{U}\|_{H^3(\Omega_T)}+\|\varphi\|_{H^3(\partial\Omega_T)} \\ &\qquad\qquad\qquad\qquad\quad  +\|\mathfrak{f}\|_{H^3(\Omega_T)}\big)
\bigl(
\|\widehat{U}\|_{H^{s+3}(\Omega_T )}+\|\hat{\varphi}\|_{H^{s+3}(\partial\Omega_T)}\bigr)\Bigr\},
\end{split}
\label{c54}
\end{equation}
where we have absorbed some norms $\|{U}\|_{H^3(\Omega_T)}$ and $\| \varphi\|_{H^3(\partial\Omega_T)}$ in the left-hand side by choosing $T$ small enough. Considering \eqref{c54} for $s=3$ and using (\ref{37}), we obtain for $T$ small enough that
\begin{equation}
\|{U}\|_{H^3(\Omega_T )}+\|\varphi\|_{H^{3}(\partial\Omega_T)}\leq C(K_0)\|\mathfrak{f}\|_{H^3(\Omega_T)}.
\label{c55}
\end{equation}
It is natural to assume that $T<1$ and, hence, we can suppose that the constant $C(K_0)$ does not depend on $T$.
Inequalities \eqref{c54} and \eqref{c55} imply the tame estimate \eqref{38'}.

Using the Moser-type calculus inequalities, inequality \eqref{tildU'} and Sobolev's embeddings, from \eqref{a87''} we get the estimate
\begin{equation}
\begin{split}
\|\mathfrak{f}  \|_{H^s(\Omega_T)} \leq C(K_0)\big\{ &\|f\|_{H^s(\Omega_T)} +\|g\|_{H^{s+1}(\partial\Omega_T)} \\ & +
\| g\|_{H^{4}(\partial\Omega_T)} \bigl(
\|\widehat{U}\|_{H^{s+1}(\Omega_T )}+\|\hat{\varphi}\|_{H^{s+1}(\partial\Omega_T)}\bigr)\big\},
\end{split}
\label{fHs}
\end{equation}
for $\mathfrak{f}$ which together with \eqref{38'} and \eqref{a87'} (recall that the indices $^{\natural}$ were dropped) gives the desired tame a priori  estimate \eqref{38}.

Formally, the existence of solutions $(\dot{U},\varphi )\in H^1(\Omega_T)\times H^1(\partial\Omega_T)$ to problem \eqref{34}--\eqref{36} was proved in Section \ref{s4} (see Theorem \ref{t1}). However, we can here omit a formal proof of the existence of solutions having an arbitrary degree of smoothness and suppose that the existence result of Theorem \ref{t1} is also valid for the function spaces $H^s(\Omega_T)\times H^s(\partial\Omega_T)$ with $s\geq 1$ because the arguments in Section \ref{s4} towards the proof of existence are easily extended to these function spaces under the same assumptions about the regularity of the basic state $(\widehat{U},\hat{\varphi})$ as in Theorem \ref{t4.1}. The proof of Theorem \ref{t4.1} is thus complete.

\subsection{Proof of the tame estimate under the fulfilment of the non-collinearity condition}

In view of the second boundary condition in \eqref{b50b.4}, it follows from \eqref{39el} that
\begin{equation}
\begin{split}
\int\limits_{\Omega}\widehat{\cal A}_0 & \partial^{\alpha}_{\rm tan}\mathcal{U}\cdot\partial^{\alpha}_{\rm tan}\mathcal{U} \, {\rm d}x \\ & =- 2\int\limits_{\partial\Omega_t}\bigl\{\partial_1\hat{p}\,\partial^{\alpha}_{\rm tan}\varphi +[\partial^{\alpha}_{\rm tan},\partial_1\hat{p}]\varphi \bigr\}\,\partial^{\alpha}_{\rm tan}{v}_N|_{x_1=0}\,{\rm d}x'{\rm d}s+{\cal R},
\end{split}
\label{39el'}
\end{equation}
Using \eqref{74}, we reduce the term $-2\partial_1\hat{p}\,\partial^{\alpha}_{\rm tan}{\varphi}\,\partial^{\alpha}_{\rm tan}{v}_N|_{x_1=0}$ to the sum of ``lower-order'' terms having one of the following forms:
\begin{equation}
\begin{split}
\hat{c}\partial^{\gamma}_{\rm tan}F_{N}^j\partial^{\alpha}_{\rm tan}{v}_N,\quad \hat{c}\partial^{\gamma}_{\rm tan}R_j\partial^{\alpha}_{\rm tan}{v}_N,\quad
 \hat{c}\partial^{\gamma}_{\rm tan}v_N\partial^{\alpha}_{\rm tan}{v}_N, & \\
\hat{c}\partial^{\gamma}_{\rm tan}\varphi\,\partial^{\alpha}_{\rm tan}{v}_N,\quad
\partial^{\alpha}_{\rm tan}{v}_N[\partial^{\gamma}_{\rm tan},\hat{c}]F_{N}^j, &  \\
\partial^{\alpha}_{\rm tan}{v}_N[\partial^{\gamma}_{\rm tan},\hat{c}]R_j,\quad \partial^{\alpha}_{\rm tan}{v}_N[\partial^{\gamma}_{\rm tan},\hat{c}]v_N,\quad \partial^{\alpha}_{\rm tan}{v}_N[\partial^{\gamma}_{\rm tan},\hat{c}]\varphi&\qquad \mbox{on}\ \partial\Omega_T ,
\end{split}
\label{bterms'}
\end{equation}
where $\partial^{\alpha}_{\rm tan}=\partial_\ell\partial^{\gamma}_{\rm tan}$, $\ell=2$ or $\ell =3$ if $\alpha_0\neq s$ and $\ell=0$ (i.e., $\partial^{\alpha}_{\rm tan}=\partial_t\partial^{\gamma}_{\rm tan}$) otherwise, $|\gamma |=s-1\geq 2$, $j=\mu$ or $j=\nu$, and $\hat{c}$ is the common notation for a coefficient appearing in \eqref{74} as a component of $\pm\hat{a}_m$. The terms $\hat{c}\partial^{\gamma}_{\rm tan}F_{N}^j\partial^{\alpha}_{\rm tan}{v}_N|_{x_1=0}$ (with $j=\mu$ and $j=\nu$) are estimated by passing to the volume integral and integrating by parts:
\[
\begin{split}
\int\limits_{\partial\Omega_t} &\hat{c}\partial^{\gamma}_{\rm tan}F_{N}^j\partial^{\alpha}_{\rm tan}{v}_N|_{x^1=0}\,{\rm d}x'{\rm d}s=-\int\limits_{\Omega_t}\partial_1\bigl(\tilde{c}h_1\partial_\ell h_2\bigr){\rm d}x{\rm d}s \\
&=\int\limits_{\Omega_t}\Bigl\{\tilde{c}\partial_{\ell}h_1\partial_1h_2 +(\partial_{\ell}\tilde{c})h_1\partial_1h_2 -\tilde{c}\partial_1h_1\partial_{\ell}h_2-(\partial_1\tilde{c})h_1\partial_1h_2\Bigr\}{\rm d}x{\rm d}s \\
&\quad-\int\limits_{\Omega_t}\partial_{\ell}\bigl(\tilde{c}h_1\partial_1h_2\bigr){\rm d}x{\rm d}s\\
& \leq C(K)\|h\|^2_{H^1(\Omega_t )}-\int\limits_{\Omega_t}\partial_{\ell}\bigl(\tilde{c}h_1\partial_1h_2\bigr){\rm d}x{\rm d}s,
\end{split}
\]
where $\tilde{c}|_{x_1=0}=\hat{c}$, $h_1=\partial^{\gamma}_{\rm tan}F_{N}^j$, $h_2= \partial^{\gamma}_{\rm tan}v_N$ and $h=(h_1,h_2)$. If $\ell =2$ or $\ell =3$ the last integral above is equal to zero. But for $\ell =0$ we have:
\[
-\int\limits_{\Omega_t}\partial_s\bigl(\tilde{c}h_1\partial_1h_2\bigr){\rm d}x{\rm d}s=
-\int\limits_{\Omega}\tilde{c}h_1\partial_1h_2{\rm d}x{\rm d}s.
\]
Using the Young inequality, the elementary inequality \eqref{elin} and the calculus inequality \eqref{c40}, we estimate the last integral as follows:
\[
\begin{split}
-\int\limits_{\Omega}\tilde{c}h_1\partial_1h_2{\rm d}x{\rm d}s &\leq C(K)\left\{ \tilde{\varepsilon}\nt \mathcal{U}(t)\nt^2_{H^s(\Omega )}
+\frac{1}{\tilde{\varepsilon}}\|h_1 (t)\|^2_{L^2(\Omega )}\right\}\\ &\leq C(K)\left\{ \tilde{\varepsilon}\nt \mathcal{U}(t)\nt^2_{H^s(\Omega )}
+\frac{1}{\tilde{\varepsilon}}\|F_{N}^j\|^2_{H^s(\Omega_t )}\right\}\\
&\leq \tilde{\varepsilon}C(K)\nt \mathcal{U}(t)\nt^2_{H^s(\Omega )} \\ & \qquad +\frac{C(K)}{\tilde{\varepsilon}}\Big\{
\|\mathcal{U}\|_{H^{s}(\Omega_t)} +\|U\|^2_{L_{\infty}(\Omega_T)}\|\hat{\varphi}\|^2_{H^{s+1}(\partial\Omega_T)}\Big\},
\end{split}
\]
where $\tilde{\varepsilon} $ is a small positive constant.

The rest terms in \eqref{bterms'} as well as the term $-2 [\partial^{\alpha}_{\rm tan},\partial_1\hat{p}]\varphi \,\partial^{\alpha}_{\rm tan}{v}_N|_{x_1=0}$ appearing in the boundary integral in \eqref{39el'} are estimated in the same way as above. The functions $R_j$ and $\varphi$ appear in the volume integral as $\chi (x_1)R_j$ and $\Psi$ respectively, where $\chi (x_1)$ is the ``lifting'' function from \eqref{14}. We also use the following ``higher-order'' version  of estimate \eqref{60} for $R_j$:
\begin{equation}
\begin{split}
\|R_j\|_{H^s(\partial\Omega_t)}\leq & C(K)\Big\{  \|\mathfrak{f}\|_{H^{s+1}(\Omega_T)} +\|\varphi\|_{H^{s}(\partial\Omega_t)}\\ &
+\big(\|\mathfrak{f}\|^2_{L_{\infty}(\Omega_T)}+\|\varphi\|^2_{L_{\infty}(\partial\Omega_T)}\big)\left( 1+\nt{\rm coeff} \nt^2_{s+2}\right)\Big\}.
\end{split}
\label{60'}
\end{equation}
Omitting technical details, we finally obtain the following estimate for the boundary integral in \eqref{39el'}:
\[
\begin{split}
- 2\int\limits_{\partial\Omega_t}\bigl\{\partial_1\hat{p}\,\partial^{\alpha}_{\rm tan}\varphi  +[\partial^{\alpha}_{\rm tan},& \partial_1\hat{p}]\varphi \bigr\} \,\partial^{\alpha}_{\rm tan}{v}_N|_{x_1=0}\,{\rm d}x'{\rm d}s \\ &  \leq
C(K)\Big\{\tilde{\varepsilon}\nt \mathcal{U}(t)\nt^2_{H^s(\Omega )}+\frac{1}{\tilde{\varepsilon}}\widetilde{\mathcal{M}}(t)\Big\},
\end{split}
\]
where
\[
\begin{split}
\widetilde{\mathcal{ M}}(t) &=  \widetilde{\mathcal{ N}}(T)+\int\limits_0^t\mathcal{ I}(\tau )\,{d}\tau,\qquad
\mathcal{ I}(t)=\nt \mathcal{U}(t)\nt^2_{H^s(\Omega )}+
\nt \varphi (t)\nt^2_{H^s(\partial\Omega )},
\\
\widetilde{\mathcal{ N}}(T) & = \|\mathfrak{f}\|^2_{H^{s+1}(\Omega_T)} \\ & \qquad
+\left(\|{U}\|^2_{W^1_{\infty}(\Omega_T)}+\|\varphi\|^2_{W^1_{\infty}(\partial\Omega_T)}+\|\mathfrak{f}\|^2_{L_{\infty}(\Omega_T)}\right)\left( 1+\|{\rm coeff}\|^2_{s+2}\right).
\end{split}
\]
Using then \eqref{c42}, it follows from \eqref{39el'} that
\begin{equation}
\nt \mathcal{U}(t)\nt^2_{{\rm tan},s}\leq C(K)\Big\{\tilde{\varepsilon}\nt \mathcal{U}(t)\nt^2_{H^s(\Omega )}+\frac{1}{\tilde{\varepsilon}}\widetilde{\mathcal{M}}(t)\Big\}.
\label{n63'}
\end{equation}
In fact, estimates \eqref{n68} (with $k=s$) and \eqref{omega'} were roughened ones in the sense that we could replace the norms $\|{\rm coeff}\|_{s+3}$ appearing there in $\mathcal{M}(t)$ with the norms $\|{\rm coeff}\|_{s+2}$. The combination of \eqref{n63'} with the improved versions of estimates \eqref{n68} and \eqref{omega'} (where the norms $\|{\rm coeff}\|_{s+3}$ are replaced with norms $\|{\rm coeff}\|_{s+2}$) gives
\[
\nt \mathcal{U}(t)\nt^2_{H^s(\Omega )}\leq C(K)\Big\{\tilde{\varepsilon}\nt \mathcal{U}(t)\nt^2_{H^s(\Omega )}+\frac{1}{\tilde{\varepsilon}}\widetilde{\mathcal{M}}(t)\Big\}.
\]
We then absorb the term $\tilde{\varepsilon}C (K)\nt \mathcal{U}(t)\nt^2_{H^s(\Omega )}$ in the left-hand side of the last inequality by choosing $\tilde{\varepsilon}$ small enough:
\begin{equation}
\nt \mathcal{U}(t)\nt^2_{H^s(\Omega )}\leq C(K)\widetilde{\mathcal{M}}(t).
\label{80'}
\end{equation}

By utilizing the elementary inequality \eqref{elin} (with $u=\varphi$ and $D=\partial\Omega$), from \eqref{80'} we deduce
\begin{equation}
\nt \mathcal{U}(t)\nt^2_{H^s(\Omega )}+\nt \varphi (t)\nt^2_{H^{s-1}(\partial\Omega )}\leq C(K)\widetilde{\mathcal{M}}(t).
\label{81'}
\end{equation}
Applying the differential operator $\partial^{\gamma}_{\rm tan}$ with $|\gamma |=s-1$ to \eqref{74}, using the Moser-type inequalities, estimate \eqref{60'}, inequality \eqref{elin} and the trace theorem, we obtain the inequalities
\begin{equation}
\sum\limits_{|\alpha |=s}\|\partial_{\rm tan}^{\alpha}\varphi\|^2_{L^2(\partial\Omega_t)}\leq C(K)\bigg\{ \widetilde{\mathcal{ N}}(T)+\int\limits_0^t\widetilde{\mathcal{ I}}(\tau )\,{d}\tau\bigg\},
\label{82^}
\end{equation}
\begin{equation}
\sum\limits_{|\alpha |=s}\|\partial_{\rm tan}^{\alpha}\varphi (t)\|^2_{L^2(\partial\Omega )}\leq C(K)\left\{ \widetilde{\mathcal{ N}}(T)+\widetilde{\mathcal{ I}}(t )\right\},
\label{82"}
\end{equation}
with $\widetilde{\mathcal{ I}}(t)=\nt \mathcal{U}(t)\nt^2_{H^s(\Omega )}+ \nt \varphi (t)\nt^2_{H^{s-1}(\partial\Omega )}$.
Combination \eqref{81'} with \eqref{82^}  gives
\[
\widetilde{\mathcal{ I}}(t)\leq C(K)\bigg\{ \widetilde{\mathcal{ N}}(T)+\int\limits_0^t\widetilde{\mathcal{ I}}(\tau )\,{d}\tau\bigg\}.
\]
Applying then Gronwall's lemma, we get the energy a priori estimate
\begin{equation}
\widetilde{\mathcal{ I}}(t)\leq C(K)\,e^{C(K)T}\widetilde{\mathcal{ N}}(T).
\label{84^}
\end{equation}

Integrating \eqref{84^} over the interval $[0,T]$, we come to the estimate
\begin{equation}
\widetilde{\mathcal{ I}}(t)\leq C(K)Te^{C(K)T}\widetilde{\mathcal{ N}}(T).
\label{85^}
\end{equation}
From \eqref{82"} and \eqref{85^} we derive the inequality
\[
\sum\limits_{|\alpha |=s}\|\partial_{\rm tan}^{\alpha}\varphi (t)\|^2_{L^2(\partial\Omega )}\leq C(K) \widetilde{\mathcal{ N}}(T)
\]
for $T<1$ whose integration over the interval $[0,T]$ gives
\[
\sum\limits_{|\alpha |=s}\|\partial_{\rm tan}^{\alpha}\varphi\|^2_{L^2(\partial\Omega_t)}\leq C(K)T\widetilde{\mathcal{ N}}(T).
\]
The last estimate and \eqref{85^} imply
\begin{equation}
\mathcal{ I}(t)\leq C(K)T\widetilde{\mathcal{ N}}(T)
\label{84'}
\end{equation}
for $T<1$.

Applying then arguments similar to that in \eqref{c52}--\eqref{c54}, we get
\begin{equation}
\begin{split}
\|{U} & \|_{H^s(\Omega_T )}+\|\varphi\|_{H^{s}(\partial\Omega_T)} \\
&
\begin{split}
\leq C(K)T^{1/2}\Bigl\{ \|\mathfrak{f}\|_{H^{s+1}(\Omega_T)}
+&\left( \|{U}\|_{H^3(\Omega_T)}+\|\varphi\|_{H^3(\partial\Omega_T)} +\|\mathfrak{f}\|_{H^4(\Omega_T)}\right) \\ &\,\cdot
\bigl(
\|\widehat{U}\|_{H^{s+2}(\Omega_T )}+\|\hat{\varphi}\|_{H^{s+2}(\partial\Omega_T)}\bigr)\Bigr\},
\end{split}
\end{split}
\label{c54'}
\end{equation}
where we have absorbed some norms $\|{U}\|_{H^3(\Omega_T)}$ and $\| \varphi\|_{H^3(\partial\Omega_T)}$ in the left-hand side by choosing $T$ small enough. Considering \eqref{c54'} for $s=3$ and using \eqref{37'}, we obtain for $T$ small enough that
\begin{equation}
\|{U}\|_{H^3(\Omega_T )}+\|\varphi\|_{H^{3}(\partial\Omega_T)}\leq C(K_0)\|\mathfrak{f}\|_{H^4(\Omega_T)}.
\label{c55'}
\end{equation}
Inequalities \eqref{c54'} and \eqref{c55'} imply the tame estimate \eqref{38''}. Using \eqref{a87'}, \eqref{38''}  and the estimate \eqref{fHs} for $s$ replaced with $s+1$ gives the desired tame a priori  estimate \eqref{38"}. With reference to the comments about the existence of solutions from the last paragraph of the previous subsection, the proof of Theorem \ref{t4.2} is complete.

\section{Compatibility conditions and approximate solution}
\label{s4^}

To use the tame estimates \eqref{38'} or \eqref{38''} for the proof of convergence of the Nash-Moser iteration, we should reduce our nonlinear problem \eqref{11.1}--\eqref{13.1} on $[0,T]\times\Omega$ to that on $\Omega_T$ (see \eqref{OmegaT}) whose solutions vanish in the past. This is achieved by the classical argument suggesting to absorb the initial data into the interior equations by constructing a so-called \textit{approximate solution}. Before constructing the approximate solution we have to define \textit{compatibility conditions} for the initial data \eqref{13.1}.

Suppose we are given initial data $(U_0,\varphi_0)$ that satisfy the hyperbolicity conditions \eqref{11} and the divergence and boundary constraints \eqref{20} and \eqref{21}. Let
\[
U^{ (0)}=(p^{ (0)},F_1^{ (0)},F_2^{ (0)},F_3^{ (0)}):=U_0^{}\quad\mbox{and}\quad \varphi^{(0)}:=\varphi_0,
\]
where $v^{ (i )}:=\bigl(v_1^{ (i )},v_2^{ (i )},v_3^{ (i )}\bigr)$ and $F_j^{ (i )}:=\bigl(F_j^{ (i )},F_j^{ (i )},F_j^{ (i )}\bigr)$ ($j=1,2,3$). Here $i =0$ but below these notations will be used with indices $i\geq 0$. Let also $\Psi^{ (i )}:=\chi ( x_1)\varphi^{(i )}(t,x')$. In view of the hyperbolicity conditions \eqref{11}, we rewrite system \eqref{11.1} in the form
\begin{equation}
\partial_t U = -\left(A_0(U)\right)^{-1}\left( \widetilde{A}_1(U , {\Psi})\partial_1U +
A_2(U )\partial_2U + A_3(U )\partial_3U \right).
\label{56^}
\end{equation}
The traces
\[
U^{ (j)}=(p^{ (j)},v^{ (j)},F_1^{ (j)},F_2^{ (j)},F_3^{ (j)})=\partial_t^jU|_{t=0}\quad \mbox{and}\quad \varphi^{(j)}=\partial_t^j\varphi|_{t=0},
\]
with $j\geq 1$, are recursively defined by the formal application of the differential operator $\partial_t^{j-1}$
to the boundary condition
\begin{equation}
\partial_t\varphi =\left.\left(v_1-v_2\partial_2\varphi -v_3\partial_3\varphi \right)\right|_{x_1=0}
\label{57^}
\end{equation}
and \eqref{56^} and evaluating $\partial_t^j\varphi$ and $\partial_t^jU$ at $t=0$.

We naturally define the zeroth-order compatibility condition as $p^{(0)}_{|x_1=0}=0$. Evaluating \eqref{57^} at $t=0$, we get
\begin{equation}
\varphi^{(1)} =\bigl(v_1^{(0)}-v_2^{(0)}\partial_2\varphi^{(0)} -v_3^{(0)}\partial_3\varphi^{(0)} \bigr)\bigr|_{x_1=0},
\label{58^}
\end{equation}
and then, with $\partial_t\Psi^{\pm}|_{t=0}=\chi(\pm x_1)\varphi^{(1)}(x')$, from \eqref{56^} evaluated at $t=0$ we define $U^{\pm(1)}$. The first-order compatibility condition $p^{(1)}_{|x_1=0}=0$ depends on $\varphi^{(0)}$ and $\varphi^{(1)}$. Knowing $\varphi^{(1)}$ and $U^{\pm(1)}$ we can then find $\varphi^{(2)}$, $U^{\pm(2)}$, etc.
The following lemma is the analogue of lemma 4.2.1 in \cite{Met}, lemma 19 in \cite{ST} and lemma 4.1 in \cite{Tcpam}.

\begin{lemma}
Let $\mu\in\mathbb{N}$, $\mu \geq 3$, and $(U_0,\varphi_0)\in H^{\mu +1/2}(\Omega)\times H^{\mu +1/2}(\partial\Omega )$. Then, the procedure described above determines $U^{(j)}\in H^{\mu +1/2 -j}(\Omega )$ and $\varphi^{(j)}\in H^{\mu +1/2-j}(\partial\Omega)$ for $j= 1,\ldots ,\mu$. Moreover,
\begin{equation}
\sum_{j=1}^{\mu}\left( \bigl\|U^{(j)}\bigr\|_{H^{\mu +1/2-j}(\Omega )}+\bigl\|\varphi^{(j)}\bigr\|_{H^{\mu +1/2-j}(\partial\Omega)} \right)
\leq CM_0,
\label{59^}
\end{equation}
where
\begin{equation}
M_0=\|U_0\|_{H^{\mu +1/2}(\Omega)}+\|U_0\|_{H^{\mu +1/2}(\Omega)}+\|\varphi_0\|_{H^{\mu +1/2}(\partial\Omega)},
\label{60^}
\end{equation}
the constant $C>0$ depends only on $\mu$, $\|U_0\|_{W^{1}_{\infty}(\Omega)}$, and
$\|\varphi_0\|_{W^1_{\infty}(\partial\Omega)}$.
\label{l4.1}
\end{lemma}

The proof is almost evident and based on the multiplicative properties of Sobolev spaces.

\begin{definition}
Let $\mu\in\mathbb{N}$, $\mu \geq 3$. The initial data $(U_0,\varphi_0)\in H^{\mu +1/2}(\Omega)\times H^{\mu +1/2}(\partial\Omega )$ are said to be compatible up to order $\mu$ when $\big(U^{(j)}, \varphi^{(j)}\big)\in H^{\mu +1/2-j}(\Omega)\times H^{\mu +1/2-j}(\partial\Omega )$ satisfy
\begin{equation}
p^{(j)}_{|x_1=0}=0
\label{61^}
\end{equation}
for $j=0,\ldots ,\mu$.
\label{d1}
\end{definition}

We are ready to construct the approximate solution.

\begin{lemma}
Suppose the initial data \eqref{13.1} are compatible up to order $\mu$ and satisfy the hyperbolicity conditions \eqref{11} and the divergence constraints \eqref{20} for all $x\in\Omega$ and the boundary constraints \eqref{21} for all $x\in \partial\Omega$. Then there exists a vector-function $(U^{a},\varphi^a)\in H^{\mu +1}(\Omega_T)\times H^{\mu +1}(\partial\Omega_T)$
that is further called the approximate solution to problem \eqref{11.1}--\eqref{13.1} such that
\begin{equation}
\partial_t^j\mathbb{L}(U^{a},\Psi^{a} )|_{t=0}=0 \quad\mbox{in}\ \Omega\ \mbox{for}\ j=0,\ldots , \mu -1,
\label{62^}
\end{equation}
and it satisfies the boundary conditions \eqref{12.1} and the initial data \eqref{13.1}:
\begin{equation}
\mathbb{B}(U^{a},\varphi^a )=0\quad\mbox{on}\ \partial\Omega_T,\label{12.1"}
\end{equation}
\begin{equation}
U^{a}|_{t=0}=U_0\quad\mbox{in}\ \Omega, \qquad \varphi^a |_{t=0}=\varphi_0\quad \mbox{on}\ \partial\Omega,\label{13.1"}
\end{equation}
where ${\Psi}^{a} =\chi ( x_1)\varphi^a$. The approximate solution  obeys the estimate
\begin{equation}
\|U^{a}\|_{H^{\mu +1}(\Omega_T)}+\|\varphi^a\|_{H^{\mu +1}(\partial\Omega_T)}\leq C_1(M_0)
\label{63^}
\end{equation}
where $C_1=C_1(M_0)>0$ is a constant depending on $M_0$ (see \eqref{60^}). Moreover, the approximate solution satisfies the divergence and boundary constraints \eqref{20}  and \eqref{21} as well as equations \eqref{31}. It also satisfies the hyperbolicity conditions \eqref{11} in $[0,T]\times\Omega$. If the initial data satisfy the Rayleigh-Taylor sign condition \eqref{RT}, then the approximate solution satisfies \eqref{RT} on $[0,T]\times\partial\Omega$. If the initial data satisfy the non-collinearity condition \eqref{NC}, then the approximate solution satisfies \eqref{NC} on $[0,T]\times\partial\Omega$.
\label{l2}
\end{lemma}

We omit the proof of Lemma \ref{l2} because it is absolutely analogous to the proof of lemma 6 in \cite{T09}, lemma 4.3 in \cite{Tcpam}, lemma 21 in \cite{ST} and lemma 4.3 in \cite{MTTcont}.

\begin{remark}{\rm
We construct the approximate solution by lifting the trace $(U_0,\varphi_0)$ from the hypersurface $t=0$ to $\mathbb{R}\times\Omega$. This is why $(U_0,\varphi_0)\in H^{\mu +1/2}(\Omega)\times H^{\mu +1/2}(\partial\Omega )$ and $(U^{a},\varphi^a)\in H^{\mu +1}(\Omega_T)\times H^{\mu +1}(\partial\Omega_T)$. In this connection, we note that there was a technical mistake in \cite{Tcpam} in the definition of the compatibility conditions requiring that $\big(U^{(j)}, \varphi^{(j)}\big)\in H^{\mu -j}(\Omega)\times H^{\mu-j}(\partial\Omega )$ but not $\big(U^{(j)}, \varphi^{(j)}\big)\in H^{\mu +1/2-j}(\Omega)\times H^{\mu +1/2-j}(\partial\Omega )$.
\label{rAS}
}\end{remark}

Without loss of generality we can suppose that
\begin{equation}
\|U_0\|_{H^{\mu +1/2}(\Omega)}+\|\varphi_0\|_{H^{\mu +1/2}(\partial\Omega)}\leq 1,\quad
\|\varphi_0\|_{H^{\mu +1/2}(\partial\Omega )}\leq 1/2.
\label{65^}
\end{equation}
Then for a sufficiently short time interval $[0,T]$ the smooth solution whose existence we are going to prove satisfies
$\|\varphi\|_{L_{\infty}([0,T]\times\partial\Omega)}\leq 1$, which implies $\partial_1\Phi\geq 1/2$
(recall that $\|\chi'\|_{L_{\infty}(\mathbb{R})}<1/2$, see Section \ref{s2}). Let $\mu$ be an integer number that will appear in the regularity assumption for the initial data in the existence theorem for problem \eqref{11.1}--\eqref{13.1}. Looking ahead, we take $\mu=m+7$, with $m\geq 6$ for Theorem \ref{t01} and $m\geq 8$ for Theorem \ref{t02}. In the end of the next section  we will see that this choice is suitable. Taking into account \eqref{65^}, we rewrite \eqref{63^} as
\begin{equation}
\|U^{a}\|_{H^{m +8}(\Omega_T)}+\|\varphi^a\|_{H^{m +8}(\partial\Omega_T)}\leq C_*,
\label{66^}
\end{equation}
where $C_*=C_1(1)$. Let us introduce
\begin{equation}
f^{a}:=\left\{ \begin{array}{lr}
- \mathbb{L}(U^{a},{\Psi}^{a} ) & \quad \mbox{for}\ t>0,\\
0 & \ \mbox{for}\ t<0.\end{array}\right.
\label{67^}
\end{equation}
Since $(U^{a},\varphi^a)\in H^{m +8}(\Omega_T)\times H^{m +8}(\partial\Omega_T )$, taking into account \eqref{62^}, we get $f^{a} \in H^{m+7}(\Omega_T)$ and
\begin{equation}
\|f^{a}\|_{H^{m+7}(\Omega_T)}\leq \delta_0 (T),
\label{68^}
\end{equation}
where the constant $\delta_0(T)\rightarrow 0$ as $T\rightarrow 0$. To prove estimate \eqref{68^} we use the Moser-type and embedding inequalities and the fact that $f^{a}$ vanishes in the past.

Given the approximate solution defined in Lemma \ref{l2}, $(U ,\varphi)= (U^{a},\varphi^a)+ (\widetilde{U} ,\tilde{\varphi})$ is a solution of the original problem \eqref{11.1}--\eqref{13.1} on $[0,T]\times\Omega$ if $(\widetilde{U} ,\tilde{\varphi})$ satisfies the following problem on $\Omega_T$ (tildes are dropped):
\begin{align}
 \mathcal{ L}(U ,{\Psi})=f^{a} &\quad\mbox{in}\ \Omega_T, \label{69^}\\[3pt]
 \mathcal{ B}(U ,\varphi )=0 &\quad\mbox{on}\ \partial\Omega_T,\label{70^}
\\[3pt]
 (U,\varphi )=0 &\quad  \mbox{for}\ t<0,\label{71^}
\end{align}
where
\[
\mathcal{ L}(U ,{\Psi} ):=\mathbb{L}(U^{a} +U  ,{\Psi}^{a}+{\Psi} ) -
\mathbb{L}(U^{a} ,{\Psi}^{a}),\quad
 \mathcal{ B}(U,\varphi):=\mathbb{B}(U^{a} +U ,\varphi^a+\varphi ).
\]
From now on we concentrate on the proof of the existence of solutions to problem \eqref{69^}--\eqref{71^}.

\section{Nash-Moser iteration}
\label{s5^}

\setcounter{subsubsection}{0}

We solve problem \eqref{69^}--\eqref{71^} by a suitable Nash-Moser-type iteration scheme. In short, this scheme is a modified Newton's scheme, and at each Nash-Moser iteration step we smooth the coefficient $u_n$ of a corresponding linear problem for $\delta u_n =u_{n+1}-u_n$. Errors of a classical Nash-Moser iteration are the ``quadratic'' error of Newton's scheme and the ``substitution'' error caused by the application of smoothing operators $S_{\theta}$  (see, e.g., \cite{Al,Herm,Sec16} and references therein). Moreover, we have the additional error caused by the introduction of an intermediate (or modified) state $u_{n+1/2}$ satisfying some constraints. In our case, these constraints are  \eqref{29} and \eqref{30} and either the Rayleigh-Taylor sign condition \eqref{42} or the non-collinearity condition \eqref{40} together with assumption \eqref{31}, which were required to be fulfilled for the basic state \eqref{a21}. Also, the additional error is caused by dropping the zeroth-order term in $\Psi$ in the  linearized interior equations written in terms of the ``good unknown'' (see \eqref{32}).

Since assumption \eqref{37} on the basic state and the tame estimate \eqref{38} in Theorem \ref{t4.1} are the same as the corresponding assumption and tame estimate in theorem 3.1 in \cite{Tcpam}, the Nash-Moser procedure towards the proof of the existence result of Theorem \ref{t01} is absolutely the same as in \cite{Tcpam}.\footnote{The process of construction of the modified state $(U_{n+1/2},\varphi_{n+1/2})$ is similar to that in \cite{Tcpam} because the constraints for the deformation gradient in \eqref{30} are not more involved as the first assumption in \eqref{30}.} Referring to \cite{Tcpam} and taking into account Remark \ref{rAS}, we get the existence result of Theorem \ref{t01}.
The basic a priori estimates \eqref{41} and \eqref{43} for the linearized problem imply \textit{uniqueness} of a solution to the nonlinear problem \eqref{11.1}--\eqref{13.1} that can be proved by standard argument (see, e.g., \cite{Ticvs}). With this short remark, we shall no longer discuss the problem of uniqueness. That is, the proof of Theorem \ref{t01} is complete and it remains to prove the existence of smooth solutions to problem \eqref{11.1}--\eqref{13.1} formulated in Theorem \ref{t02}.

\setcounter{subsubsection}{0}

\subsection{Iteration scheme for solving problem (\ref{69^})--(\ref{71^})}

Regarding the proof of the existence result of Theorem \ref{t02}, we may be very brief here and almost everywhere just refer to \cite{Tcpam}. The only place which requires little attention is the construction of the modified state because the additional constraint \eqref{31} is more involved as those in  \cite{Tcpam}. Following \cite{Tcpam}, we describe the iteration scheme for problem \eqref{69^}--\eqref{71^}. We first list the important properties of smoothing operators \cite{Al,Herm,Sec16}.

\begin{proposition}
There exists such a family $\{S_{\theta}\}_{\theta\geq 1}$ of smoothing operators in $H^s(\Omega_T)$ acting on the class of functions vanishing in the past that
\begin{align}
  \|   S_{\theta}u\|_{H^{\beta}(\Omega_T) }\leq C\theta^{(\beta-\alpha )_+}\|u\|_{H^{\alpha}(\Omega_T) },\quad &    \alpha ,\beta \geq 0,
\label{72^}\\
  \|  S_{\theta}u-u\|_{H^{\beta}(\Omega_T) }\leq C\theta^{\beta-\alpha }\|u\|_{H^{\alpha}(\Omega_T) },\quad & 0 \leq \beta \leq \alpha ,  \label{73^}\\
  \Bigl\|  \frac{d}{d\theta}S_{\theta}u\Bigr\|_{H^{\beta}(\Omega_T) }\leq C\theta^{\beta-\alpha -1}\|u\|_{H^{\alpha}(\Omega_T) },\quad    & \alpha  ,\beta \geq 0, \label{74^}
\end{align}
where $C>0$ is a constant, and $(\beta-\alpha )_+:=\max (0,\beta -\alpha )$.  Moreover, there is another family of smoothing operators (still denoted
$S_{\theta}$) acting on functions defined on the boundary $\partial\Omega_T$ and meeting properties \eqref{72^}--\eqref{74^} with the norms $\|\cdot \|_{H^{\alpha}(\partial\Omega_T)}$.
\label{p1a}
\end{proposition}

We choose
\[
U_0 =0,\quad \varphi_0=0
\]
and assume that
\[
(U_k, \varphi_k)=(p_k,v_{1,k},v_{2,k},v_{3,k},F_{11,k},F_{21,k},\ldots ,F_{33,k},\varphi_k)
\]
are already given for $k=0,\ldots ,n$. Moreover, let $(U_k,\varphi_k)$ vanish in the past, i.e., they satisfy \eqref{71^}.  We define
\[
U_{n+1}=U_n+\delta U_n,\quad \varphi_{n+1}=\varphi_n+\delta \varphi_n,
\]
where the differences $\delta{U}_n$ and $\delta \varphi_n$ solve the linear problem
\begin{equation}
\left\{\begin{array}{lr}
\mathbb{L}'_e({U}^a +{U}_{n+1/2} ,{\Psi}^a+{\Psi}_{n+1/2})\delta\dot{{U}}_n={f}_n &\quad\mbox{in}\ \Omega_T, \\[6pt]
\mathbb{B}'_{n+1/2}(\delta\dot{{U}}_n,\delta {\varphi}_n)={g}_n  & \quad\mbox{on}\ \partial\Omega_T,\\[6pt]
(\delta\dot{{U}}_n,\delta \varphi_n)=0 &\quad \mbox{for}\ t<0.
\end{array}\right.
\label{75^}
\end{equation}
Here
\begin{equation}
\delta\dot{U}_n:= \delta{U}_n-\frac{\delta\Psi_n}{\partial_1(\Phi^a+\Psi_{n+1/2})}\,\partial_1({U}^a+{U}_{n+1/2})
\label{76^}
\end{equation}
is the ``good unknown'' (cf. \eqref{32}),
\[
\mathbb{B}'_{n+1/2}:=\mathbb{B}'_e(({U}^a +{U}_{n+1/2})|_{x_1=0} ,\varphi^a+\varphi_{n+1/2}),
\]
the operators $\mathbb{L}'_e$ and $\mathbb{B}'_e$ are defined in \eqref{L'e} and \eqref{B'e}, and $({U}_{n+1/2},\varphi_{n+1/2})$ is a smooth modified state such that $({U}^a +{U}_{n+1/2},\varphi^a+\varphi_{n+1/2})$ satisfies constraints  \eqref{29}--\eqref{31} and \eqref{40} ($\Psi_n$, ${\Psi}_{n+1/2}$, and $\delta\Psi_n$ are associated to $\varphi_n$, $\varphi_{n+1/2}$, and $\delta\varphi_n$ like ${\Psi}$ is associated to $\varphi$).
The right-hand sides ${f}_n$ and ${g}_n$  are defined through the accumulated errors at the step $n$.

The errors of the iteration scheme are defined from the following chains of decompositions:
\begin{align*}
& \mathcal{ L}({U}_{n+1} ,{\Psi}_{n+1})-\mathcal{ L}({U}_{n} ,{\Psi}_{n})\\
& \quad =\mathbb{L}'({U}^a +{U}_{n} ,{\Psi}^a+{\Psi}_{n})(\delta{{U}}_n,\delta{\Psi}_{n})+{e}'_n\\
& \quad =\mathbb{L}'({U}^a +S_{\theta_n}{U}_{n} ,{\Psi}^a+S_{\theta_n}{\Psi}_{n})(\delta{{U}}_n,\delta{\Psi}_{n})+{e}'_n+{e}''_n
\\
&\quad= \mathbb{L}'({U}^a +{U}_{n+1/2} ,{\Psi}^a+{\Psi}_{n+1/2})(\delta{{U}}_n,\delta{\Psi}_{n})+{e}'_n+{e}''_n+{e}'''_n\\
& \quad =\mathbb{L}'_e({U}^a +{U}_{n+1/2},{\Psi}^a+{\Psi}_{n+1/2})\delta\dot{{U}}_n+{e}'_n+{e}''_n+{e}'''_n+{D}_{n+1/2}\delta{\Psi}_{n}
\end{align*}
and
\begin{align*}
&\mathcal{ B}({U}_{n+1}|_{x_1=0},\varphi_{n+1})-\mathcal{ B}({U}_{n}|_{x_1=0},\varphi_{n})
\\ &\quad =\mathbb{B}'(({U}^a +{U}_{n})|_{x_1=0},\varphi^a+\varphi_{n})(\delta{{U}}_n|_{x_1=0},\delta \varphi_{n})+\tilde{e}'_n
\\ &\quad =\mathbb{B}'(({U}^a +S_{\theta_n}{U}_{n})|_{x_1=0},\varphi^a+S_{\theta_n}\varphi_{n})(\delta{{U}}_n|_{x_1=0},\delta \varphi_{n})+\tilde{e}'_n+\tilde{e}''_n
\\ &\quad =\mathbb{B}'_{n+1/2}(\delta\dot{{U}}_n,\delta\varphi_n)+\tilde{e}'_n+\tilde{e}''_n+\tilde{e}'''_n,
\end{align*}
where $S_{\theta_n}$ are smoothing operators enjoying the properties of Proposition \ref{p1a}, with the sequence $(\theta_n)$ defined by
\[
\theta_0\geq 1,\quad \theta_n=\sqrt{\theta_0+n} ,
\]
and we use the notation
\[
{D}_{n+1/2}:= \frac{1}{\partial_1(\Phi^a+\Psi_{n+1/2})}\,\partial_1\left\{ \mathbb{L}(U^a +U_{n+1/2} ,{\Psi}^a+{\Psi}_{n+1/2})\right\}.
\]
The errors ${e}'_n$ and  $\tilde{e}'_n$ are the usual quadratic errors of Newton's method, and ${e}''_n$,
$\tilde{ e}''_n$ and ${e}'''_n$, $\tilde{e}'''_n$ are the first and the second substitution errors
respectively.

Let
\begin{equation}
{e}_n:={e}'_n+{e}''_n+{e}'''_n+{D}_{n+1/2}\delta{\Psi}_{n}, \quad
\tilde{e}_n:= \tilde{e}'_n+\tilde{e}''_n+\tilde{e}'''_n,
\label{77^}
\end{equation}
then the accumulated errors at the step $n\geq 1$ are
\begin{equation}
{E}_n=\sum_{k=0}^{n-1}{e}_k,\quad \widetilde{E}_n=\sum_{k=0}^{n-1}\tilde{e}_k,
\label{78^}
\end{equation}
with ${E}_0:=0$ and $\widetilde{E}_0:=0$. The right-hand sides ${f}_n$ and ${g}_n$ are recursively computed from the equations
\begin{equation}
\sum_{k=0}^{n}{f}_k+S_{\theta_n}{E}_n=S_{\theta_n}{f}^a,\quad
\sum_{k=0}^{n}{g}_k+S_{\theta_n}\widetilde{E}_n=0,
\label{79^}
\end{equation}
where ${f}_0:=S_{\theta_0}{f}^a$ and ${g}_0:=0$. Since $S_{\theta_N}\rightarrow I$ as $N\rightarrow \infty$, one can show that we formally obtain
the solution to problem \eqref{69^}--\eqref{71^} from $\mathcal{ L}(U_{N} ,{\Psi}_{N})\rightarrow {f}^a$ and
$\mathcal{ B}(U_{N}|_{x_1=0},\varphi_{N})\rightarrow 0$, provided that $({e}_N,\tilde{e}_N)\rightarrow 0$.

Below we closely follow the plan of \cite{Tcpam}. Let us first formulate our inductive hypothesis, which is actually the same as in \cite{Tcpam}.

\subsection{Inductive hypothesis}

Given a small number $\delta >0$,\footnote{We use the same Greek letter $\delta $ as in the differences $\delta{U}_n$ and $\delta \varphi_n$ above. But we hope that this will not lead to confusion because from the context it is always clear that $\delta $ written before ${U}_n$ or $\varphi_n$ is not a multiplier.} the integer $\alpha :=m+1$, and an integer $\tilde{\alpha}$, our inductive hypothesis reads:
\[
(H_{n-1})\quad \left\{
\begin{array}{ll}
{\rm a})\;& \forall\, k=0,\ldots , n-1,\quad \forall s\in [3,\tilde{\alpha}]\cap\mathbb{N},\\[3pt]
 & \|\delta U_k\|_{H^s(\Omega_T)} +\|\delta \varphi_k\|_{H^s(\partial\Omega_T)}\leq \delta\theta_k^{s-\alpha -1}\Delta_k,\\[6pt]
{\rm b}) & \forall\, k=0,\ldots , n-1,\quad \forall s\in [3,\tilde{\alpha}-2]\cap\mathbb{N},\\[3pt]
 & \|\mathcal{ L}(U_k,{\Psi}_k)-f^a\|_{H^s(\Omega_T)}\leq 2\delta\theta_k^{s-\alpha -1},\\[6pt]
{\rm c}) & \forall\, k=0,\ldots , n-1,\quad \forall s\in [4,\alpha ]\cap\mathbb{N},\\[3pt]
 & \|\mathcal{ B}(U_k|_{x_1=0},\varphi_k)\|_{H^s(\partial\Omega_T)}\leq \delta\theta_k^{s-\alpha -1},
\end{array}\right.
\]
where $\Delta_k= \theta_{k+1}-\theta_k$. Note that the sequence $(\Delta_n)$ is decreasing and tends to zero, and
\[
\forall\, n\in\mathbb{N},\quad \frac{1}{3\theta_n}\leq\Delta_n=\sqrt{\theta_n^2+1} -\theta_n\leq \frac{1}{2\theta_n}.
\]
Recall that $(U_k,\varphi_k)$ for $k=0,\ldots ,n$ are also assumed to satisfy \eqref{71^}. Looking a few steps ahead, we observe that
we will need to use inequalities \eqref{66^} and \eqref{68^} with $m=\tilde{\alpha}-6$, i.e., we now choose $\tilde{\alpha}=m+6 =\alpha +5$.
Our goal is to prove that ($H_{n-1}$) implies ($H_n$) for a suitable choice of parameters $\theta_0\geq 1$ and $\delta >0$, and for a sufficiently short time $T>0$. After that we shall prove ($H_0$). From now on we assume that ($H_{n-1}$) holds.  As in \cite{CS2,Tcpam}, we have the following consequences of ($H_{n-1}$) and Proposition \ref{p1a}.

\begin{lemma}
If $\theta_0$ is big enough, then for every $k=0,\ldots ,n$ and for every integer $s\in [3,\tilde{\alpha}]$ we have
\begin{align}
& \|{U}_k \|_{H^s(\Omega_T)}+\| \varphi_k\|_{H^{s}(\partial\Omega_T)}\leq\delta\theta_k^{(s-\alpha )_+}, \quad \alpha\neq s,\label{80g}\\[3pt]
& \|{U}_k \|_{H^{\alpha}(\Omega_T)}+\| \varphi_k\|_{H^{\alpha}(\partial\Omega_T)}\leq\delta\log \theta_k,  \label{81^}\\[3pt]
& \|(I-S_{\theta_k}){U}_k \|_{H^s(\Omega_T)}+\|(1-S_{\theta_k}) \varphi_k\|_{H^{s}(\partial\Omega_T)}\leq C\delta\theta_k^{s-\alpha },  \label{82g}
\end{align}
and for every $k=0,\ldots ,n$ and for every integer $s\geq 3$ we have
\begin{align}
&\|S_{\theta_k}{U}_k \|_{H^s(\Omega_T)}+\| S_{\theta_k} \varphi_k\|_{H^{s}(\partial\Omega_T)}\leq C\delta\theta_k^{(s-\alpha )_+}, \quad \alpha\neq s,\label{83g}\\[3pt]
&\|S_{\theta_k}{U}_k \|_{H^{\alpha}(\Omega_T)}+\| S_{\theta_k}\varphi_k\|_{H^{\alpha}(\partial\Omega_T)}\leq C\delta\log \theta_k.   \label{84g}
\end{align}
\label{l3}
\end{lemma}

\subsection{Estimate of the quadratic and first substitution errors}

Referring to \cite{Tcpam} for the proof (see there lemmata 4.8 and 4.9), we have the following results:

\begin{lemma}
Let $\alpha \geq 4$. There exist $\delta >0$ sufficiently small, and $\theta_0 \geq 1$ sufficiently large, such that
for all $k=0,\ldots n-1$, and for all integer $s\in [3,\widetilde{\alpha}-1]$, we have the estimates
\begin{align}
 \|{e}'_k\|_{H^s(\Omega_T)}\leq & C\delta^2\theta_k^{L_1(s)-1}\Delta_k,\nonumber\\
 \|\tilde{e}'_k\|_{H^s(\partial\Omega_T)}\leq & C\delta^2\theta_k^{L_1(s)-1}\Delta_k,\nonumber
\end{align}
where $L_1(s)=\max \{ (s+1-\alpha )_+ +4-2\alpha ,s+2-2\alpha \}$.
\label{l4}
\end{lemma}

\begin{lemma}
Let $\alpha \geq 4$. There exist $\delta >0$ sufficiently small, and $\theta_0 \geq 1$ sufficiently large, such that
for all $k=0,\ldots n-1$, and for all integer $s\in [6,\widetilde{\alpha}-2]$, one has
\begin{align}
& \|{e}''_k\|_{H^s(\Omega_T)}\leq C\delta^2\theta_k^{L_2(s)-1}\Delta_k,\nonumber\\
& \|\tilde{e}''_k\|_{H^s(\partial\Omega_T)}\leq C\delta^2\theta_k^{L_2(s)-1}\Delta_k,\nonumber
\end{align}
where $L_2(s)=\max \{ (s+1-\alpha )_+ +6-2\alpha ,s+5-2\alpha \}$.
\label{l5}
\end{lemma}

\subsection{Construction and estimate of the modified state}

Since the approximate solution satisfies the hyperbolicity conditions \eqref{11} and the non-collinearity condition \eqref{NC} on $[0,T]\times\partial\Omega$ (see Lemma \ref{l2}) and since we shall require that the smooth modified state vanishes in the past, the state $(U^a+U_{n+1/2},\varphi^a+\varphi_{n+1/2})$ will satisfy \eqref{29} and \eqref{40}  for a sufficiently short time $T>0$. Therefore, while constructing the modified state we may focus only on constraints \eqref{30} and  \eqref{31}.

\begin{proposition}
Let $\alpha \geq 4$. The exist some functions $U_{n+1/2}$ and $\varphi_{n+1/2}$, that vanish in the past, and such that
$(U^a+U_{n+1/2},\varphi^a+\varphi_{n+1/2})$ satisfies \eqref{29}--\eqref{31} and \eqref{40} for a sufficiently short time $T$. Moreover, these functions satisfy
\begin{equation}
\varphi_{n+1/2}=S_{\theta_n}\varphi_n, \quad p_{n+1/2}=S_{\theta_n}p_n,\quad v_{j,n+1/2}=S_{\theta_n}v_{j,n}\ (j=2,3),
 \label{93^}
\end{equation}
and
\begin{equation}
\|U_{n+1/2} - S_{\theta_n}U_n\|_{H^s(\Omega_T)}\leq C\delta\theta_n^{s+2-\alpha}\quad \mbox{for}\ s\in [3,\tilde{\alpha}+2]
\label{94^}
\end{equation}
for sufficiently small $\delta>0$ and $T>0$, and a sufficiently large $\theta_0\geq 1$.
\label{p3}
\end{proposition}

\begin{proof}
Note that estimate \eqref{94^} can be proved for every $s\geq 3$ but below we will need it only for $s\in [3,\tilde{\alpha}+2]$.
Let $\varphi_{n+1/2}$, $p_{n+1/2}$  and $v_{j,n+1/2}$ ($j=2,3$) be defined by \eqref{93^}. We define $v_{1,n+1/2}$ as in \cite{Tcpam}:
\[
v_{1,n+1/2}:= S_{\theta_n}v_{1,n} +{\cal R}_T{\cal G},
\]
where
\[
\begin{split}
{\cal G}= & \partial_t \varphi_{n+1/2} -(S_{\theta_n}v_{1,n})|_{x_1=0}\\
 & +\sum_{k=2}^3\bigl(
(v_k^{a}+v_{k,n+1/2} )\partial_k\varphi_{n+1/2} + v_{k,n+1/2}\partial_k\varphi^a\bigr)\bigr|_{x_1=0}
\end{split}
\]
and ${\cal R}_T:\; H^s(\partial\Omega_T)\longrightarrow H^{s+1/2}(\Omega_T)$ is the lifting operator from the boundary to the interior. Since we mainly prefer to work with integer indices of Sobolev spaces (see Remark \ref{r4}), instead of ${\cal R}_T$ we could even write the function $\chi =\chi (x_1)$ being the same $C_0^{\infty}$ function which was used in \eqref{4}. Clearly, the above definition of $v_{1,n+1/2}$ implies the first boundary condition in \eqref{30} written for $(U^a+U_{n+1/2},\varphi^a+\varphi_{n+1/2})$. Referring to \cite{Tcpam} for detailed arguments and technical calculations, we have the estimate
\[
\|v_{1,n+1/2} - S_{\theta_n}v_{1,n}\|_{H^s(\Omega_T)} \leq C\delta\theta_n^{s+1-\alpha}.
\]

We now define
\[
F_{kj,n+1/2}:= S_{\theta_n}F_{kj,n} +\chi \left(b_{kj,n} -(S_{\theta_n}F_{kj,n})|_{x_1=0}\right)\quad (k=2,3),
\]
where for given $v_{n+1/2}$ the functions $b_{kj,n}$ satisfy the linear systems
\[
\mathcal{K}(v^a+v_{n+1/2})(F_{{\rm tan}_j}^a+b_{j,n})=0 \quad \mbox{on}\ \partial\Omega_T
\]
with the initial data $\left.\left(b_{kj,n} -(S_{\theta_n}F_{kj,n})|_{x_1=0}\right)\right|_{t=0}=0$. Here the linear diffe\-rential operator
\[
\begin{split}
\mathcal{K}(v^a+v_{n+1/2}):= & I_2\partial_t +\sum_{k=2}^3(v_k^a+v_{k,n+1/2})I_2\partial_k \\ & -
\begin{pmatrix}
\partial_2(v_2^a+v_{2,n+1/2}) & \partial_3(v_2^a+v_{2,n+1/2}) \\
\partial_2(v_3^a+v_{3,n+1/2}) & \partial_3(v_3^a+v_{3,n+1/2})
\end{pmatrix}\cdot \,,
\end{split}
\]
and $b_{j,n}=(b_{2j,n},b_{3j,n})$ and $F_{{\rm tan}_j}^a =(F_{2j}^a,F_{3j}^a)$. We can also rewrite the linear systems above by taking into account that the approximate state satisfies equations \eqref{31}, i.e.,
\[
\mathcal{K}(v^a)F_{{\rm tan}_j}^a=0 \quad \mbox{on}\ \partial\Omega_T.
\]
After that we define the first components of the vectors $F_{j,n+1/2}$ as
\[
F_{1j,n+1/2}:= S_{\theta_n}F_{1j,n} +\chi {\cal G}_j,
\]
where
\[
\begin{split}
{\cal G}_j= & -(S_{\theta_n}F_{1j,n})|_{x_1=0} \\ &
 +\sum_{k=2}^3\bigl(
(F_{kj}^{a}+F_{kj,n+1/2} )\partial_k\varphi_{n+1/2} + F_{kj,n+1/2}\partial_k\varphi^a\bigr)\bigr|_{x_1=0}.
\end{split}
\]
Clearly, the components of the vectors $F_{j,n+1/2}$ defined above $(U^a+U_{n+1/2},\varphi^a+\varphi_{n+1/2})$ satisfy \eqref{30} and \eqref{31}.

Omitting technical calculations, we only note that for estimating $b_{kj,n} -(S_{\theta_n}F_{kj,n})|_{x_1=0}$ we have to estimate the norms
$
\|(1-S_{\theta_n})F_{kj,n}\|_{H^{s+1}(\partial\Omega_T)}.
$
That is, we have to use the trace theorem whose application together with inequa\-lity \eqref{82g} gives
\[
\|(1-S_{\theta_n})F_{kj,n}\|_{H^{s+1}(\partial\Omega_T)}\leq C\|(1-S_{\theta_n})F_{kj,n}\|_{H^{s+2}(\Omega_T)}\leq C\delta\theta_n^{s+2-\alpha }.
\]
Omitting details, we  get the estimates
\[
\|F_{kj,n+1/2}- S_{\theta_n}F_{kj,n}\|_{H^{s}(\Omega_T)}\leq C\delta\theta_n^{s+2-\alpha } \quad (k=2,3).
\]
After that the estimates for $F_{1j,n+1/2}$ are obtained similarly to the estimates for $v_{1,n+1/2}$. Roughening the estimates for $v_{1,n+1/2}$, we finally get the desired estimate \eqref{94^}. The arguments towards the proof of estimate \eqref{94^} are really similar to those in \cite{ST,T09} in spite of the fact that our process of the definition of the modified state differs a little from that in \cite{ST,T09} (see also Remark \ref{r_ms} below).
\end{proof}

\begin{remark}{\rm
Our assumption \eqref{31} for the basic state requiring that the linear equations for $F_{2j}$ and $F_{3j}$ (for a given ${v}$) contained in \eqref{11.1} hold only on the boundary $x_1=0$ but not in the interior of the domain does not give any advantage in comparison with the assumption for the magnetic field in \cite{ST,T09}, where the corresponding linear equation for the magnetic field was assumed to be satisfied in the interior. Indeed, as in \cite{ST,T09}, in the right-hand side of estimate \eqref{94^} we have the multiplier $\theta_n^{s+2-\alpha }$ but not $\theta_n^{s+1-\alpha }$. That is, alternatively we could assume that the basic state satisfies the linear equations for $F_{j}$  in the interior and just refer to \cite{ST,T09} where the process of getting estimates like \eqref{94^} is described in more details.
\label{r_ms}
}\end{remark}

\subsection{Estimate of the second substitution errors and the last error term}

Referring again for a detailed proof to \cite{Tcpam} (see there lemmata 4.11 and 4.12), here we just formulate the following results.

\begin{lemma}
Let $\alpha \geq 5$. There exist $\delta >0$, $T>0$ sufficiently small, and $\theta_0 \geq 1$ sufficiently large, such that for all $k=0,\ldots n-1$, and for all integer $s\in [3,\widetilde{\alpha}-1]$, one has
\begin{align}
& \|{e}'''_k\|_{H^s(\Omega_T)}\leq C\delta^2\theta_k^{L_3(s)-1}\Delta_k,\nonumber \\
& \|\tilde{e}'''_k\|_{H^s(\partial\Omega_T)}\leq C\delta^2\theta_k^{L_3(s)-1}\Delta_k,\nonumber
\end{align}
where $L_3(s)=\max \{ (s+1-\alpha )_+ +10-2\alpha ,s+6-2\alpha \}$.
\label{l6}
\end{lemma}

\begin{lemma}
Let $\alpha \geq 5$. There exist $\delta >0$, $T>0$ sufficiently small, and $\theta_0 \geq 1$ sufficiently large, such that for all $k=0,\ldots n-1$, and for all integer $s\in [3,\widetilde{\alpha}-2]$, one has
\[
\|{D}_{k+1/2}\delta{\Psi}_k\|_{H^s(\Omega_T)}\leq C\delta^2\theta_k^{L(s)-1}\Delta_k,
\]
where $L(s)=\max \{ (s+2-\alpha )_+ +10-2\alpha ,(s+1-\alpha)_++11-2\alpha ,s+7-2\alpha \}$.
\label{l7}
\end{lemma}

The functions $L_3(s)$ and $L(s)$ differ from those in \cite{Tcpam} because in the right-hand side of estimate \eqref{94^} we have the multiplier $\theta_n^{s+2-\alpha }$ but not $\theta_n^{s+1-\alpha }$ as in \cite{Tcpam}. Going inside technical calculations in \cite{Tcpam} (see also \cite{CS2}), we can easily modify them by using estimate \eqref{94^} and get the formulae for $L_3(s)$ and $L(s)$.

\subsection{Convergence of the iteration scheme}

Lemmata \ref{l4}--\ref{l7} yield the estimate of ${e}_n$ and $\tilde{e}_n$ defined in \eqref{77^} as the sum of all the errors of the $n$th step.

\begin{lemma}
Let $\alpha \geq 5$. There exist $\delta >0$, $T>0$ sufficiently small, and $\theta_0 \geq 1$ sufficiently large, such that
for all $k=0,\ldots n-1$, and for all integer $s\in [3,\widetilde{\alpha}-2]$, one has
\begin{equation}
\|{e}_k\|_{H^s(\Omega_T)}+\|\tilde{e}_k\|_{H^s(\partial\Omega_T)}\leq C\delta^2\theta_k^{L(s)-1}\Delta_k,
\label{103}
\end{equation}
where $L(s)$ is defined in Lemma \ref{l7}.
\label{l8}
\end{lemma}

Lemma \ref{l8} gives the estimate of the accumulated errors ${E}_n$ and $\widetilde{E}_n$.

\begin{lemma}
Let $\alpha \geq 9$.  There exist $\delta >0$, $T>0$ sufficiently small, and $\theta_0 \geq 1$ sufficiently large, such that
\begin{equation}
\|{E}_n\|_{H^{\alpha +3}(\Omega_T)}+\|\widetilde{E}_n\|_{H^{\alpha +3}(\partial\Omega_T)}\leq C\delta^2\theta_n.
\label{104}
\end{equation}
\label{l9}
\end{lemma}

\begin{proof} One can check that $L(\alpha +3)\leq 1$ if $\alpha \geq 9$, where $L(s)$ is defined in Lemma \ref{l7}. It follows from \eqref{103} that
\[
\begin{split}
\|{E}_n\|_{H^{\alpha +3}(\Omega_T)}  +\|\widetilde{E}_n\|_{H^{\alpha +3}(\partial\Omega_T)} &
\leq \sum_{k=0}^{n-1}\left(
\|{e}_k\|_{H^{\alpha +3}(\Omega_T)}+\|\tilde{e}_k\|_{H^{\alpha +3}(\partial\Omega_T)}\right)\\ & \leq \sum_{k=0}^{n-1}C\delta^2\Delta_k\leq C\delta^2\theta_n
\end{split}
\]
for $\alpha \geq 9$ and $\alpha +3\in [3,\tilde{\alpha}-2]$, i.e., $\tilde{\alpha}\geq \alpha +5$. The minimal possible
$\tilde{\alpha}$ is $\alpha +5$, i.e., our choice $\tilde{\alpha}= \alpha +5$ is suitable.
\end{proof}

We can now derive the estimates of the source terms ${f}_n$ and ${g}_n$ defined in \eqref{79^}.

\begin{lemma}
Let $\alpha \geq 9$.  There exist $\delta >0$, $T>0$ sufficiently small, and $\theta_0 \geq 1$ sufficiently large, such that for all integer $s\in [3,\widetilde{\alpha}+2]$, one has
\begin{align}
& \|{f}_n\|_{H^{s}(\Omega_T)}\leq  C\Delta_n\bigl\{ \theta_n^{s-\alpha -3}\left( \|{f}^a\|_{H^{\alpha +2}(\Omega_T)}+\delta^2\right)+\delta^2\theta_n^{L(s)-1} \bigr\},\nonumber\\[6pt]
 & \|{g}_n\|_{H^{s}(\partial\Omega_T)}\leq  C\delta^2\Delta_n\bigl( \theta_n^{L(s)-1}+\theta_n^{s-\alpha -3}\bigr).
\nonumber
\end{align}
\label{l10}
\end{lemma}

The proof of Lemma \ref{l10} (with the help of \eqref{72^}, \eqref{74^}, \eqref{103} and \eqref{104}) is absolutely analogous to the proof of lemma 4.16 in \cite{Tcpam}. We are now in a position to obtain the estimate of the solution to problem \eqref{75^} by employing the tame estimate \eqref{38''}. Then the estimate of $(\delta U_n,\delta\varphi_n)$ follows from formula \eqref{76^}.

\begin{lemma}
Let $\alpha \geq 9$.  There exist $\delta >0$, $T>0$ sufficiently small, and $\theta_0 \geq 1$ sufficiently large, such that for all integer $s\in [3,\widetilde{\alpha}]$, one has
\begin{equation}
\|\delta U_n\|_{H^{s}(\Omega_T)}+\|\delta \varphi_n\|_{H^{s}(\partial\Omega_T)}\leq  \delta\theta_n^{s-\alpha -1}\Delta_n.
\label{107}
\end{equation}
\label{l11}
\end{lemma}

\begin{proof}
Without loss of generality we can take the constant $K_0$ appearing in estimate \eqref{38''} that $K_0=2C_*$, where
$C_*$ is the constant from \eqref{66^}. In order to apply Theorem \ref{t4.2}, by using \eqref{83g} and (\ref{94^}), we check that
\[
\|{U}^a+{U}_{n+1/2}\|_{H^{5}(\Omega_T)}+\|\varphi^a+S_{\theta_n}\varphi_n\|_{H^{5}(\partial\Omega_T)}\leq 2C_*
\]
for $\alpha \geq 9$ and $\delta$ small enough. That is, assumption \eqref{37'} is satisfied for the coefficients of problem (\ref{75^}). By applying the tame estimate (\ref{38}), for $T$ small enough one has
\begin{equation}
\begin{split}
\|\delta & \dot{{U}}_n\|_{H^s(\Omega_T)}+\|\delta \varphi_n\|_{H^{s}(\partial\Omega_T)} \\ &
\begin{split}
\leq C\bigl\{ &
\|{f}_n\|_{H^{s+1}(\Omega_T)}+ \|{g}_n\|_{H^{s+2}(\partial\Omega_T)} \\ &
+\bigl( \|{f}_n\|_{H^{4}(\Omega_T)}+ \|{g}_n\|_{H^{5}(\partial\Omega_T)} \bigr)\\ &\,\quad\cdot\bigl(
\|{U}^a +{U}_{n+1/2}\|_{H^{s+2}(\Omega_T)}+\|\varphi^a +S_{\theta_n}\varphi_n\|_{H^{s+2}(\partial\Omega_T)}\bigr)\bigr\}.
\end{split}
\end{split}
\label{108}
\end{equation}

Using Moser-type inequalities, from formula (\ref{76^}) we obtain
\[
\begin{split}
\|\delta{U}_n\|_{H^s(\Omega_T)}\leq &\, \|\delta\dot{{U}}_n\|_{H^s(\Omega_T)} +C\bigl\{\|\delta \varphi_n\|_{H^{s}(\partial\Omega_T)}
\\ &
+ \|\delta \varphi_n\|_{H^{3}(\partial\Omega_T)}\|\varphi^a +S_{\theta_n}\varphi_n\|_{H^{s}(\partial\Omega_T)}
\bigr\}.
\end{split}
\]
Then (\ref{108}) yields
\begin{equation}
\begin{split}
\|\delta & {{U}}_n\|_{H^s(\Omega_T)}+\|\delta \varphi_n\|_{H^{s}(\partial\Omega_T)} \\ & \begin{split} \leq C\bigl\{ &
\|{f}_n\|_{H^{s+1}(\Omega_T)}+ \|{g}_n\|_{H^{s+2}(\partial\Omega_T)} \\ &
 +\bigl( \|{f}_n\|_{H^{4}(\Omega_T)}+ \|{g}_n\|_{H^{5}(\partial\Omega_T)} \bigr)\\ &\,\quad\cdot\bigl(
\|{U}^a +{U}_{n+1/2}\|_{H^{s+2}(\Omega_T)}+\|\varphi^a +S_{\theta_n}\varphi_n\|_{H^{s+2}(\partial\Omega_T)}\bigr)\bigr\}
\end{split}
\end{split}
\label{109}
\end{equation}
for all integer $s\in [3,\widetilde{\alpha}]$.

Applying Lemma \ref{l11}, (\ref{83g}), and Proposition \ref{p3}, from (\ref{109}) we derive the estimate
\begin{equation}
\begin{split}
\|\delta & {{U}}_n\|_{H^s(\Omega_T)}+\|\delta \varphi_n\|_{H^{s}(\partial\Omega_T)} \\ & \begin{split}\leq & \,C
\bigl\{ \theta_n^{s-\alpha -1}\left( \|{f}^a\|_{H^{\alpha+2}(\Omega_T)}+\delta^2\right)+\delta^2\theta_n^{L(s+2)-1} \bigr\}\Delta_n\\
&  +C\delta\Delta_n\bigr\{\theta_n^{2-\alpha}\left( \|{f}^a
\|_{H^{\alpha+2}(\Omega_T)}
+\delta^2\right)
+\delta^2\theta_n^{11-2\alpha}\bigr\}\\ & \quad\cdot\bigl\{ C_*+\theta_n^{(s+2-\alpha )_+}+\theta_n^{s+4-\alpha}\bigr\}.
\end{split}
\end{split}
\label{110}
\end{equation}
We can now check that the inequalities
\begin{equation}
\left\{
\begin{array}{l}
 L(s+2)\leq s-\alpha,\quad (s +2-\alpha)_+ +2-\alpha \leq s-\alpha -1,\\[6pt]
 (s +2-\alpha)_+ +11-2\alpha \leq s-\alpha -1,\\[6pt]
 s+6-2\alpha \leq s-\alpha -1,\quad s+15-3\alpha \leq s-\alpha -1
\end{array}
\right.
\label{111}
\end{equation}
hold for $\alpha\geq 9$ and $s\in [3,\tilde{\alpha}]$. That is, it follows from (\ref{110}) and (\ref{68^}) that
\[
\|\delta{{U}}_n\|_{H^s(\Omega_T)}+\|\delta \varphi_n\|_{H^{s}(\partial\Omega_T)}\leq C\left( \delta_0(T)+\delta^2\right)
\theta_n^{s-\alpha -1}\Delta_n \leq \delta\theta_n^{s-\alpha -1}\Delta_n
\]
for $\delta$ and $T$ small enough.
\end{proof}

Inequality (\ref{107}) is point a) of ($H_n$). The lemma below gives us points b) and c) of ($H_n$).

\begin{lemma}
Let $\alpha \geq 9$.  There exist $\delta >0$, $T>0$ sufficiently small, and $\theta_0 \geq 1$ sufficiently large, such that for all integer $s\in [3,\widetilde{\alpha}-2]$
\begin{equation}
\|\mathcal{ L}({U}_n,{\Psi}_n)-{f}^a\|_{H^s(\Omega_T)}\leq 2\delta\theta_n^{s-\alpha -1}.
\label{193}
\end{equation}
Moreover, for all integer $s\in [4,{\alpha}]$ one has
\begin{equation}
\|\mathcal{ B}({U}_n|_{x_1=0},\varphi _n)\|_{H^s(\partial\Omega_T)}\leq \delta\theta_n^{s-\alpha -1}.
\label{194}
\end{equation}
\label{l12}
\end{lemma}

Again, we can just refer to \cite{Tcpam} for the proof of Lemma \ref{l12} (see the proof of lemma 4.19 in \cite{Tcpam}). As follows from Lemmata \ref{l11} and \ref{l12}, we have proved that $(H_{n-1})$ implies $(H_{n})$, provided that
$\alpha \geq 9$, $\tilde{\alpha}=\alpha +5$, the constant $\theta_0\geq 1$ is large enough, and $T>0$, $\delta >0$ are small enough. Fixing now the constants $\alpha$, $\delta$, and $\theta_0$, exactly as in \cite{Tcpam} (see there the proof of lemma 4.20), we can prove that $(H_0)$ is true. That is, we have

\begin{lemma}
If the time $T>0$ is sufficiently small, then $(H_0)$ is true.
\label{l13}
\end{lemma}

\subsection{The proof of Theorem \ref{t02}}

We consider initial data $({U}_0,\varphi_0)\in H^{m+15/2}(\Omega)\times H^{m+15/2}(\partial\Omega)$ satisfying all the assumptions of Theorem \ref{t02}. In particular, they satisfy the compatibility conditions up to order $\mu=m+7$ (see Definition \ref{d1}). Then, thanks to Lemmata \ref{l4.1} and \ref{l2} we can construct an approximate solution $({U}^{a},\varphi^a)\in H^{m +8}(\Omega_T)\times H^{m +8}(\partial\Omega_T)$ that satisfies (\ref{66^}). As follows from Lemmata \ref{l11}--\ref{l13}, $(H_n)$ holds for all integer $n\geq 0$, provided that $\alpha \geq 9$, $\tilde{\alpha}=\alpha +5$, the constant $\theta_0\geq 1$ is large enough, and the time $T>0$ and the constant $\delta >0$ are small enough.  In particular, ($H_n$) implies
\[
\sum_{n=0}^{\infty}\left\{ \|\delta{U}_n\|_{H^m(\Omega_T)} +\|\delta \varphi_n\|_{H^m(\partial\Omega_T)}\right\} < \infty.
\]
Hence, the sequence $({U}_n,\varphi_n)$ converges in $H^{m}(\Omega_T)\times H^{m}(\partial\Omega_T)$ to some
limit $({U} ,\varphi )$. Recall that $m=\alpha -1 \geq 8$. Passing to the limit in (\ref{193}) and (\ref{194}) with
$s=m$, we obtain (\ref{69^})--(\ref{71^}). Consequently, ${U} :={U} +{U}^a$, $\varphi := \varphi +\varphi^a$ is a solution of problem \eqref{11.1}--\eqref{13.1}. As was already noticed above, this solution is unique. This completes the proof of Theorem \ref{t02}.

\section{Ill-posedness for simultaneous failure of non-collinearity and Rayleigh-Taylor sign conditions}
\label{s5}

Since below our goal will be the construction of particular exponential solutions, it is natural that in \eqref{frozen}, \eqref{frozen_bound} we drop the source terms $f$ and $g$, introduce non-zero initial data and consider the problem for all times $t>0$ and in the whole half-space $\mathbb{R}^3_+=\{ x_1>0,\ x'\in\mathbb{R}^2\}$ (without the periodicity conditions). We refer to \cite{Tcpaa} for further discussions on Rayleigh-Taylor instability detected as ill-posedness for frozen coefficients. As was noted in \cite{Tcpaa}, if the front symbol is not elliptic, the zero-order terms in $\varphi$ appearing in the boundary conditions (see \eqref{35} and \eqref{frozen_bound}) may play a crucial role for the well-posedness of a frozen coefficients linearized free boundary problem. This can be easily seen on such a classical example as the free boundary problem for the incompressible Euler equations with the vacuum boundary condition $p|_{\Gamma}=0$. For the reader's convenience, we below repeat corresponding short calculations from \cite{Tcpaa}.

\setcounter{subsubsection}{0}

\subsection{Incompressible Euler equations with a vacuum boundary condition}

Without going into details, by analogy with problem \eqref{frozen}, \eqref{frozen_bound} we just write down the frozen coefficients linearized problem associated with the free boundary problem for the incompressible Euler equations with the vacuum boundary condition $p|_{\Gamma}=0$:
\begin{equation}
\partial_tv +\nabla p=0,   \quad
{\rm div}\,v=0, \quad\mbox{in}\ \mathbb{R}_+\times \mathbb{R}^3_+,\label{froz2Db}
\end{equation}
\begin{equation}
 \partial_t\varphi=v_1 +\hat{a}_0
\varphi, \quad  p= \hat{a} \varphi,  \quad\mbox{on}\ \mathbb{R}_+\times \{x_1=0\}\times\mathbb{R}^2,
\label{froz2Dd}
\end{equation}
where $p$ is the scaled pressure (divided by the constant density $\hat{\rho}$) and the constants $\hat{a}_0$ and $\hat{a}$ are the same as in \eqref{frozen_bound}. For simplifying calculations we exclude $\varphi$ from the boundary conditions \eqref{froz2Dd}:
\begin{equation}\label{froz2Dd'}
\partial_tp-\hat{a}_0p-\hat{a}v_1=0  \quad\mbox{on}\ \mathbb{R}_+\times \{x_1=0\}\times\mathbb{R}^2.
\end{equation}

The Hadamard-type ill-posedness example is a sequence of solutions proportional to
\[
\exp \left\{ n\left(s t +\lambda x_1+ i(\omega ', x') \right)\right\}\quad \mbox{for}\  x_1>0,
\]
with  $n=1,2,3,\ldots$ and
\begin{equation}
\Re\,s>0,\quad \Re\,\lambda<0,\label{re_re}
\end{equation}
where $s$ and $\lambda$ are complex constants, $\omega '=(\omega_2,\omega_3)$ and $\omega_{2,3}$ are real constants.
Substituting the sequence
\begin{equation}
\begin{pmatrix} p_n \\ v_n \end{pmatrix} =\begin{pmatrix} \bar{p}_n \\ \bar{v}_n \end{pmatrix}\exp \left\{ n\left(s t +\lambda x_1+ i(\omega ', x') \right)\right\}
\label{exp_sol}
\end{equation}
into equations \eqref{froz2Db}, we get the dispersion relation
\[
s^2(\lambda^2-\omega^2)=0
\]
giving the root $\lambda = -\omega$ with $\Re\lambda <0$ (cf. \eqref{re_re}), where $(\bar{p}_n , \bar{v}_n)$ is a constant vector and $\omega =|\omega'|$. Without loss of generality we may assume that $\omega =1$, i.e. $\lambda =-1$. Then from the first equation in \eqref{froz2Db} we have
\begin{equation}
\bar{v}_{1n}=\frac{\bar{p}_{n}}{s},
\label{vbar}
\end{equation}
where $\bar{v}_{1n}$ is the first component of the vector $\bar{v}_n$.

In view of \eqref{vbar}, substituting \eqref{exp_sol} into the boundary condition \eqref{froz2Dd'}, we get for the constant $\bar{p}_{n}$ an equation which has a non-zero solution if
\begin{equation}
ns^2-\hat{a}_0s-\hat{a}=0.
\label{eq_s}
\end{equation}
For big enough $n$ and $\hat{a}\neq 0$, the last equation has the ``unstable" root
\begin{equation}
s=\frac{\hat{a}_0}{2n}+\sqrt{\frac{\hat{a}}{n}+\frac{\hat{a}_0^2}{4n^2}}=\frac{\sqrt{\hat{a}}}{\sqrt{n}}+\mathcal{O}\left(\frac{1}{n} \right)
\label{root_s}
\end{equation}
if and only if $\hat{a}>0$, i.e., if and only if \eqref{antiRT} holds (the Rayleigh-Taylor sign condition fails).

It is very important that root \eqref{root_s} for $\hat{a}>0$ after the substitution into the exponential sequence \eqref{exp_sol} gives an infinite growth in time of order $\sqrt{n}$ as $n\rightarrow \infty$ for any fixed (even very small) $t>0$. This is not only usual exponential instability but {\it ill-posedness} (violent instability). Note that for $\hat{a}=0$  equation \eqref{eq_s} has the ``unstable" root $s=\hat{a}_0/n$ for $\hat{a}_0>0$ but it just gives exponential instability but not ill-posedness.

\subsection{Free boundary problem for an incompressible inviscid elastic fluid}

Before proceeding to our frozen coefficients problem \eqref{frozen}, \eqref{frozen_bound} we consider the technically simpler case of incompressible inviscid elastic fluid \cite{HaoWang}. Keeping the same arguments and notations as above, we write down the ``incompressible'' counterpart of problem \eqref{frozen}, \eqref{frozen_bound}:
\begin{equation}\label{frozen'}
\left\{
\begin{array}{l}
\displaystyle
\partial_tv +\nabla p -\sum_{j=1}^3\mathcal{L}_jF_j=0, \\[12pt]
\partial_tF_j -\mathcal{L}_jv =0,\quad  {\rm div}\,v =0,  \quad\mbox{in}\ \mathbb{R}_+\times \mathbb{R}^3_+,
\end{array}
\right.
\end{equation}
\begin{equation}
  \partial_t\varphi=v_1 +\hat{a}_0
\varphi, \quad
 p=\hat{a} \varphi ,  \quad\mbox{on}\ \mathbb{R}_+\times \{x_1=0\}\times\mathbb{R}^2,\label{frozen_bound'}
\end{equation}
where the boundary conditions \eqref{frozen_bound'} can be again reduced to \eqref{froz2Dd'}. Moreover, we can totally exclude $F_j$ from system \eqref{frozen'}:
\begin{equation}\label{frozen"}
\biggl(\partial_t^2-\sum_{j=1}^3\mathcal{L}_j^2\biggr) v +\nabla\partial_tp=0,\quad {\rm div}\,v =0,  \quad\mbox{in}\ \mathbb{R}_+\times \mathbb{R}^3_+.
\end{equation}

We seek the sequence of exponential solutions \eqref{exp_sol} to problem \eqref{frozen"}, \eqref{froz2Dd'}. As for the incompressible Euler equations, we easily find $\lambda=-1$ either from the dispersion relation for system \eqref{frozen"} or just from the Laplace equation $\triangle p=0$ following from this system (we again assume without loss of generality that $\omega=1$). The result of substitution of \eqref{exp_sol} with $\lambda=-1$ into the first equation in \eqref{frozen"} implies
\begin{equation}\label{94}
\bar{v}_{1n}=\frac{s}{s^2+\hat{r}^2}\,\bar{p}_{n},
\end{equation}
where $\hat{r}^2=\hat{w}_1^2+\hat{w}_2^2+\hat{w}_3^2$ and $\hat{w}_j=\widehat{F}_{2j}\omega_2+\widehat{F}_{3j}\omega_3$ ($j=1,2,3$). Taking \eqref{94} into account and substituting \eqref{exp_sol} into the boundary condition \eqref{froz2Dd'}, we get for the constant $\bar{p}_{n}$ the equation
\[
(ns-\hat{a}_0)\bar{p}_{n}-\frac{\hat{a}s}{s^2+\hat{r}^2}\,\bar{p}_{n}=0
\]
which has a non-zero solution if
\begin{equation}
\label{95}
(ns-\hat{a}_0)(s^2+\hat{r}^2)-\hat{a}s=0.
\end{equation}

We will go ahead and say that, exactly as for problem \eqref{froz2Db}, \eqref{froz2Dd}, the coefficient $\hat{a}_0$ plays no role for the existence/nonexistence of roots $s$ giving ill-posedness. Therefore, we first assume that $\hat{a}_0=0$. Then, ignoring the neutral mode $s=0$, from \eqref{95} we obtain
\begin{equation}\label{96}
s^2=-\hat{r}^2+\frac{\hat{a}}{n}.
\end{equation}
Recall that $\omega =1$, i.e. $\omega'\neq  0$ (for the 1D case $\omega' =0  $ we get $\lambda =0$). Then $\hat{r}=0$ if and only if $\hat{w}_1=\hat{w}_2=\hat{w}_3=0$, i.e., when each vector $(\widehat{F}_{2j},\widehat{F}_{3j})$ for $j=1,2,3$ is either perpendicular to $\omega'$ or zero. In other words, $\hat{r}=0$ for some $\omega'$ if and only if the three vectors $(\widehat{F}_{2j},\widehat{F}_{3j})$ , $j=1,2,3$, are collinear, i.e.,  \eqref{collin-fr} holds.

If $\hat{r}\neq 0$, then the right-hand side in \eqref{96} is negative for large $n$, and we have no ``unstable'' roots (with  $\Re s>0$) for this case. If \eqref{collin-fr} holds, we choose $\omega'$ such that $\hat{w}_1=\hat{w}_2=\hat{w}_3=0$. Then, from \eqref{96} we find the root
\[
s=\frac{\sqrt{\hat{a}}}{\sqrt{n}}
\]
giving an ill-posedness example if and only if \eqref{antiRT} holds.

Let us now $\hat{a}_0\neq 0$. Expanding \eqref{96} in powers of $1/\sqrt{n}$, for $\hat{r}\neq 0$ we find the tree roots
\[
s=\mathcal{O}\left(\frac{1}{n}\right)\quad\mbox{and}\quad s=\pm i\hat{r}+\mathcal{O}\left(\frac{1}{n}\right)
\]
for which $\Re s$ is of order $1/n$. Even if $\Re s>0$, this does not give an infinite growth of the exponential solutions \eqref{exp_sol} for a fixed time $t> 0$ as $n\rightarrow \infty$. For collinear vectors $(\widehat{F}_{2j},\widehat{F}_{3j})$ , $j=1,2,3$, we again choose $\omega'$ such that $\hat{w}_1=\hat{w}_2=\hat{w}_3=0$, i.e., $\hat{r}=0$. Then, ignoring the neutral mode $s=0$, from \eqref{95} we obtain equation \eqref{eq_s} which has root \eqref{root_s} giving an ill-posedness example if and only if \eqref{antiRT} holds. That is, we have proved Theorem \ref{t3} for incompressible elastic fluids.

\subsection{Proof of Theorem \ref{t3}}

As for problem \eqref{frozen'}, \eqref{frozen_bound'}, we first exclude $F_j$ from system \eqref{frozen}:
\begin{equation}\label{97}
\frac{1}{\hat{\rho}\hat{c}^2}\,\partial_t p+{\rm div}\,v =0,\quad
\hat{\rho} \biggl(\partial_t^2-\sum_{j=1}^3\mathcal{L}_j^2\biggr) v +\nabla\partial_tp=0, \quad\mbox{in}\ \mathbb{R}_+\times \mathbb{R}^3_+.
\end{equation}
Moreover, we reduce the boundary conditions \eqref{frozen_bound} to \eqref{froz2Dd'}. In fact, we can easily obtain the dispersion relation from the following consequence of \eqref{97}:
\[
\biggl(\partial_t^2 -\hat{c}^2\triangle -\sum_{j=1}^3\mathcal{L}_j^2 \biggr)\partial_tp=0.
\]
Substituting \eqref{exp_sol} into the last equation and ignoring the neutral mode $s=0$, we get the dispersion relation
\[
\frac{s^2}{\hat{c}^2}+\frac{\hat{r}^2}{\hat{c}^2}-\lambda^2+1=0,
\]
giving the root
\begin{equation}
\lambda = -\sqrt{1+\left(\frac{\hat{r}}{\hat{c}}\right)^2+\left(\frac{s}{\hat{c}}\right)^2},
\label{98}
\end{equation}
where without loss of generality we again assumed that $\omega =1$, and $\hat{r}$ is the same as for incompressible fluids above.

By virtue of \eqref{98}, from the second equation in \eqref{97} we have
\begin{equation}
\bar{v}_{1n}=\frac{\bar{p}_{n}s}{\hat{\rho}(s^2+\hat{r}^2)}\,\sqrt{1+\left(\frac{\hat{r}}{\hat{c}}\right)^2+\left(\frac{s}{\hat{c}}\right)^2}.
\label{99}
\end{equation}
Substituting \eqref{exp_sol} into the boundary condition \eqref{froz2Dd'} and using \eqref{99}, we get for the constant $\bar{p}_{n}$ the equation
\[
(ns-\hat{a}_0)\bar{p}_{n}-\frac{\hat{a}s}{\hat{\rho}(s^2+\hat{r}^2)}\,\sqrt{1+\left(\frac{\hat{r}}{\hat{c}}\right)^2+\left(\frac{s}{\hat{c}}\right)^2}\,\bar{p}_{n}=0
\]
which has a non-zero solution if
\begin{equation}
\label{100}
\hat{\rho}(ns-\hat{a}_0)(s^2+\hat{r}^2)-\hat{a}s\sqrt{1+\left(\frac{\hat{r}}{\hat{c}}\right)^2+\left(\frac{s}{\hat{c}}\right)^2}=0.
\end{equation}

Expanding \eqref{100} in powers of $1/\sqrt{n}$, for $\hat{r}\neq 0$ we find the tree roots
\[
s=\mathcal{O}\left(\frac{1}{n}\right)\quad\mbox{and}\quad s=\pm i\hat{r}+\mathcal{O}\left(\frac{1}{n}\right)
\]
for which $\Re s$ is of order $1/n$. As was noted above, this cannot give ill-posedness. For collinear vectors $(\widehat{F}_{2j},\widehat{F}_{3j})$ , $j=1,2,3$, we choose $\omega'$ such that $\hat{r}=0$. Then, ignoring the neutral mode $s=0$, from \eqref{100} we obtain the equation
\begin{equation}
\label{101}
\hat{\rho}s(ns-\hat{a}_0)-\hat{a}\sqrt{1+\left(\frac{s}{\hat{c}}\right)^2}=0.
\end{equation}
Assuming that $\hat{a}\neq 0$ and expanding \eqref{101} in powers of $1/\sqrt{n}$, we find the root
\[
s=\sqrt{\frac{\hat{a}}{\hat{\rho}}}\,\frac{1}{\sqrt{n}}+\mathcal{O}\left(\frac{1}{n} \right)
\]
which gives ill-posedness if and only if $\hat{a}<0$, i.e., \eqref{antiRT} holds. The proof of Theorem \ref{t3} is thus complete.

\section{Open problems}
\label{s6}

It is clear that the results of Theorems \ref{t01} and \ref{t02} for the reduced nonlinear problem in a fixed domain can be reformulated for the original free boundary problem. We have thus proved the local-in-time existence of a unique smooth solution of the free boundary problem \eqref{7}, \eqref{4} under the hyperbolicity conditions \eqref{11} and suitable compatibility conditions for the initial data, provided that either the Rayleigh-Taylor sign condition \eqref{5} or the non-collinearity condition \eqref{6} holds. At the same time, it seems that the most general ``stability'' assumption for the initial data should only require the fulfilment of  the Rayleigh-Taylor sign condition at all those points of $\Gamma (0)$ where the non-collinearity condition fails. But even for the linearized problem the corresponding well-posedness result under the assumption that the inequality in \eqref{42} holds at all those points of $\Omega_T$  where the inequality in \eqref{40} fails is yet an open problem.

Moreover, regarding the non-collinearity condition \eqref{42} for the linearized problem, formally the requirement that the front symbol is elliptic, i.e., the boundary conditions \eqref{b51b} supplemented with relations \eqref{59} for $j=1,2,3$ are  resolvable for $\nabla_{t,x'}\varphi$, is equivalent to the assumption
\[
{\rm rank}
\begin{pmatrix}
0 & 0 & 0\\
\widehat{F}_{21} & \widehat{F}_{22}& \widehat{F}_{23}\\
\widehat{F}_{31} & \widehat{F}_{32}& \widehat{F}_{33}
\end{pmatrix}=2\quad\mbox{on}\ \partial\Omega_T.
\]
It is still unclear how to deduce an a priori estimate under this most general ``non-collinearity'' assumption which is less restrictive than \eqref{42}. This is also an open problem for future research.

\end{document}